\documentclass[11pt]{article}

\usepackage[]{graphicx}
\usepackage[]{color}
\usepackage{alltt}
\usepackage{array}
\usepackage{amsmath}

\usepackage{amssymb}
\usepackage{amsfonts}
\usepackage{bbm}
\usepackage{todonotes}
\usepackage{mathabx}
\usepackage{mathtools}
\usepackage{enumitem}
\usepackage{booktabs}
\usepackage{makecell}
\let\underbrace\LaTeXunderbrace

\usepackage{amsthm,thmtools}

\DeclareMathOperator*{\argmax}{arg\,max}
\DeclareMathOperator*{\argmin}{arg\,min}

\theoremstyle{definition}
\newtheorem{exmp}{Example}[section]

\newcommand{\myparagraph}[1]{\paragraph{#1}}

\newcommand{\suppl}{1}

\newif\ifmoditem
\newcommand{\setupmodenumerate}{\global\moditemfalse
  \let\origmakelabel\makelabel
  \def\moditem##1{\global\moditemtrue\def\mesymbol{##1}\item}\def\makelabel##1{\origmakelabel{##1\ifmoditem\rlap{\mesymbol}\fi\enspace}\global\moditemfalse}}

\let\originalparagraph\paragraph
\renewcommand{\paragraph}[2][]{\originalparagraph{#2#1}}

\usepackage[a4paper, width = 160mm, top = 35mm, bottom = 30mm, 
bindingoffset = 0mm]{geometry}
\usepackage[utf8]{inputenc}
\usepackage{ragged2e}
\usepackage{xcolor}
\usepackage[round, comma]{natbib}
\usepackage{fancyhdr}

\usepackage{hyperref}
\hypersetup{
  colorlinks = true,
  linkcolor ={blue},
  urlcolor = {blue},
  citecolor ={blue}}
\usepackage[capitalize,nameinlink]{cleveref}
\crefname{exmp}{Example}{Examples}
\Crefname{exmp1}{Example}{Examples}
\usepackage{crossreftools}
\usepackage{xcolor}

\definecolor{LMUblue}{RGB}{0, 17, 88}
\definecolor{LMUlightblue}{RGB}{92,177,235}
\definecolor{LMUgreen}{RGB}{0,136,58}
\definecolor{LMUlightgreen}{RGB}{170, 173, 0}
\definecolor{LMUred}{RGB}{190,25,8}
\definecolor{LMUpurple}{RGB}{176, 32, 121}
\definecolor{LMUorange}{RGB}{241, 135, 0}
\definecolor{bggray}{RGB}{245, 245, 245}
\definecolor{LMUgray}{RGB}{0.2,0.2,0.2}

 \usepackage[tracking=true,expansion=true,stretch=15,shrink=15]{microtype}
\DeclareMicrotypeSet*[tracking]{my}
{ font = */*/*/sc/* }
\SetTracking{ encoding = *, shape = sc }{ 45 }
\SetProtrusion{encoding={*},family={bch},series={*},size={6,7}}
{1={ ,750},2={ ,500},3={ ,500},4={ ,500},5={ ,500},
  6={ ,500},7={ ,600},8={ ,500},9={ ,500},0={ ,500}}
\SetExtraKerning[unit=space]
{encoding={*}, family={qhv}, series={b}, size={large,Large}}
{1={-200,-200},
  \textendash={400,400}} \usepackage{shadethm}
\usepackage{etoolbox}

\newtheoremstyle{new}
  {12pt}      {12pt}      {\itshape}  {}          {\bfseries\color{black}} {.}         { }         {}          \theoremstyle{new}
\newtheorem{theorem}{Theorem}
\newtheorem{corollary}{Corollary}
\newtheorem{proposition}{Proposition}
\newtheorem{lemma}{Lemma} 

\newtheorem{example}{Example}

\AfterEndEnvironment{theorem}{\noindent\ignorespaces}
\AfterEndEnvironment{corollary}{\noindent\ignorespaces}
\AfterEndEnvironment{proposition}{\noindent\ignorespaces}
\AfterEndEnvironment{lemma}{\noindent\ignorespaces}
\AfterEndEnvironment{definition}{\noindent\ignorespaces}
\AfterEndEnvironment{assumption}{\noindent\ignorespaces}
\AfterEndEnvironment{example}{\noindent\ignorespaces}
\AfterEndEnvironment{remark}{\noindent\ignorespaces}
\AfterEndEnvironment{claim}{\noindent\ignorespaces}

\Crefname{theorem}{Theorem}{Theorems}
\Crefname{corollary}{Corollary}{Corollaries}
\Crefname{proposition}{Proposition}{Propositions}
\Crefname{lemma}{Lemma}{Lemmas}
\Crefname{definition}{Definition}{Definitions}
\Crefname{example}{Example}{Examples}
\Crefname{remark}{Remark}{Remarks}
\Crefname{claim}{Claim}{Claims}

\usepackage[breakable, theorems, skins]{tcolorbox}
\tcbset{enhanced}

\definecolor{shadethmcolor}{cmyk}{0,0,0,0.075}    
\definecolor{shaderulecolor}{rgb}{1,1,1}

\newtheoremstyle{shad}
  {12pt}      {12pt}      {\itshape }  {}          {\bfseries\color{black}} {.}         { }         {}          \theoremstyle{shad} 

\newshadetheorem{thmbox}[theorem]{Theorem}
\newshadetheorem{corbox}[theorem]{Korollar}
\newshadetheorem{propbox}[theorem]{Proposition}
\newshadetheorem{lembox}[theorem]{Lemma}
\newshadetheorem{exbox}[theorem]{Beispiel}
\newshadetheorem{defbox}[theorem]{Definition}
\newshadetheorem{rembox}[theorem]{Anmerkung}

\newtheoremstyle{shad*}
  {12pt}      {12pt}      {\itshape }  {}          {\bfseries\color{black}} {.}         { }         {}          \theoremstyle{shad*} 
\newshadetheorem{anobox}{}

\Crefname{theorem}{Theorem}{Theoreme}
\Crefname{corbox}{Korollar}{Korollare}
\Crefname{propbox}{Proposition}{Propositionen}
\Crefname{lembox}{Lemma}{Lemmas}
\Crefname{exbox}{Beispiel}{Beispiele}
\Crefname{defbox}{Definition}{Definitionen}
\Crefname{rembox}{Anmerkung}{Anmerkungen}
 
\usepackage{bm}
\usepackage{dsfont}

\newcommand{\ind}{\mathds{1}}
\newcommand{\R}{\mathbb{R}}
\newcommand{\N}{\mathbb{N}}

\providecommand{\P}{}
\renewcommand{\P}{\mathbb{P}}
\newcommand{\Z}{\mathbb{Z}}
\newcommand{\E}{\mathbb{E}}

\newcommand{\bnull}{\bm{0}}
\newcommand{\ba}{\bm{a}}
\newcommand{\bb}{\bm{b}}

\newcommand{\be}{\bm{e}}

\newcommand{\bk}{\bm{k}}

\newcommand{\bu}{\bm{u}}
\newcommand{\bv}{\bm{v}}
\newcommand{\bw}{\bm{w}}
\newcommand{\bx}{\bm{x}}
\newcommand{\by}{\bm{y}}
\newcommand{\bz}{\bm{z}}

\newcommand{\bC}{\bm{C}}

\newcommand{\bW}{\bm{W}}
\newcommand{\bX}{\bm{X}}
\newcommand{\bY}{\bm{Y}}
\newcommand{\bZ}{\bm{Z}}

\newcommand{\Acal}{\mathcal{A}}
\newcommand{\Bcal}{\mathcal{B}}

\newcommand{\Fcal}{\mathcal{F}}

\newcommand{\Ncal}{\mathcal{N}}

\newcommand{\Xcal}{\mathcal{X}}

\newcommand{\eps}{\varepsilon}

\newcommand{\bbeta}{\bm{\beta}}

\newcommand{\bDelta}{\bm{\Delta}}

\newcommand{\btheta}{\bm{\theta}}

\newcommand{\blambda}{\bm{\lambda}}
\newcommand{\bmu}{\bm{\mu}}

\newcommand{\0}{\bm{0}}

\newcommand{\tr}{\operatorname{tr}}

\newcommand{\diag}{\operatorname{diag}}
\renewcommand{\vec}{\operatorname{vec}}

\newcommand{\sign}{\operatorname{sign}}

\renewcommand{\bar}{\overline}

\newcommand{\hbtheta}{\bm{\hat{\theta}}}

\newcommand{\htheta}{\hat{\theta}}

\newcommand{\hbu}{\bm{\hat{u}}}

\newcommand{\btu}{\bm{\tilde{u}}}

\newcommand{\bttheta}{\bm{\tilde{\theta}}}

\newcommand{\tM}{\tilde{M}}
\newcommand{\tD}{\tilde{D}}

\newcommand{\lf}{\left}
\newcommand{\ri}{\right}
 \providecommand{\Pr}{}
\renewcommand{\Pr}{\mathbb{P}}
\newcommand{\var}{{\mathds{V}\mathrm{ar}}}

\newcommand{\cov}{{\mathds{C}\mathrm{ov}}}

\newcommand{\wt}[1]{\widetilde{#1}}

\newcommand{\sumin}{\sum_{i = 1}^n}

\newcommand{\frzwei}{\frac{1}{2}}

\makeatletter
\def\blfootnote{\gdef\@thefnmark{}\@footnotetext}
\makeatother

\newcommand{\expl}[1]{\tag*{(#1)}}

\IfFileExists{upquote.sty}{\usepackage{upquote}}{}

\title{Asymptotics for estimating a diverging number of parameters --- with and without sparsity}
\author{Jana Gauss and Thomas Nagler \\[11pt] \emph{Department of Statistics, LMU Munich} \\ \emph{Munich Center for Machine Learning (MCML)}}

\begin{document}

\def\IncludeProofs{}

\pagenumbering{arabic}

\maketitle
\begin{abstract}We develop a general asymptotic theory for estimating equations whose dimension diverges with the sample size. For both unpenalized and sparse penalized problems, we establish population-level conditions for existence, consistency, uniqueness, and asymptotic normality; in the penalized setting, we also establish selection consistency under a generalized version of the mutual incoherence condition. Our results cover stepwise procedures with a diverging number of steps, independent but non-identically distributed observations, and dependent data. We allow penalties that are simultaneously nonconvex, non-coordinate-separable, and group-structured and may involve heterogeneous tuning parameters. Our population-level conditions imply a weak form of restricted strong convexity, and we provide an explicit example where the commonly used stronger form fails. The results are illustrated by several applications including Group SCAD-penalized estimation in generalized linear models, distributed inference under network dependence, and penalized stepwise estimation in causal inference.
\end{abstract} 

\section{Introduction}\label{sec:Intro}

In modern applications, statisticians are facing increasingly complex and high-dimensional problems. Many data sets have a huge number of variables, calling for similarly many parameters $p$. In other scenarios, the number of variables is moderate, but adequately modeling the data requires highly complex, nonlinear models with many parameters.
The traditional fixed-$p$-large-$n$ paradigm is inadequate in such situations.

This article adopts an asymptotic perspective, allowing both the sample size $n$ and the number of parameters $p_n$ to diverge. We consider general parametric problems where the estimator $\hbtheta$ solves an estimating equation
\begin{align} \label{eq:est_eq}
  \frac 1 n \sumin \phi(\bX_{i}; \hbtheta) = \0 \in \R^{p_n},
\end{align}
with some function $\phi\colon \Xcal \times \R^{p_n} \to \R^{p_n}$ and independent, but not necessarily identically distributed $\bX_i$.
A classical example is risk minimization, where $\phi$ is the gradient of a loss function. However, the estimating equation framework is also suited for more complex methods, such as stepwise and multi-sample estimation procedures, where an optimization-based formulation is less convenient.
In fact, the lack of general asymptotic theory for stepwise procedures when $p_n \to \infty$ is what originally motivated this work.

The main question we address is: under what conditions on the data-generating process, the function $\phi$, and the growth of $p_n$ does an estimator solving \eqref{eq:est_eq} exist, and when is it consistent, unique, and asymptotically normal?
We also consider sparse problems where $p_n$ may exceed $n$. Penalized inference within the general estimating equation framework can be formulated as
\begin{align} \label{eq:est_eq_pen}
  \frac{1}{n} \sumin \phi(\bX_{i}; \hbtheta) \in \partial p_{\blambda}(\hbtheta),
\end{align}
where $\partial p_{\blambda}$ is the generalized gradient of a penalty function $p_{\blambda}$, and $\blambda$ is a vector of tuning parameters.

\myparagraph{Contributions}
The main contribution of this work is a unified asymptotic theory for estimating equations with diverging dimension.
Rather than developing separate arguments for particular models or loss functions, we formulate verifiable population-level conditions on the estimating equation and its curvature.
The gaps in existing theory are most apparent when the estimating equations are not gradients of a single loss, the number of estimation steps diverges, observations come from heterogeneous or dependent populations, or penalties combine nonconvexity with group structure.
The resulting framework unifies several existing results and substantially extends their applicability in three directions.

\emph{First}, for unpenalized estimating equations, we provide general conditions for existence and consistency (\cref{theorem1}), uniqueness (\cref{theorem1-uniqueness}), and asymptotic normality (\cref{theorem2}).
When specialized, these results closely align with the sharpest known conditions for linear and generalized linear models and substantially improve upon those for general maximum likelihood problems.
At the same time, the theory applies directly to estimating procedures that do not arise from risk minimization.

\emph{Second}, we develop consistency (\cref{theorem3}), existence and selection consistency (\cref{theorem4}), uniqueness (\cref{theorem5}), and asymptotic normality (\cref{theorem6}) for penalized estimating equations.
The penalty does not need to be convex or coordinate-separable and may be group-structured or involve multiple tuning parameters with varying strength.
This combination is not covered by the current general theory for sparse $M$-estimation: results allowing nonseparable or group-structured penalties generally require convexity, whereas results allowing nonconvex penalties typically assume coordinate separability.
The solution constructed in \cref{theorem4} relies on a generalization of the mutual incoherence or irrepresentable condition, and suitable nonconvex penalties attain the oracle property.
Our population-level conditions imply a weak version of restricted strong convexity (RSC).
We show that the stronger form commonly used in the literature fails in a simple quadratic problem with non-rank-one sample Hessians, while our weaker condition remains easy to verify (\cref{sec:RSC-failure}).

\emph{Third}, the framework accommodates independent but non-identically distributed observations, and we extend the results to dependent data (\cref{th:dependence}).
This is particularly useful for stepwise and multi-sample procedures, whose estimating equations combine information from different stages or populations.
In the stepwise setting, the number of estimation steps is allowed to diverge with the sample size.

Our illustrative applications cover established results for $M$-estimation and also give results that are not covered by existing work, including stepwise estimation with a diverging number of steps, distributed inference under network dependence,  and the oracle property of Group SCAD-penalized estimation in generalized linear models.

\myparagraph{Relation to existing work}
The study of unpenalized problems dates back at least to \citet{Huber73}, who focused on $M$-estimators in linear models. Following his seminal work, \citet{Yohai, PortnoyCons, PortnoyAsymp, Welsh, Mammen89} established consistency and asymptotic normality under various conditions; see \citet{Li_mEst} for a comprehensive overview.
The sharpest known conditions are $p_n \ln p_n / n \to 0$ for consistency \citep{PortnoyCons}, and $(p_n \ln n)^{3/2}/n \to 0$ \citep{PortnoyAsymp} or $p_n^{3/2} \ln n / n \to 0$ \citep{Mammen89} for asymptotic normality.
\citet{Fan} extend these results from generalized linear models to general maximum likelihood problems under more restrictive conditions, requiring $p_n^4 / n \to 0$ for consistency and $p_n^5 / n \to 0$ for asymptotic normality.
\citet{He2000} derived asymptotics for $M$-estimators with convex loss under a generic stochastic equicontinuity condition, though this approach provides limited insight into the settings where it applies.
Other work studies the exact distribution or predictive risk in linear models when $p_n/n$ has a positive limit \citep{ElKaroui2013, Dobriban2018, Hastie2022}.

When $p_n$ exceeds $n$, consistent estimation is still possible if the true parameter $\btheta^*$ is sparse.
Various sparsity-inducing penalties have been proposed, including the Lasso \citep{TibLasso}, group-structured penalties such as the Group Lasso \citep{GroupL}, and bias-reducing, nonconvex penalties such as $\ell_q$-penalties \citep{Knight}, (Group) SCAD \citep{Fan1997, Fan01, GroupSCAD, GroupSCAD2}, and MCP \citep{Zhang10}.
State-of-the-art results address general $M$-estimation problems and broad classes of penalties \citep{WWpaper, Lee, LohWW, LohWW2, Loh}; see also the monograph by \citet{WWbook}.
These results typically rely on an RSC condition on the realized loss and focus on deterministic bounds.
Frameworks that accommodate nonseparable or group-structured penalties generally assume convexity \citep{WWpaper, Lee}, while results for nonconvex penalties typically focus on coordinate-separable penalties \citep{LohWW, LohWW2, Loh}.
They therefore do not cover combinations such as Group SCAD, nor do their optimization-based formulations directly encompass general stepwise estimating equations.
Our asymptotic perspective is complementary: regularity conditions are formulated at the population level, making their applicability transparent, while distributional limits provide a tight characterization of estimation uncertainty.

\myparagraph{Outline}
The remainder of this article is structured as follows. \cref{sec2} presents the main results for unpenalized estimation. \cref{sec:notation_penalty,sec:penalty_examples} introduce penalized estimating equations and discuss the conditions on the penalty.
\cref{sec:ass_penalty,sec:normality_pen} contain the main results for penalized estimation.
In \cref{sec:dependence}, we extend the theory to dependent data.
In \cref{sec:Exmp}, we illustrate its applicability through $M$-estimation, multi-sample inference, and stepwise procedures in causal inference and stochastic optimization.
All proofs are provided in the supplement.

\section{Unpenalized Estimation}
\label{sec2}
\subsection{Setup and Notation}

Suppose we observe independent random variables $\bX_1, \dots, \bX_n \in \Xcal$
and want to estimate a parameter $\btheta=(\theta_{1}, \ldots, \theta_{p_n})^\top \in \R^{p_n}$ with $p_n \to \infty$ as $n \to \infty$. 
We do not assume identical distributions to also cover multi-sample estimation problems.
The target value $\btheta^*$ is the solution to the system of equations
$\sum_{i = 1}^n\E[\phi_{i}(\btheta^*)]= \0,$ 
with continuous $\phi_{i}(\btheta) 
= \phi(\bX_{i}; \btheta) \in \mathbb{R}^{p_n}$; for example, the gradient of a log-likelihood or  loss function.
The $k$-th entry of $\phi_{i}(\btheta) 
$ is denoted by $\phi_{i}(\btheta)_k, k = 1, \ldots, p_n$.
The estimator $\hbtheta$ is the solution of
\begin{equation}
\label{eq:estim1}
 \Phi_n(\hbtheta) := \frac 1 n \sumin \phi_{i}(\hbtheta)  = \0.
\end{equation}
Define the $p_n \times p_n$ matrices
\begin{align*}
   I(\btheta) = \frac{1}{n} \sum_{i = 1}^n  \mathrm{Cov}[\phi_{i}(\btheta)], \quad J(\btheta) = \frac{1}{n} \sum_{i = 1}^n  \nabla_{\btheta} \E[ \phi_{i}(\btheta)] , \quad 
   \bar J (\btheta^*, \btheta)  = \int_0^1 J(\btheta^* + t(\btheta - \btheta^*)) d t.
\end{align*}
The parameter $\btheta^*$, functions $\phi(\btheta), J(\btheta),  \bar J (\btheta^*, \btheta), I(\btheta)$, and the support and distribution of $\bX_i$ all depend on $n$, but we suppress this in the notation to avoid clutter.
We assume that $J(\btheta)$ is continuous in $\btheta$.

Throughout the paper, $\| \cdot \|$ denotes the Euclidean norm for vectors and the spectral norm  $\| A \| =\sup_{\| \bx \| = 1} \| A\bx\|$ for matrices.
$\lambda_{\max}(A)$ denotes the largest eigenvalue of a matrix $A$ and $\| A \|_\infty$ denotes the largest absolute row sum.
Define $r_n = \sqrt{\mathrm{tr}(I(\btheta^*)) / n}$ and let $\Theta_{n} \subset \R^{p_n}$  be a sequence of sets such that for all $C < \infty$, there exists an $n_C$ such that for all $n \ge n_C$, it holds that $\Theta_n \supset \{ \btheta : \| \btheta - \btheta^* \| \le r_n C \}$.
The matrix $I(\btheta^*)$ measures the variability of the estimating functions, while $J(\btheta^*)$ describes the local sensitivity of the population equation; under the identification conditions below, $r_n$ is the corresponding natural estimation rate.

\subsection{Consistency and Uniqueness}
\label{sec:cons}

Our main assumption for consistency of the estimator $\hbtheta$ is the following:

\begin{enumerate}[label=(A\arabic*), series=assump]
  \item   \label{A:Cons1}  There exists a sequence of symmetric, matrix-valued functions $H_n(\bx)$ such that:
        \begin{enumerate}[label = (\roman*)]
          \item \label{eq:Hn-def} For all  $\bu$ such that $\btheta^* + \bu \in \Theta_n$ and $\bx \in \Xcal$, it holds that \\
                $\bu^\top [\phi(\bx; \btheta^* + \bu) - \phi(\bx; \btheta^* )] \le \bu^\top H_n(\bx) \bu;$
          \item \label{eq:Hn-identifiable}  $\limsup_{n \to \infty} \lambda_{\max}(n^{-1} \sumin\E[H_n(\bX_{i})]) \le -c < 0;$
          \item  \label{eq:Hn-bounds} For some sequence $B_n = o(n / \ln p_n)$, it holds that
                \begin{align*}
                  \sup_{\| \bu \| = 1} \frac{1}{n}  \sumin \E\left[ (\bu^\top H_n(\bX_i) \bu)^2 \right] & = o(n), \\
                     \frac{1}{n} \sum_{i = 1}^n \lf\| \E[H_n(\bX_{i})^2 \ind_{\|H_n(\bX_{i})\| \le B_n}] \ri\| &= o\left(\frac{n}{ \ln p_n}\right) , \\  
                      \sum_{i = 1}^n  \Pr(\| H_n(\bX_{i}) \| > B_n)  &= o(1).
                \end{align*}
        \end{enumerate}
\end{enumerate} Before discussing the assumptions in detail, we state our main results.
\begin{theorem}
  \label{theorem1}
  Under assumption \ref{A:Cons1}, the following holds with probability tending to 1:
  \begin{enumerate}[label=(\roman*)]
    \item The sets $\Theta_n$ contain at least one solution of the estimating equation (\ref{eq:estim1}).
    \item Every solution $\hbtheta \in \Theta_n$ satisfies $\lVert \hbtheta - \btheta^* \rVert = O_p(r_n)$.
  \end{enumerate}
\end{theorem}

\begin{theorem}  \label{theorem1-uniqueness}
  Suppose that \ref{A:Cons1} holds, and for any $\btheta, \btheta + \bu \in  \{\btheta: \|\btheta - \btheta^*\| \le a_n r_n \}$ with $a_n \to \infty$ arbitrarily slowly and all $\bx \in \Xcal$, it holds that
  \begin{align} \label{eq:Hn-def-unique}
    \bu^\top [\phi(\bx; \btheta + \bu) - \phi(\bx; \btheta )] \le \bu^\top H_n(\bx) \bu.
  \end{align}
  Then the solution $\hbtheta$ is unique on $\Theta_n$ with probability tending to 1.
\end{theorem}

The natural convergence rate $r_n = \sqrt{\tr(I(\btheta^*)) / n}$ is typically of order $\sqrt{p_n/n}$, for example when $\max_{i, k} \mathbb{E}[\phi_i(\btheta^*)_k^2] = O(1)$.
The trace formulation also accommodates diverging entries of $\phi$ and slower rates, as in the multi-sample problems of \cref{sec:MSample}.

To interpret condition \ref{A:Cons1}, consider the core idea of the proof.
A sufficient condition for the estimating equation \eqref{eq:estim1} to have a solution in a ball $\{\|\btheta - \btheta^*\| \le r_n C\}$ is  
$\sup_{\|\bu\| = r_n C} \bu^\top \Phi_n(\btheta^* + \bu)  \le 0$ 
  \citep[Theorem 2.3]{IVT}.
Condition \ref{eq:Hn-def} upper bounds the left-hand side via
\begin{align*}
  \bu^\top \Phi_n(\btheta^* + \bu) \le \bu^\top \Phi_n(\btheta^*) + \bu^\top \left( \frac 1 n \sumin H_n(\bX_i) \right) \bu.
\end{align*}
Standard arguments give $\|\Phi_n(\btheta^*)\| = O_p(r_n)$, while \ref{eq:Hn-bounds} and matrix concentration allow replacing the sample average by its negative definite expectation; see \ref{eq:Hn-identifiable}.
Thus, for $\|\bu\| = r_n C$ and sufficiently large $C$, the quadratic term dominates, yielding a solution in the ball with high probability.
Variations of this argument give consistency of other solutions and uniqueness.

The majorization approach controls the curvature around $\btheta^*$ uniformly without covering arguments, allowing non-restrictive growth conditions on $p_n$; constructions of $H_n$ are given in \cref{sec:Exmp}.
The function $H_n$ may be viewed as an envelope for the symmetrized Jacobians
\[
\left\{ H_{\btheta}(\bx) = \frac 1 2 \; \nabla_{\btheta} \phi(\bx; \btheta) + \frac 1 2 \; \nabla_{\btheta} \phi(\bx;\btheta)^\top \colon\btheta \in \Theta_n\right\}
\]
in the positive semi-definite ordering $\preceq$ of matrices,
but condition \ref{eq:Hn-def} is slightly weaker.
Condition \ref{eq:Hn-identifiable} only imposes concavity on average, not for every $\bx$.
Lemma 4 in the supplement converts rate conditions such as \ref{eq:Hn-bounds} into more interpretable moment conditions; for example, its second part holds if all $H_n(\bx)$ are negative semi-definite and $n^{-1} \sumin \E[H_n(\bX_{i})] = O(1)$.
Choosing $\Theta_n = \{\btheta: \|\btheta - \btheta^*\| \le  r_n a_n \}$ with $a_n \to \infty$ arbitrarily slowly gives the weakest conditions, as only a shrinking neighborhood of $\btheta^*$ matters; larger sets may require larger envelopes in \ref{A:Cons1}\ref{eq:Hn-def} and \eqref{eq:Hn-def-unique}.

The assumptions implicitly restrict the growth of $p_n$.
For M-estimation in GLMs, where $\phi_{i}(\btheta) = \psi( Y_{i}, \bX_{i}^\top \btheta) \bX_{i}$, they yield $p_n \ln p_n / n \to 0$ under standard assumptions, matching \citet{PortnoyCons}; see \cref{sec:MEstGen}. By contrast, \citet{Fan} required $p_n^4 / n \to 0$ for general nonlinear maximum likelihood problems.
Condition \ref{A:Cons1}\ref{eq:Hn-bounds} need not even require $p_n/n\to0$: for unrelated componentwise equations, $H_n$ can be diagonal, and uniformly bounded entries would allow $B_n=O(1)$ and $p_n\le e^{o(n)}$.

\subsection{Asymptotic Normality}

We now turn to the asymptotic normality of the estimator. Because the dimension of the parameter space grows with the sample size, we state this in terms of convergence of finite-dimensional projections.
Let $A_n \in \R^{q \times p_n}$ be some matrix.

\begin{enumerate}[label=(A\arabic*), resume=assump]
  \item \label{A:Asymp} For every $C \in (0, \infty)$ and some sequence $D_n = o(\sqrt{n} / (r_n \,p_n))$, with $\tilde \phi_i(\btheta) = \phi_i(\btheta)$ or $\tilde \phi_i(\btheta) = \phi_i(\btheta) - \E[\phi_i(\btheta)]$, it holds that
        \begin{align*}
          \sup_{\| \bu \|, \| \bu' \| \le  r_n C} \frac 1 n \sumin \frac{\E\lf[ \| A_n [\tilde \phi_{i}(\btheta^* +  \bu) - \tilde \phi_{i}(\btheta^* +  \bu')]\|^2\ri]}{\| \bu - \bu'\|^2}                        & = o\lf( \frac{1}{r_n^2 \,p_n} \ri),     \\
         \sumin    \Pr\lf( \sup_{\| \bu \|, \| \bu' \| \le r_n C} \frac{ \|A_n[\tilde \phi_{i}(\btheta^* +  \bu) - \tilde \phi_{i}(\btheta^* +  \bu')] \|}{\| \bu - \bu'\|}  >  D_n  \ri) & = o(1),                                 \\
          \sup_{\| \bu\|  \le r_n C} \| A_n[\bar J (\btheta^*, \btheta^* + \bu) - J(\btheta^*)] \|    & = o\lf(\frac{1}{\sqrt{n} \, r_n}\ri),
        \end{align*}
        where the same choice of $\tilde \phi_i(\btheta)$ must be used in both the first and the second condition.
      \item  \label{A:Cons2} It holds that $\textstyle \max_{1 \le i \le n}\mathbb{E}\left[ \| A_n \phi_{i}(\btheta^*)\|^4 \right] =o(n).$
\end{enumerate} Assumption \ref{A:Asymp} is a stochastic smoothness condition required to control fluctuations of the estimating equation.
Allowing for either $\phi_i(\btheta)$ or $\phi_i(\btheta) - \E[\phi_i(\btheta)]$ leads to sharp conditions in both standard and multi-sample settings.
The moment condition \ref{A:Cons2} typically requires $p_n^2/n \to 0$, e.g., if $\| A_n \| = O(1)$ and $\max_{i, k} \E[\phi_{i}(\btheta^*)_k^4] = O(1)$.

\begin{theorem}
  \label{theorem2}
  If conditions \ref{A:Cons1}--\ref{A:Cons2} hold for some matrix $A_n \in \R^{q \times p_n}$ with fixed $q$ for which $\Sigma = \lim_{n \to \infty} A_n I(\btheta^*)A_n^\top$ exists, it holds that $$\sqrt{n} A_n  J(\btheta^*)(\hbtheta - \btheta^*) \to_d \Ncal(\0, \Sigma).$$
\end{theorem}
In the finite-dimensional setting ($p_n = p < \infty$), choosing $A_n = J(\btheta^*)^{-1}$ yields the familiar statement
\[
\sqrt{n} (\hbtheta - \btheta^*) \to_d \Ncal(\bnull, J(\btheta^*)^{-1} I(\btheta^*) J(\btheta^*)^{-\top} ).
\]
The proof of \cref{theorem2} uses mixed-entropy inequalities \citep[Section 2.14]{vdV} building on Talagrand's work on generic chaining \citep{talagrand2005generic} to obtain weak conditions. 
For $M$-estimators in the linear regression model, we obtain asymptotic normality if $p_n^2 \ln n/n \to 0$ under standard assumptions, see \cref{cor:Mest} in \cref{sec:MEstGen}.
\citet{PortnoyAsymp} and \citet{Mammen89} obtain the slightly weaker condition $p_n^{3/2} \ln n / n \to 0$.
Both authors exploit the linear model structure through fourth-order expansions and a design-dependent standardization, effectively alleviating the need for the sample Hessian to converge. \citet{portnoy1986central} showed that a general CLT cannot hold unless $p_n^2/n \to 0$, so our result appears to be close to optimal.
\citet{Fan} require the much stronger condition $p_n^5 / n \to 0$ for maximum likelihood estimation in general nonlinear models.

\section{Penalized Estimation}
\label{sec3}

\subsection{Penalization of Estimating Equations}
\label{sec:notation_penalty}

We now allow $p_n$ to exceed $n$ while assuming that the target parameter $\btheta^*$ is sparse.
To motivate our formulation for general estimating equations, consider a penalized risk minimization problem.
Let $L$ be some differentiable loss function and $p_{\lambda_n}(\btheta )$ be a function penalizing the magnitude of $\btheta$, e.g., the $\ell_1$-penalty $p_{\lambda_n}(\btheta ) = \lambda_n \| \btheta \|_1$.
The estimator is given by
\begin{equation*}
\hbtheta =
  \argmin_{ \btheta }\left( \frac{1}{n} \sum_{i=1}^n L(\boldsymbol{X}_{i}; \btheta) + p_{\lambda_n}(\btheta ) \right).
\end{equation*}
If both penalty and loss are differentiable, the first-order condition for $\hbtheta$ is
\begin{align*}
  \frac{1}{n} \sum_{i=1}^n -\nabla_{\btheta}  L(\boldsymbol{X}_{i}; \hbtheta) - \nabla_{\btheta} p_{\lambda_n}(\hbtheta )  = \0.
\end{align*}
Sparsity inducing penalties are not differentiable at points with zero entries.
The vector $\hbtheta$ is a solution to the minimization problem when the gradients cross $\0$ at $\hbtheta$, either continuously or with a jump.

To appropriately deal with non-differentiable or nonconvex penalties, we reformulate the estimation problem.
For a locally Lipschitz function $f \colon \R^p \to \R$, 
\citet[Chapter 2]{clarke1990optimization} uses the \emph{generalized directional derivative} of $f$ at $\btheta$ in the direction $\bDelta$, denoted by $f^\circ(\btheta; \bDelta)$, to define the \emph{generalized gradient} of $f$ at $\btheta$, denoted by $\partial f(\btheta)$: 
\[
\partial f(\btheta) = \{ \bz \in \R^p :  f^\circ(\btheta; \bDelta) \ge \langle \bz, \bDelta \rangle \, \forall \bDelta \in \R^p \}, \quad \text{where } f^\circ(\btheta; \bDelta)  = \limsup_{\bttheta \to \btheta, \, t \downarrow 0 } \frac{f(\bttheta + t \bDelta) - f(\bttheta)}{t}.
\]
The generalized gradient reduces to the usual gradient if $f$ is differentiable.
For convex functions, $\partial f (\btheta)$ coincides with the \emph{subdifferential}, i.e., the set of all $\bz \in \R^p$ satisfying $f(\btheta + \bDelta) - f(\btheta) \geq  \langle \bz,  \bDelta \rangle$.
With this notation, the estimator $\hbtheta$ of a penalized risk minimization problem solves $\frac{1}{n} \sum_{i=1}^n -\nabla_{\btheta}  L(\boldsymbol{X}_{i}; \hbtheta)   \in \partial p_{\lambda_n}(\hbtheta )$.
More generally, we define a penalized estimator $\hbtheta$ as a solution to
\begin{equation}
  \label{eq:Def_M}
  \Phi_n(\btheta) :=  \frac{1}{n} \sum_{i=1}^n \phi_{i}(\btheta) \in \partial p_{\blambda_n}(\btheta).
\end{equation}
We allow for a vector of tuning parameters $\blambda_n $. This encompasses, for example, coordinate-separable penalties of the form $p_{\blambda_n}(\btheta) = \sum_{k=1}^{p_n}   p_{\lambda_{n, k}}(\theta_{k})$, which can be useful when the estimator $\hbtheta$ is composed of solutions to sub-problems of different dimensionality or sample size.
The target parameter $\btheta^*$ still solves the unpenalized population equation $\sum_{i=1}^n \E[\phi_{i}(\btheta^*)] = \0$.
 In what follows, we assume without loss of generality that the true $\btheta^*$ can be written as $\btheta^* = (\btheta^*_{(1)},  \btheta^*_{(2)})$ with $\btheta^*_{(1)} \in \R^{s_n}$ and $\btheta^*_{(2)} = \0 \in \R^{p_n - s_n}$.
Similarly, write $\bv_{(1)} = (v_1, \dots, v_{s_n})$, $\bv_{(2)} = (v_{s_n + 1}, \dots, v_{p_n})$ for any vector $\bv \in \R^{p_n}$.
By $\partial f(\btheta)_{(1)}$ and $\partial f(\btheta)_{(2)}$, we denote the sets of all subvectors $\bz_{(1)}$ and $\bz_{(2)}$ such that $\bz \in \partial f(\btheta)$.

A penalty that is not differentiable at $\0$ can induce sparsity. To illustrate this, consider the Lasso penalty $p_{\lambda_n}(\btheta) = \lambda_n \|\btheta\|_1$ with scalar $\lambda_n$.
Its generalized gradient is a singleton at nonzero coordinates and the interval $[-\lambda_n,\lambda_n]$ at zero.
Hence, a nonzero solution satisfies $\Phi_n(\hbtheta)_k = \lambda_n \sign(\hat\theta_k)$, which induces shrinkage, while $\hat\theta_k=0$ only requires $\Phi_n(\hbtheta)_k \in [-\lambda_n,\lambda_n]$.
This additional slack makes sparse solutions possible; SCAD \citep{Fan1997, Fan01} and MCP \citep{Zhang10} have the same behavior near zero.

To cover non-coordinate-separable penalties such as the Group Lasso \citep{GroupL}, let $G_1, \ldots, G_{K_n} \subseteq \{1,\ldots,p_n\}$ be potentially overlapping groups covering all coordinates, and denote the corresponding subvectors by $\btheta_{G_g}$ and $\partial f(\btheta)_{G_g}$.
Let $I_{(2)}$ contain the groups associated with $\btheta_{(2)}$, so that $\bigcup_{g \in I_{(2)}}G_g=\{s_n+1,\ldots,p_n\}$, and define $I_{(1)}$ analogously.
The Group Lasso $p_{\blambda_n}(\btheta) =  \sum_{g= 1}^{K_n} \lambda_{n, g}  \| \btheta_{ G_g} \|_2$ replaces the interval at an inactive coordinate by an $\ell_2$-ball on an inactive group.
The supplement lists popular penalties covered by this framework and their generalized gradients.
To our knowledge, existing results do not jointly cover nonconvex and group-structured penalties, including combinations such as Group SCAD.

\subsection{Conditions for the Penalty} \label{sec:penalty_examples}

Let $\tilde r_n$ be the target rate defined in \cref{subsec:assump_penalty}.
Retain the partition $\btheta=(\btheta_{(1)},\btheta_{(2)})$, the groups $G_1,\ldots,G_{K_n}$, and the index sets $I_{(1)}$ and $I_{(2)}$ introduced above.
Define $\Theta_n' = \{(\btheta_{(1)}, \0)\colon \|\btheta - \btheta^*\| \le \tilde r_n a_n\}$, where $a_n \to \infty$ arbitrarily slowly, and let $\Theta_n \supseteq \Theta_n'$ be specified in the theorems below.

\begin{enumerate}[label=(P\arabic*)]

  \item \label{A:Penalty1} The penalty $p_{\blambda_n}$ is twice continuously differentiable with respect to $\btheta_{(1)}$ at $\btheta \in \Theta'_{n}$ with $ \sup_{\btheta \in \Theta'_{n} } \|\bar p_{\blambda_n}(\btheta^*, \btheta)  \| = o(1)$, where $\bar p_{\blambda_n}(\btheta^*, \btheta) \coloneqq \int_0^1 \nabla^2_{\btheta_{(1)}} p_{\blambda_n}(\btheta) \rvert_{\btheta = \btheta^* + t(\btheta - \btheta^*)} dt$.
\end{enumerate}
Note that this implies that $\partial p_{\blambda_n} (\btheta)_{(1 )}$ contains the single vector $\nabla_{\btheta_{(1)}} p_{\blambda_n}(\btheta)$ for $\btheta \in \Theta'_n$.
\begin{enumerate}[label=(P\arabic*), resume]
  \item \label{A:Penalty2} The generalized gradient satisfies, for all $\btheta \in \Theta_n'$ and $g \in I_{(2)}$,
  $
  \partial p_{\blambda_n}(\btheta)_{G_g} \supseteq B(\0_{ | G_g |}, \lambda_{n,g}).
  $

  \item \label{A:Penalty3} There is $\mu_n \ge 0$ such that for any $\btheta, \btheta' \in \Theta_n$ and $\bz \in \partial p_{\blambda_n}(\btheta), \bz' \in \partial p_{\blambda_n}(\btheta')$,
        \begin{align*}
          \langle \btheta' - \btheta, \bz' - \bz \rangle \ge - \mu_n \|\btheta' - \btheta\|^2.
        \end{align*}

  \item \label{A:Penalty4} The Hessian matrix of $p_{\blambda_n}$  satisfies $\sup_{\btheta \in \Theta'_{n} } \| \bar p_{\blambda_n}(\btheta^*, \btheta)  \| = o( (\sqrt{n} \, \tilde r_n )^{-1})$.
\end{enumerate}
 For the Lasso and SCAD, differentiability in \ref{A:Penalty1} requires the beta-min condition $\min_{1 \leq k \leq s_n} | \theta^*_{k} | / \tilde r_n \to \infty$; for the Group Lasso, the active groups must likewise remain separated from zero.
The Hessian condition is mild and typically implied by $\max_g \lambda_{n, g} \to 0$.

Assumption \ref{A:Penalty2} induces sparsity.
It formalizes the sparsity mechanism described above 
reduces to
\[
 \partial p_{\blambda_n}(\btheta)_{(2)} \supseteq [-\lambda_{n, s_n + 1}, \lambda_{n, s_n + 1} ] \times \cdots \times  [-\lambda_{n, p_n}, \lambda_{n, p_n} ],
\]
for coordinate-separable penalties. The general condition holds with equality for (Group) Lasso, (Group) SCAD, and MCP; the more general inclusion also covers $\ell_q$-penalties with $q < 1$ and asymmetric penalties.
Assumption \ref{A:Penalty3} controls nonconvexity and is used for uniqueness.
It holds with $\mu_n=0$ for convex penalties, with $\mu_n=(a-1)^{-1}$ and $a^{-1}$ for (Group) SCAD and MCP, respectively, and is weaker than the $\mu$-amenability condition of \citet{LohWW2,Loh}.
It fails for $\ell_q$-penalties with $q<1$, which cannot yield unique solutions unless $\btheta^*=\0$.
Finally, \ref{A:Penalty4} refines \ref{A:Penalty1} for asymptotic normality and is immediate for penalties such as Lasso and SCAD once their Hessian vanishes near the active coefficients.
Having isolated the required properties of the penalty, we now combine them with curvature and stochastic regularity conditions on the estimating equation.

\subsection{Notation and Assumptions}\label{subsec:assump_penalty}

To simplify some conditions, we assume from now on that $p_n \ge n^{a}$ for some $a > 0$. This only excludes uninteresting edge cases where $p_n$ is effectively constant.
Define for $A \in \R^{p_n \times p_n}$
\begin{align*}
A_{(1)} & = \left(A_{ij}\right)_{i, j = 1, \ldots, s_n} \in \R^{s_n \times s_n}, \quad 
A_{(2,1)} = \left(A_{ij}\right)_{i = s_n + 1, \ldots, p_n; j = 1, \ldots, s_n} \in \R^{(p_n - s_n) \times s_n}, \\
A_{G_g, (1)} & = \left(A_{ij}\right)_{i \in G_g; j = 1, \ldots, s_n} \in \R^{|G_g| \times s_n},
\end{align*}
e.g., $J(\btheta)_{G_g, (1)}  = \frac 1 n \sum_{i = 1}^n \nabla^\top_{\btheta_{(1)}}   \E[\phi_{}(\bX_{i};\btheta)_{G_g}] $.
To state the required assumptions, let 
\begin{align}
\label{eq:rTilde}
  b_n^*= \|\nabla_{\btheta_{(1)}} p_{\blambda_n}(\btheta^*) \|_{\infty}, \quad \bar b_{n} =  \sup_{\bv \in \partial p_{\blambda_n}(\btheta), \ \btheta \in \Theta_n}\|\bv \|_\infty, \quad \tilde r_n =\sqrt{\tr(I(\btheta^*)_{(1)})/n} + \sqrt{s_n}b_n^*,
\end{align}
and let $\eta_n$ be any sequence such that $\eta_n \ge 2 \sigma_n \sqrt{\ln(p_n)/n}$, where $\sigma_n = \max_{1 \le k \le p_n} \sqrt{\frac 1 n \sumin \E[ \phi_{i}(\btheta^*)_k^2]}$.
Recall $\Theta_n' = \{(\btheta_{(1)}, \0)\colon \|\btheta - \btheta^*\| \le \tilde r_n a_n\}$ with $a_n \to \infty$ arbitrarily slowly and $\btheta_{(1)} \in \R^{s_n}$. 
Define the cones 
\[
\Theta^*(\nu_n) =  \{\btheta \colon \|\btheta^* - \btheta\|_1 \le \sqrt{\nu_n} \|\btheta^* - \btheta\|_2 \},
\]
and suppose that $\Theta_n$ are sets such that $\Theta_n'\subseteq \Theta_n \subseteq \Theta^*(\nu_n)$ for some $\nu_n$.

The regularity conditions below have two roles: Assumptions \ref{A:phi-H-penalty} and \ref{A:phi-moments} control curvature and score tails, while \ref{A:lambda}--\ref{A:Penalty_emp_pr2} ensure that a solution of the reduced problem also solves the full penalized equation.

\begin{enumerate}[label=(A\arabic*), resume=assump]
  \item \label{A:phi-H-penalty} There exists a sequence of symmetric, matrix-valued functions $H_n(\bx)$ such that:
  \begin{enumerate}[label = (\roman*)]
    \item \label{eq:Hn-def2} For all  $\btheta^* + \bu \in \Theta_n$ and $\bx \in \Xcal$, it holds that
    $\bu^\top [\phi(\bx; \btheta^* + \bu) - \phi(\bx; \btheta^* )] \le \bu^\top H_n(\bx) \bu;$
    \item \label{eq:Hn-identifiable2}  $\limsup_{n \to \infty}\lambda_{\max}(n^{-1} \sumin\E[H_n(\bX_{i})]) \le -c < 0;$  
    \item  \label{eq:Hn-bounds2} For some sequence $\tilde B_n = o(n/(\nu_n \ln p_n))$, it holds that
    \begin{align*}
      \max_{1 \le j, k \le p_n} \frac 1 n  \sumin  \E[H_n(\bX_{i})_{j, k}^2]  & = o\left(\frac{n }{\nu^2_n \ln p_n}\right), \\ \sumin \Pr\lf(\max_{1 \le j, k \le p_n}  | H_n(\bX_{i})_{j, k} | > \tilde B_n\ri) & = o(1).
    \end{align*}
  \end{enumerate}

  \item \label{A:phi-moments} 
  It holds that
$\displaystyle \sum_{i = 1}^n \Pr\lf(\|\phi_{i}(\btheta^*)\|_\infty > \sigma_n \sqrt{\frac{n}{4\ln p_n}} \ri) = o(1).$
\item \label{A:lambda}
        $\bar J (\btheta^*, \btheta)_{(1)}$ is invertible for all $\btheta \in \Theta'_n$ and there is $\alpha\in [0,1)$ such that 
        \begin{equation}
          \label{eq:alpha}
           \max_{g \in I_{(2)}} \lambda_{n,g}^{-1} \|  \bar J(\btheta^*, \btheta)_{G_g,(1)} \;\bar J(\btheta^*, \btheta)^{-1}_{(1)} \;  \nabla_{\btheta_{(1)}} p_{\blambda_n}(\btheta) \|_2 \le \alpha \quad \text{ uniformly for all $\btheta \in \Theta'_n$}.
        \end{equation}
        
  \item \label{A:lambda2} For all $g \in I_{(2)}$, $\lambda_{n,g}$ must fulfill, with $\alpha$ as defined in \ref{A:lambda} and $\eta_n$ as defined above,
        \[
          \lambda_{n,g} \ge  \frac{4}{1-\alpha} \: \eta_n \: \sqrt{| G_g |} \quad \text{and} \quad \lambda_{n,g} \ge \frac{4}{1-\alpha} \: \eta_n \: J_{n,g}  ,
        \]
        where $J_{n,g} = \sup_{\btheta \in  \Theta'_n, \| \bx \|_\infty \le 1}\| \bar J(\btheta^*, \btheta)_{G_g,(1)} \;\bar J(\btheta^*, \btheta)^{-1}_{(1)}  \bx \|_2$.

  \item \label{A:Penalty_emp_pr2} With $\tilde r_n$, $\eta_n$ as defined above and some sequence $\tD_n = o( n \eta_n / (\tilde r_ns_n +\tilde r_n\ln p_n))$ and $\tilde \phi_i(\btheta) = \phi_i(\btheta)$ or $\tilde \phi_i(\btheta) = \phi_i(\btheta) - \E[\phi_i(\btheta)]$, it holds that
        \begin{align*}
            \max_{1 \le k \le p_n}  \sup_{\btheta, \btheta' \in \Theta'_n} \frac 1 n \sumin \frac{\E[| \tilde \phi_{i}(\btheta)_k - \tilde \phi_{i}(\btheta')_k |^2]}{\| \btheta - \btheta' \|^2} & = o\lf(\frac{n \eta_n^2 }{\tilde r_n^2(s_n + \ln p_n)}\ri),      \\
          \sumin \Pr\lf( \sup_{\btheta, \btheta' \in \Theta'_n}  \frac{\|\tilde \phi_{i}(\btheta) - \tilde \phi_{i}(\btheta') \|_\infty}{\| \btheta - \btheta' \|} > \tD_n \ri)   & = o(1),
        \end{align*}
        where the same choice of $\tilde \phi_i(\btheta)$ must be used in both conditions.
\end{enumerate} Assumption \ref{A:phi-H-penalty} is a variant of \ref{A:Cons1} ensuring restricted concavity with high probability; see the following section.
Assumption \ref{A:phi-moments} trades off the tails of the estimating equation against $p_n$: if $\sigma_n = O(1)$ and all $\phi_{}(\bX_{i}; \btheta^*)_k$ are sub-Gaussian, it allows $p_n \sim e^{n^a}$ for $a \in (0, 1/2)$.
Assumption \ref{A:lambda} generalizes the \emph{mutual incoherence} or \emph{irrepresentable conditions} \citep{WWbook, Zhao, BuehlBook}; see \cref{sec:mutual}, while \ref{A:lambda2} makes the inactive-group penalties large enough to induce sparsity.
Unlike sample-dependent formulations based on $n^{-1}\|\sumin \phi_{i}(\btheta^*)\|_{\infty}$ \citep{WWpaper, LohWW}, we use the population bound $\eta_n$.
Finally, \ref{A:Penalty_emp_pr2} controls convergence to the population equation and imposes growth restrictions on $p_n$ analogous to \ref{A:phi-moments}.
Increasing $\eta_n$ weakens this condition but requires larger $\lambda_{n,g}$ in \ref{A:lambda2}, which is less problematic for nonconvex penalties because these parameters typically do not enter the convergence rate; see the discussion after \cref{theorem4}.

\subsection{Estimation and selection consistency}
\label{sec:ass_penalty}
\subsubsection{Main results}\label{sec:main-results-pen}

We will provide three main results for penalized estimators. The first guarantees that any sparse solution to the penalized estimating equation is close to the true parameter vector.

\begin{theorem} \label{theorem3} 
  Suppose that \ref{A:phi-H-penalty} and \ref{A:phi-moments} hold.
  Then, any solution $\hbtheta \in \Theta_n \subseteq \Theta^*(\nu_n)$ satisfies 
  \begin{align*}
    \lVert \hbtheta - \btheta^* \rVert = O_p\left(\sqrt{\nu_n}(\eta_n + \bar b_{n}) \right).
  \end{align*}
\end{theorem}

The optimal rate of convergence $\sqrt{s_n} \eta_n$ is attained if $\nu_n = O(s_n)$ and $\bar b_n = O(\eta_n)$ which typically holds if $\|\blambda_n\|_\infty = O(\eta_n)$. In most classical applications, this rate simplifies to the usual convergence rate $\sqrt{s_n \ln p_n / n}$.
A similar rate was obtained by \citet{LohWW} under a \emph{restricted strong convexity} (RSC) condition on the realized sample, whereas \cref{theorem3} uses population-level conditions.
In most standard cases, \ref{A:phi-H-penalty} implies with high probability the weaker RSC-type condition
\begin{align*}
  \lf\langle \bu, \Phi_n(\btheta^*) - \Phi_n(\btheta^* + \bu)  \ri\rangle \ge  c\|\bu\|^2 - c_1\|\bu\|_1^2 \eta_n,
\end{align*}
for some constant $c_1 \ge 0$.
Unlike \cref{theorem3}, which is restricted to cones $\Theta^*(\nu_n)$, \citet{LohWW, WWpaper} and \citet[Chapter 9]{WWbook} allow larger sets of the form $\{\|\btheta\|_1 \le k_n\}$ by imposing the stronger condition
\begin{align} \label{eq:RSC-LW}
  \lf\langle \bu, \Phi_n(\btheta^*) - \Phi_n(\btheta^* + \bu)  \ri\rangle \ge  c\|\bu\|^2 - c_1\|\bu\|_1^2 \eta_n^2,
\end{align}
with high probability for all sufficiently small $\bu$.
Under mild assumptions on the penalty and tuning parameter, \eqref{eq:RSC-LW} implies that $\hbtheta \in \Theta^*(\nu_n)$ with $\nu_n = O(s_n)$; see Lemma 5 in the supplement.
To the best of our knowledge, \eqref{eq:RSC-LW} has only been verified for variations of the linear model, where the rank-one structure of the sample-Hessian can be exploited, and it fails in the example of \cref{sec:RSC-failure}.
Restricting the statement to cones avoids this issue: any $\nu_n$-sparse solution belongs to $\Theta^*(\nu_n)$, so \cref{theorem3} applies.

\cref{theorem3} does not say anything about existence of solutions and selection consistency, which also explains why it does not require a lower bound on the tuning parameters.
The next theorem shows that there is an estimation and selection consistent solution to the penalized equation with high probability.
Specifically, we show that, with high probability, the reduced problem 
\begin{align} \label{eq:Def_M_reduced}
\frac{1}{n} \sum_{i = 1}^n \phi_{i}((\btheta_{(1)}, \0))_{(1)} \in \partial p_{\blambda_n}((\btheta_{(1)}, \0))_{(1)}.
\end{align}
has a solution $\hbtheta_{(1)}$ close to $\btheta_{(1)}^*$, and that $(\hbtheta_{(1)}, \0)$ is also a solution to the full problem \eqref{eq:Def_M}. 
\begin{theorem}
  \label{theorem4}
  Suppose the reduced problem \eqref{eq:Def_M_reduced} satisfies \ref{A:Cons1} on sets $ \{\btheta_{(1)}\colon \btheta \in \Theta_n'\}$.
  Suppose further that \ref{A:Penalty1}, \ref{A:Penalty2} and \ref{A:phi-moments}--\ref{A:Penalty_emp_pr2} hold.
  Then, with probability tending to $1$, the sets $\Theta_n'$ contain a solution $\hbtheta$ of the penalized estimating equation (\ref{eq:Def_M}) such that
  \[ 
  \lVert \hbtheta - \btheta^* \rVert = O_p\left(\sqrt{\tr(I(\btheta^*)_{(1)})/n} + \sqrt{s_n}b_n^*\right), \quad \hbtheta_{(2)}  =  \0.
  \]
\end{theorem}

To simplify the discussion, suppose $\tr(I(\btheta^*)_{(1)}) = O(s_n)$, which is the most common situation.
The rate of convergence depends on the number of non-zero coefficients $s_n$ of $\btheta^*$ and the bias $b_n^*$ (defined in \eqref{eq:rTilde}) induced by the penalty. 
For example, we have $b_n^* = \lambda_n$ for the Lasso,  and $b_n^*  = 0$ for SCAD for $n$ large enough if $\min_{1 \leq k \leq s_n} | \theta^*_{k} |  / \lambda_n \to \infty$.
For a Lasso-type penalty with different $\lambda_{n, k}$, i.e., $p_{\blambda_n}(\btheta) = \sum_{k=1}^{p_n} \lambda_{n, k} | \theta_{k} |$, we obtain $b_n^* = \max_{k = 1, \ldots, s_n} \lambda_{n, k}$. 
The rate of convergence does not explicitly involve the overall number of parameters $p_n$, although it may enter through the bias $b_n^*$ (typically logarithmically).  Similarly, the regularity conditions above usually depend on $p_n$ only logarithmically.
This shows that estimation of $\btheta^*$ is possible even if the total number of parameters $p_n$ is much larger than the sample size $n$. 
The non-asymptotic results of \citet{WWpaper} and \citet{LohWW} imply a rate of $O_p(\sqrt{s_n \ln p_n / n})$. For the SCAD and similar penalties, this is still suboptimal as \cref{theorem4} gives the rate $O_p(\sqrt{s_n / n})$.

In practice, we do not know which parameters $\theta_{k}$ are 0, so it is unclear whether the solution from \cref{theorem4} can be found.
Our final theorem gives conditions under which the solution is unique.

\begin{theorem} \label{theorem5}
  Suppose that the conditions of  \cref{theorem3} are satisfied and \ref{A:Penalty3} holds with $\limsup_{n \to \infty} \mu_n \allowbreak < c$, where $c$ is defined in \ref{A:phi-H-penalty}.
  Suppose further that for all $\btheta , \btheta + \bu \in \Theta_n  \cap \{\|\btheta - \btheta^*\| \le a_n  \sqrt{\nu_n}(\eta_n + \bar b_n) \}$ with $a_n \to \infty$ arbitrarily slowly and all $\bx \in \Xcal$, it holds that
  \begin{align} \label{eq:Hn-def-unique-pen}
    \bu^\top [\phi(\bx; \btheta + \bu) - \phi(\bx; \btheta )] \le \bu^\top H_n(\bx) \bu.
  \end{align}
  Then, with probability tending to 1, there is at most one solution in $\Theta_n$.
\end{theorem}
\cref{theorem5} implies the uniqueness of any solution in $\Theta_n$. 
If the conditions of \cref{theorem4} hold, the selection consistent estimator $\hbtheta$ is unique on $\Theta_n$ with probability tending to 1.

\subsubsection{Mutual incoherence \& irrepresentable condition}\label{sec:mutual}

Assumption \ref{A:lambda} generalizes familiar mutual incoherence and irrepresentable conditions to our broader framework.
It requires some $\alpha \in [0,1)$ such that
\[
 \sup_{\btheta \in \Theta'_n} \max_{g \in I_{(2)}} \lambda_{n,g}^{-1} \|  \bar J(\btheta^*, \btheta)_{G_g,(1)} \;\bar J(\btheta^*, \btheta)^{-1}_{(1)} \;  \nabla_{\btheta_{(1)}} p_{\blambda_n}(\btheta) \|_2 \le \alpha.
 \]
For SCAD and MCP, the condition holds with $\alpha=0$ whenever $\min_{1 \leq k \leq s_n} (|\theta_k^*|-\tilde r_n)/\lambda_n\to\infty$; for Group SCAD, the analogous requirement is $\min_{g\in I_{(1)}}(\|\btheta_{G_g}^*\|-\tilde r_n)/\lambda_{n,g}\to\infty$.
For the Lasso and a scalar $\lambda_n$, the terms involving $\lambda_n$ cancel and the condition can be simplified to
\begin{align}
\label{eq:alpha1}
\sup_{\btheta \in \Theta'_n} \left\| \bar J(\btheta^*, \btheta)_{(2,1)} \;\bar J(\btheta^*, \btheta)^{-1}_{(1)} \; \sign(\btheta_{(1)})  \right\|_\infty \le \alpha,  
\end{align}
where $\sign(\btheta)$ denotes the vector of signs.
This condition is implied by the stronger assumption
\begin{equation}\label{eq:alpha_stronger}
  \sup_{\btheta \in \Theta'_n} \lf\|   \bar J(\btheta^*, \btheta)_{(2,1)} \;\bar J(\btheta^*, \btheta)^{-1}_{(1)}  \ri\|_\infty \le \alpha,
\end{equation}
where $\| \cdot \|_\infty$ denotes the maximum absolute row sum of a matrix.
Both conditions characterize the population equation and are independent of a scalar $\lambda_n$.
The normalized Jacobian block $\bar J(\btheta^*,\btheta)_{(2,1)}\bar J(\btheta^*,\btheta)^{-1}_{(1)}$ measures how errors in the active coordinates affect the inactive equations.
\ref{A:lambda} requires this effect to be small enough for the penalty to remove the resulting bias on $\btheta_2$, induced by the biased estimation of $\btheta_1$.
In the linear model, $J(\btheta)_{(2,1)}=\E[\bX_{(2)}\bX_{(1)}^\top]$ and $J(\btheta)_{(1)}=\E[\bX_{(1)}\bX_{(1)}^\top]$, so the condition restricts the dependence between active and inactive covariates.
\eqref{eq:alpha_stronger} is a population version of the mutual incoherence condition of \citet[Section 7.5.1]{WWbook}, which requires approximate orthogonality.
\eqref{eq:alpha1} is the irrepresentable condition of \citet{Zhao} and is necessary for Lasso selection consistency \citep{Zhao,Zou}, but not for SCAD, MCP, and related nonconvex penalties \citep{LohWW2}.
For the Group Lasso, \ref{A:lambda} reduces to the condition of \citet{Bach}.

For heterogeneous tuning parameters, \ref{A:lambda} also describes how the entries of $\blambda_n$ should reflect the sensitivity of the equations.
For example, consider the weighted Lasso $p_{\blambda_n}(\btheta)=\sum_{k=1}^{p_n}\lambda_{n,k}|\theta_k|$.
Assuming $J(\btheta)_{(1)} = I_{s_n}$ and replacing $ \bar J(\btheta^*, \btheta)_{k,(1)}$ by $J(\btheta)_{k,(1)}$ for simplicity, condition \ref{A:lambda} is implied by 
\begin{align*}
  \frac{1}{\lambda_{n, j}} \sum_{k = 1}^{s_n} \lambda_{n, k}  \lf|  \E\lf[ \frac{\partial }{\partial \theta_{k}} \phi(\btheta)_j \ri] \ri| \leq \alpha < 1 \quad \text{uniformly for all }j = s_n + 1, \ldots, p_n.
\end{align*}
Thus, a large effect of active-coordinate errors on equation $j$ requires smaller active penalties or a larger $\lambda_{n,j}$.
Although the unknown support precludes tuning directly by activity, the weights may sometimes be scaled using the magnitudes of $\E[\nabla_{\btheta}\phi(\btheta)_j]$.

\subsection{Asymptotic Normality and the Oracle Property}
\label{sec:normality_pen}

Finally, we establish asymptotic normality of the estimator $\hbtheta_{(1)}$.

\begin{theorem}
\label{theorem6}
Suppose that  \ref{A:Penalty4} holds and that for some matrix $A_n \in \R^{q \times s_n}$ with $\| A_n \| = O(1)$ and for which $\Sigma = \lim_{n \to \infty} A_n  I(\btheta^*)_{(1)}  A_n^\top$ exists, the reduced problem \eqref{eq:Def_M_reduced} satisfies the conditions of \cref{theorem2} with $r_n$ replaced by $\tilde r_n $ as defined in \eqref{eq:rTilde}.
Then, the estimator $\hbtheta_{(1)}$ in \cref{theorem4} satisfies
\[
\sqrt{n} A_n \left[J(\btheta^*)_{(1)}(\hbtheta_{(1)} - \btheta^*_{(1)}) - \nabla_{\btheta_{(1)}} p_{\blambda_n}(\btheta^*) \right] \to_d \Ncal(\bnull, \Sigma).
\]
\end{theorem}

Recall that $\hbtheta$ has the oracle property if it behaves like the hypothetical \emph{oracle estimator} that knows $\btheta^*_{(2)} = \0$ in advance. This is the case if  $\hbtheta_{(2)} = \0$ and $\nabla_{\btheta_{(1)}} p_{\blambda_n}(\btheta^*) = o(1/\sqrt{n})$. 
For SCAD and MCP with $\lambda_n \to 0$ fast enough, one obtains $b_n^* = 0$, so the rate of convergence is the same as if $\btheta^*_{(2)} = \0$ is known in advance.
Since $\nabla_{\btheta_{(1)}} p_{\blambda_n}(\btheta^*) = \0$, the SCAD and MCP-penalized estimators have  the same efficiency as the oracle estimator.
As long as $\blambda_n$ is small enough that $\nabla_{\btheta_{(1)}} p_{\blambda_n}(\btheta^*) = \0$ asymptotically, it can be chosen large enough such that \ref{A:Penalty_emp_pr2} is fulfilled.  

For the Lasso, we have $b_n^* = \lambda_n$ and $\nabla_{\btheta_{(1)}} p_{\blambda_n}(\btheta^*) = O(\sqrt{s_n} \lambda_n)$. 
For the oracle property to hold, we would need $\lambda_n = o(1/\sqrt{s_n n})$.
This condition cannot be satisfied, because  \ref{A:lambda}  requires $\lambda_n  \ge \eta_n  \approx \sqrt{\ln p_n / n}.$
This is in line with the results from \citet{Zou}, who shows that in the linear model, the Lasso is only variable selection consistent at the cost of a slower rate of convergence.
 
\section{Extension to dependent data}\label{sec:dependence}

Our results were stated for independent data mainly to simplify the discussion of already
technical conditions and reduce the proofs to the core ideas. The arguments and results can
be extended to various types of dependent data, such as time series, spatial data, or graph
dependence. 

To set the stage, suppose we observe a sequence of dependent random variables
$\bZ_1,\dots,\bZ_N$ and we are interested in the solution $\hbtheta$
solving the estimating equation
$
   \sum_{j=1}^N \varphi(\bZ_j;\btheta)=\bnull.
$
Let $m_N\in\N$ be a divisor of $N$, put $n_N=N/m_N$, and let
$I_{1,N},\dots,I_{n_N,N}$ be a partition of $\{1,\dots,N\}$ with
$|I_{i,N}|=m_N$. For time series one may take
$I_{i,N}=\{(i-1)m_N+1,\dots,im_N\}$. Define
\[
 \textstyle \bX_i=(\bZ_j:j\in I_{i,N}),
  \qquad
  \phi(\bX_i;\btheta)
  =
  \frac 1 {m_N}  \sum_{j\in I_{i,N}}\varphi(\bZ_j;\btheta),
  \qquad
  i=1,\dots,n_N.
\]
Then the estimating equation can be rewritten as
$
  \frac 1{n_N}\sum_{i=1}^{n_N} \phi(\bX_i;\btheta)
  =
  \bnull.
$

If the blocks
$\bX_1,\dots,\bX_{n_N}$ are independent, we can apply the
results from the previous sections to the block-wise estimating equation.
While the block-wise independence assumption is strong, it serves as a useful starting point
for understanding how to extend our results to dependent data. To comply with the high-level
view of this paper, we will now state an abstract dependence condition that applies to a wide
range of dependence scenarios and slots nicely into our arguments for the independent case.
Blocking and coupling arguments of this type are standard in asymptotic theory for weakly dependent processes \citep[see, e.g.,][]{rio2017asymptotic}.

For sub-$\sigma$-fields $\mathcal A$ and $\mathcal C$, define the beta-mixing
coefficient by
\[
  \textstyle\beta(\mathcal A,\mathcal C)
  =
  \frac12
  \sup_{\{A_i\},\{C_j\}}
  \sum_{i,j}
  \left|\Pr(A_i\cap C_j)-\Pr(A_i)\Pr(C_j)\right|,
\]
where the supremum is over all finite partitions of the sample space with
$A_i\in\mathcal A$ and $C_j\in\mathcal C$.

\begin{enumerate}[label=(D)] 
  \item  \label{ass:dep} There exist $m_N = o(N), L_N$ and a partition $I_{1,N},\dots,I_{n_N,N}$ of $\{1,\dots,N\}$ with $n_N \coloneqq N/m_N$ and $|I_{i,N}|=m_N$, such that: 
  \begin{enumerate}[label=(\roman*)]
    \item $L_N / n_N \to1$,
    \item $L_N \max_{1\le i \le L_N}
    \beta\!\left(
      \sigma\{\bZ_j: j \in I_{i,N}\},
      \sigma\{\bZ_j: j \in\bigcup_{ \ell \le L_N, \ell\ne i} I_{\ell,N} \}
    \right) \to 0$.
  \end{enumerate}
\end{enumerate}
Intuitively, the condition demands that, if we drop blocks $I_{k,N}$ for $k>L_N$, the dependence between the remaining blocks is weak. Part (i) requires that we only drop a negligible fraction of blocks. Part (ii) ensures that the dependence fades sufficiently quickly such that a standard coupling argument allows replacing the retained blocks by independent versions with high probability.

For the omitted blocks to have truly negligible effect, we require an additional condition that depends on the context. The argument is summarized in the following lemma, whose proof is given in the supplement.

\begin{lemma}\label{lem:dep}
 Suppose that condition \ref{ass:dep} holds, and let $\Fcal$ be a collection of functions. 
 If 
 \begin{align} \label{cond:envelope-dropped}
  \sum_{i = L_N + 1}^{n_N} \Pr\left(\sup_{f \in \Fcal} |f(\bX_i)| > \frac{\eps_N}{n_N - L_N} \right) = o(1),
 \end{align}
 holds for a given $\eps_N$, then there are independent random variables $\bX_1^*, \dots, \bX_{n_N}^*$ with $\bX_i^* =_d \bX_i$, $i=1,\dots,n_N$, such that
 \begin{align*}
   \sup_{f \in \Fcal} \left| \frac{1}{n_N}\sum_{i=1}^{n_N} f(\bX_i) -  \frac{1}{n_N}\sum_{i=1}^{n_N}f(\bX_i^*) \right| = o_p(\eps_N / n_N).
 \end{align*}
\end{lemma}

Adapting the proofs of our main results to the dependent case now amounts to invoking \cref{lem:dep} with the appropriate function class $\Fcal$ and choice of threshold $\eps_N$ before applying any concentration or limit theorem. 
Conveniently, the assumptions of our main results already involve threshold sequences that control \eqref{cond:envelope-dropped} for the relevant quantities.
The following theorem summarizes the conditions under which our main results can be extended to the dependent case.

\begin{theorem}\label{th:dependence}
Consider the above setting and suppose that \ref{ass:dep} holds. Let $n:=n_N$ and $B_n, \tilde B_n, D_n,  r_n, \tilde r_n$ denote the corresponding quantities from \ref{A:Cons1}--\ref{A:Penalty_emp_pr2}. 
We require additional conditions of the form 
\begin{equation}\label{eq:assumpDep}
   \sum_{i=L_N + 1}^{n} \Pr\left( T_n > E_n \right) = o(1).
\end{equation}
\begin{itemize}
  \item The conclusions of Theorems \ref{theorem1} and \ref{theorem1-uniqueness} remain true provided $B_n = O(n/(n-L_N))$ and \eqref{eq:assumpDep} holds with $T_n = \| \phi_i(\btheta^*) \|_\infty, E_n = r_n n /(\sqrt{p_n} (n-L_N))$.
  \item The conclusion of \cref{theorem2} remains true provided $D_n = O(\sqrt{n}/(n - L_N))$, $(n - L_N)^5/n = O(1)$.
\item  The conclusions of Theorems \ref{theorem3} and \ref{theorem5} remain true provided $\tilde B_n = O(n / (\nu_n (n- L_N)))$ and \eqref{eq:assumpDep} holds with $T_n = \| \phi_i(\btheta^*) \|_\infty, E_n = \eta_n n / (n-L_N)$.
  \item The conclusion of \cref{theorem4} remains true provided the reduced problem \eqref{eq:Def_M_reduced} fulfills the adapted conditions of \cref{theorem1} and \eqref{eq:assumpDep} holds with $T_n = \sup_{\btheta \in \Theta'_n} \| \phi_i(\btheta)_{(2)} \|_{\infty} , E_n =  \eta_n n / (n-L_N)$.
  \item The conclusion of \cref{theorem6} remains true provided the reduced problem \eqref{eq:Def_M_reduced} fulfills the adapted conditions of \cref{theorem2} with $r_n$ replaced by $\tilde r_n$.
\end{itemize}
\end{theorem}
The additional assumption \eqref{eq:assumpDep} implicitly imposes restrictions on $p_n$ and $n - L_N$. 
For example, if $r_n = \sqrt{p_n / n}$, Markov's inequality and the union bound imply \eqref{eq:assumpDep} with $T_n = \| \phi_i(\btheta^*) \|_\infty, E_n = r_n n /(\sqrt{p_n} (n-L_N))$ provided ${(n - L_N)^5 p_n \max_{i, k} \E[ \phi_i(\btheta^*)_k^4] /n^2 = o(1)}$.
The blocking construction affects the covariance matrix appearing in the asymptotic normality statements:
$$I_n(\btheta^*) = \frac{1}{n} \sum_{i=1}^n \cov(\phi(\bX_i;\btheta^*)) = \frac{1}{N m_N} \sum_{i = 1}^n \sum_{j, k \in I_{i, N}} \cov(\varphi(\bZ_j; \btheta^*), \varphi(\bZ_k; \btheta^*))$$
Although $I_n(\btheta^*)$ depends on the block assignments $I_{i, N}$ in \ref{ass:dep}, in most applications $m_N I_n(\btheta^*)$ can be replaced asymptotically by the covariance matrix $\wt I_N(\btheta^*) = N^{-1} \sum_{i=1}^N \sum_{j=1}^N \cov(\varphi(\bZ_i; \btheta^*), \varphi(\bZ_j; \btheta^*))$ of the original estimating equation.
A simple sufficient condition is that the covariance between retained blocks decays sufficiently quickly, e.g., 
\begin{align} \label{eq:cov-decay}
      \max_{1 \le i < j \le L_N} \| \cov(A_n \phi(\bX_i; \btheta^*), A_n \phi(\bX_j; \btheta^*)) \| = o(1/n),
\end{align}
and that $(n - L_N)^2 / n = o(1)$, which ensures that the contribution of the dropped blocks is negligible.

Let us now discuss how to verify the abstract condition \ref{ass:dep} in a few more concrete settings. 

\begin{example}[$\beta$-mixing time series]\label{ex:timeseries}
Let $(\bZ_t)_{t\in\mathbb Z}$ be a polynomially $\beta$-mixing time series, i.e., $\beta_{\rm ts}(q) = \sup_{t\in\mathbb Z}\beta\left(\sigma(\bZ_s:s\le t), \sigma(\bZ_s:s\ge t+q)\right) \lesssim q^{-\kappa}$.
Take
$m_N=\lfloor N^\gamma\rfloor$ and $q_N=\lfloor N^\delta\rfloor$ with
$0<\delta<\gamma<1$ and $\gamma+\kappa\delta>1$.
Define
$ L_N=\lfloor N / (m_N+q_N) \rfloor$
and $I_{k,N}
  =
  \{(k-1)(m_N+q_N)+1,\ldots,(k-1)(m_N+q_N)+m_N\}$ for $k=1,\ldots,L_N$.
This takes blocks of size $m_N$ and leaves gaps of size $q_N$ in between.
The omitted observations are partitioned arbitrarily into sets
$I_{L_N+1,N},\ldots,I_{n_N,N}$ of cardinality $m_N$.
Then $L_N/n_N \sim m_N/(m_N+q_N) \to 1$, and
\begin{align*}
    L_N
  \max_{1\le k\le L_N}
  \beta\!\left(
    \sigma\{\bZ_i:i\in I_{k,N}\},
    \sigma\{\bZ_i:i\in \cup_{\ell\le L_N,\ell\ne k} I_{\ell,N}\}
  \right)
  \lesssim
  \frac{N}{m_N}q_N^{-\kappa}
  \lesssim N^{1- \gamma} N^{-\kappa \delta}
  \to0 .
\end{align*}
Thus condition \ref{ass:dep} holds.
\cref{cor:TS} in \cref{sec:MEstGen} derives consistency and asymptotic normality of M-estimators in generalized linear models in this setting.
\end{example}
 
A similar construction works for a spatial grid, where retained blocks are formed from spatial hyper-cubes with gaps. A more interesting example arises when dependence is induced by a network.
Let $G = (V, E)$ be a graph with nodes $V=\{1, \ldots, N\}$ and edges $E$, and define the graph distance as $d_G(i,j) = \min\{k : \text{there is a path from }i\text{ to }j\text{ using }k\text{ edges}\}$.
Suppose that the observed data $\bZ_1, \ldots, \bZ_N$ are generated from a process such that the dependence between $\bZ_i$ and $\bZ_j$ decays as the graph distance $d_G(i,j)$ increases.
In a subcritical Erdős--Rényi graph, for instance, the largest connected component is of size $O(\log N)$, and different components are independent. In this case, we can construct blocks of size $m_N \gg \log N$ from connected components, and the retained blocks are independent. Such graphs are not realistic models for many applications in, e.g., social or transport networks, where the largest connected component is often of order $N$, and some nodes (hubs) have very high degree. The following examples illustrate how to handle such cases.

\begin{example}[Hub-screened network dependence]\label{ex:network}
Let $G_N=(V_N,E_N)$ be a graph with $V_N=\{1,\ldots,N\}$ and graph
distance $d_N$. For $A,B\subset V_N$, write
  $d_N(A,B)=\min_{i\in A,j\in B}d_N(i,j),$
and suppose the following bound for the graph-distance beta-mixing coefficient
\begin{align*}
  \beta_{{\rm net},N}(q)
  =
  \textstyle \sup_{A,B\subset V_N, d_N(A,B)\ge q}
  \beta\!\left(
    \sigma\{\bZ_i:i\in A\},
    \sigma\{\bZ_j:j\in B\}
  \right) \lesssim \beta(q)
\end{align*}
with some $\beta(q) \to 0$ for $q \to \infty$.
Suppose that there exist hub sets $H_N\subset V_N$ and radii
$q_N\to\infty$ such that, with $U_N=\{i\in V_N:d_N(i,H_N)\le q_N\}$, and with $\mathcal C_N$ denoting the connected components of
$G_N[V_N\setminus H_N]$ (the graph obtained by removing the hub nodes and their incident edges), $|U_N|=o(N)$ and $\max_{C\in\mathcal C_N}|C\setminus U_N|=o(N)$.
Then, condition \ref{ass:dep} holds.
The proof is given in the supplement.
\end{example}
 
\section{Applications}

\label{sec:Exmp}
Our general results are stated under rather abstract regularity conditions to keep them widely applicable.
In the following, we provide several examples of how these conditions simplify in specific applications.
While some of these examples can be analyzed using less sophisticated approaches tailored to the specific problem, this section illustrates the wide applicability of our framework.

\subsection{\emph{M}-Estimation}
\label{sec:MEst}

\subsubsection{Generalized Linear Models}
\label{sec:MEstGen}

Let $(Y_1, \bX_1), \dots, (Y_n, \bX_n) \in \R^{1 + p_n}$ be \emph{iid} and consider an estimator $\hbtheta$ that satisfies
\begin{equation}
\label{eq:MEstGLM}
\sumin \phi_{i}(\hbtheta) =  \sumin \psi( Y_{i}, \bX_{i}^\top \hbtheta) \bX_{i} = \0
\end{equation}
with some function $\psi$.
This includes the least squares estimator and other $M$-estimators in the linear model studied by \citet{Huber73,  PortnoyCons, PortnoyAsymp} and \citet{Mammen89} as well as likelihood inference in generalized linear models as special cases.
If the function $\psi$ is differentiable with nonpositive derivative, the matrix $H_n$ in \ref{A:Cons1}  can be constructed as $H_n(\bx) = - \inf_{\btheta \in \Theta_n}|\psi'(Y_{i}, \bx^\top \btheta)| \bx  \bx^\top$, where $\psi'(Y_{i}, \eta) = \frac{\partial}{\partial \eta}\psi(Y_{i}, \eta)$.

\begin{corollary} \label[corollary]{cor:Mest} 
Let $\hbtheta$ solve \eqref{eq:MEstGLM} and let $\psi'$ be nonpositive, Lipschitz in $\eta$ and uniformly bounded.
Suppose that $\max_k \E[\phi_{i}(\btheta^*)_k^4] = O(1)$ and, with some $k \ge 2$, for all $\ba \in \R^{p_n}$,
\begin{align} \label{eq:design-cond}
  \begin{aligned}
    &\|\E[\bX \bX^\top]\| = O(1), \quad \E[|\ba^\top \bX|^{2k}] = O(\|\ba\|^{2k}) , \quad \E[\rho(|\| \bX \| - \E[\| \bX \|]|^2)]= O(1)
  \end{aligned}
\end{align}
with some increasing continuous function $\rho \colon (0, \infty) \to (0, \infty)$ such that $\rho(x)/x$ is increasing.
Suppose further that there are $c> 0$ and sets $\Theta_n\supset \{ \btheta : \| \btheta - \btheta^* \| \le r_n C \}$ for all $C < \infty$ and $n$ large,
\begin{align} \label{eq:Mest-eigcond}
  \lambda_{\min}\lf(\E\lf[\inf_{\btheta \in \Theta_n}|\psi'(Y_{i}, \bX^\top \btheta)| \bX  \bX^\top\ri] \ri) \ge c.
\end{align}
\begin{enumerate}[label=(\roman*)]
  \item If  $p_n \ln p_n / n \to 0$ and $\rho^{-1}(n) \ln p_n / n \to 0$, $\hbtheta$ is a $\sqrt{n/p_n}$-consistent estimator of $\btheta^*$ and unique on $\Theta_n$ with probability tending to 1.
  \item  If the matrix $A_n$ in \cref{theorem2} satisfies $\| A_n \| = O(1)$, $p_n^2/n \to 0$ and $p_n^{3k}/n^{(2k - 2)} \to 0$,
  $\hbtheta$ is asymptotically normal. 
\end{enumerate}
\end{corollary}
The design conditions \eqref{eq:design-cond} are relatively mild.
The second requires a weak form of isotropy and the third requires some form of concentration. 
All conditions are easily satisfied for, e.g., independent, uniformly sub-Gaussian variables, for which $\rho(x) = \exp(c x)-1$ with some sufficiently small $c>0$ and arbitrarily large $k$ work, and the growth bound for asymptotic normality becomes $p_n^2/ n \to 0$.
In contrast to the results of \citet{PortnoyAsymp}, the corollary also applies to covariates with heavier tails. For example, taking $k = 2$ (which gives a fourth moment condition), the estimator is asymptotically normal as long as $p_n^3 / n  \to 0$. 
For the least squares estimator, $\psi'$ is constant, so the eigenvalue condition \eqref{eq:Mest-eigcond} simplifies to $\lambda_{\min}\lf(\E\lf[\bX  \bX^\top\ri] \ri) \ge c$.
This simplification applies  generally if $\psi'$ is uniformly bounded away from zero.

To provide an example of the extension of our results to dependent data, we revisit the time series setting from \cref{ex:timeseries}:
Let $(\bZ_1, Y_1), \ldots, (\bZ_N, Y_N) \in \R^{p_n + 1}$ be a stationary, polynomially $\beta$-mixing time series and consider an estimator $\hbtheta$ that satisfies $\sumin \phi_i(\hbtheta) = \bnull$ with
\begin{equation}\label{eq:phiTS}
   \phi_i(\btheta) \coloneqq \frac{1}{m_N} \sum_{j \in I_{i,N}} \varphi(Y_j, \bZ_j; \btheta), \quad \varphi(Y_j, \bZ_j; \btheta) \coloneq \psi(Y_j, \bZ_j^\top \btheta)\bZ_j.
\end{equation}

\begin{corollary}\label{cor:TS}
  Let $\hbtheta$ solve $\sumin \phi_i(\hbtheta) = \bnull$ with $\phi_i(\btheta)$ as defined in \eqref{eq:phiTS} and let $\psi'(Y_j, \eta) = \frac{\partial}{\partial \eta}\psi(Y_j, \eta)$ be nonpositive, Lipschitz in $\eta$ and uniformly bounded.
  Suppose that $\max_k \E[| \varphi(Y_j, \bZ_j; \btheta^*)_k|^{4 + \eps}] = O(1)$ with some $\eps > 0$ and $\kappa > 2(4 + \eps )/\eps$ ($\kappa$ is defined in \cref{ex:timeseries}), and \eqref{eq:Mest-eigcond} and the first two conditions in \eqref{eq:design-cond} hold with $\bX$ replaced by $\bZ$.

  For consistency and uniqueness, let $b_n$ be a sequence such that
  \[
    p_n / b_n = O(1), \qquad \limsup_{N \to \infty} \max_{1 \le j \le N} \| \bZ_j \|^2 / b_n \le 1/2 \quad \text{almost surely}.
  \]
  If $b_n/N^{(1 - 1/(1 + \kappa) - \tilde \eps)} \to 0$ with $\tilde \eps > 0$ arbitrarily close to 0, $\hbtheta$ is a $\sqrt{N/p_n}$-consistent estimator of $\btheta^*$ and unique on $\Theta_n$ with probability tending to 1.

  For asymptotic normality, let $A_n$ be a matrix such that $\| A_n \|_\infty = O(1)$.
  Suppose that $\Sigma = \lim_{n \to \infty} m_N A_n I(\btheta^*) A_n^T$ exists.
  Suppose that $\kappa > 5$ and, for some $\tau > 4$ and all $\ba \in \R^{p_n}$,
  \[
    \E[| \ba^\top \bZ |^{\tau}] = O(\| \ba \|^{\tau}), \qquad \kappa > 3 \tau/(\tau - 4).
  \]
  If $p_n^{3r}/N^{(1.5 r - 1/2)}  = o(1)$ with $r < \kappa \tau /(\tau + 4 \kappa)$, then $\sqrt{N} A_n J(\btheta^*) (\hbtheta - \btheta^*) \to_d \Ncal(\bnull, \Sigma)$.
\end{corollary}
We illustrate the conditions with two limiting cases.
For consistency, the moment and mixing requirements trade off against one another: if the moment condition holds for arbitrarily large $\eps$, we only require $\kappa > 2$, whereas $\kappa \to \infty$ permits $\eps$ to be arbitrarily close to zero.
If the entries of $\bZ$ are independent and sub-Gaussian and $p_n$ grows at least polynomially, sub-Gaussian concentration and the Borel--Cantelli lemma yield $b_n \lesssim p_n$.
As $\kappa$ increases, the resulting dimensional restriction approaches $p_n/N^{1 - \delta} \to 0$ for arbitrarily small $\delta > 0$.

For asymptotic normality, if $\kappa \to \infty$, bounded $(4 + \delta)$th moments for arbitrarily small $\delta > 0$ and $p_n^3/N \to 0$ are sufficient, in line with \cref{cor:Mest}.
Conversely, if the moment bound for $\ba^\top \bZ$ holds for arbitrarily large $\tau$, the restriction on $r$ reduces to $r < \kappa$.
Thus, stronger mixing moves the dimensional restriction towards the nearly optimal condition $p_n^2/N \to 0$ from \cref{cor:Mest}.

\subsubsection{Penalization in GLMs}

We now return to the \emph{iid} setting and consider the Lasso penalty $p_{\lambda_n}(\btheta) = \lambda_n \| \btheta \|_1$.
\begin{corollary} \label[corollary]{cor:MestP} 
Suppose the reduced problem \eqref{eq:Def_M_reduced} satisfies the conditions from \cref{cor:Mest} with some $k$ and some function $\rho$ and  
\begin{align} \label{eq:design-cond-pen}
 \begin{aligned}
  \max_{1 \le k \le p_n} \E[\rho(X_{k}^2\psi(Y, \bX^\top \btheta^*)^2)] &= O(1), \quad \max_{1 \le  k \le p_n}\E[\rho( X_{k}^2)] = O(1).
 \end{aligned} 
\end{align}
Denote $\bar  \psi (Y, \bX, \btheta)= \int_0^1 \psi'(Y, \bX^\top [\btheta^* + t(\btheta - \btheta^*)]) dt$ and let
\begin{align} \label{eq:Mest-mutual-incoherence}
\sup_{\btheta \in \Theta'_n}  \lf\|  \E\left[\bar  \psi (Y, \bX, \btheta) \bX_{(2)} \bX_{(1)}^\top \right] \, \E\left[\bar  \psi (Y, \bX, \btheta)  \bX_{(1)} \bX_{(1)}^\top \right]^{-1} \ri\|_{\infty} \le \alpha
\end{align}
for some $\alpha \in [0, 1)$, suppose that $\sigma^2 = \max_{1 \le k \le p_n}\E[ \phi_{i}(\btheta^*)_k^2]$ is bounded away from zero and infinity, $\min_{k : \theta^*_k \neq 0 } | \theta_{k}^*|/ \tilde r_n \to \infty$ with $\tilde r_n = \sqrt{s_n \ln p_n / n}$ and $\lambda_n \ge 8 (1 - \alpha)^{-1} \sqrt{\sigma^2 \ln p_n/n}$ and $\lambda_n = O(\sqrt{\ln p_n / n})$.
Then, if 
\begin{align} \label{eq:MestP-dim-cond}
  s_n \ln p_n = o\lf(\sqrt{n}\ri), \qquad p_n = o\lf(\frac{\rho(n / (s_n \ln p_n)^2 )}{n}\ri), \qquad \frac{s_n^{3k} (\ln p_n)^k}{n^{2k - 2}} = o(1),
\end{align}
the Lasso-penalized equation \eqref{eq:MEstGLM} has a solution $\hbtheta$ with $\hbtheta_{(2)} = \0$ with probability tending to 1, $\| \hbtheta - \btheta^* \| = O_p(\sqrt{s_n \ln p_n/n})$, and $\hbtheta_{(1)}$ is asymptotically normal (with $\| A_n \| = O(1)$).

If $|\psi'|$ is uniformly bounded away from $0$ and $\lambda_{\min}\lf(\E\lf[\bX  \bX^\top\ri] \ri) \ge c$ for some $c > 0$,
this solution is unique with probability tending to 1.
\end{corollary}
The moment conditions in \eqref{eq:design-cond} and \eqref{eq:design-cond-pen} constrain the growth of $p_n$ through the function $\rho$. Assuming $s_n = O(1)$ for simplicity, the choice $\rho(x) = \exp(c x)-1$ for some $c>0$ allows $p_n \sim \exp(n^{1/3 - \eps})$ for any $\eps > 0$. Similarly, polynomial moment bounds translate into polynomial growth conditions on $p_n$. For example, the choice $\rho(x) = x^3$ requires sixth moments and allows $p_n \sim n^{2 - \eps}$.

If $\btheta$ can be partitioned into $K_n$ groups $G_1, \ldots, G_{K_n} \subseteq \{ 1, \ldots, p_n\}$, we can apply a group-structured penalty such as the Group SCAD $p_{\lambda_n} = \sum_{g=1}^{K_n} p_{\lambda_n}(\| \btheta_{G_g} \|_2)$, where $p_{\lambda_n}(\theta)$ denotes the one-dimensional SCAD penalty.
This contains the SCAD penalty $p_{\lambda_n} = \sum_{k=1}^{p_n} p_{\lambda_n}(|\theta_k|)$ as a special case.

\begin{corollary} \label[corollary]{cor:MestSCAD} 
  Suppose the reduced problem \eqref{eq:Def_M_reduced} satisfies the conditions from \cref{cor:Mest} with some $k$ and some function $\rho$ and \eqref{eq:design-cond-pen} holds and let 
  \begin{align} \label{eq:Mest-mutual-incoherence-SCAD}
\alpha_n \coloneq \max_{g \in I_{(2)}} \sup_{\btheta \in \Theta'_n, \| \bx \|_\infty \le 1}  \lf\|  \E\Big[\bar  \psi (Y, \bX, \btheta) \bX_{G_g} \bX_{(1)}^\top \Big] \, \E\Big[\bar  \psi (Y, \bX, \btheta)  \bX_{(1)} \bX_{(1)}^\top \Big]^{-1} \bx \ri\|_2 .
\end{align}
Suppose that $\sigma^2 = \max_{1 \le k \le p_n}\E[ \phi_{i}(\btheta^*)_k^2]$ is bounded away from zero and infinity, $(\min_{g \in I_{(1)}} \| \btheta^*_{G_g} \|_2 - \sqrt{s_n / n})/\lambda_n \to \infty$, $\sqrt{n/s_n} \min_{g \in I_{(1)}} \| \btheta^*_{G_g} \|_2 \to \infty$ and $\lambda_n \ge 4 \eta_n \max\left\{  \sqrt{| G_1 |}, \ldots , \sqrt{| G_{K_n} |}, \alpha_n \right\}$ with some $\eta_n \ge \sqrt{4 \sigma_n^2 \ln p_n / n}$.
Then, if $s_n = o(\sqrt{n})$, $s_n^{3k}/n^{(2k - 2)} = o(1)$, and
\[
 \frac{s_n(s_n + \ln p_n)}{n^2 \eta_n^2} = o(1), \quad  p_n = o\left( \frac{\max\{ \rho(n/(s_n \ln p_n)) , \ \rho(n^{3/2} \eta_n/(s_n^2 + s_n \ln p_n))\}}{n} \right),
\]
the Group SCAD-penalized equation has a solution $\hbtheta$ with $\hbtheta_{(2)} = \0$ with probability tending to 1, $\| \hbtheta - \btheta^* \| = O_p(\sqrt{s_n/n} )$, and $\hbtheta_{(1)}$ is asymptotically normal (with $\| A_n \| = O(1)$) and as efficient as the oracle solution.
If $\lambda_{\min}(\E[\inf_{\btheta \in \Theta_n} | \psi'(Y, \bX^\top \btheta)| \bX  \bX^\top]) \ge c$ for some $c > 0$ and the Group SCAD penalty is used with $a > 1 + \frac 1 c$, this solution is unique with probability tending to 1.
\end{corollary}
Compared to the Lasso, we do not require the mutual incoherence condition \eqref{eq:Mest-mutual-incoherence}, and $\alpha_n$ in \eqref{eq:Mest-mutual-incoherence-SCAD} may even diverge.
To ensure uniqueness, the non-convexity parameter $a$ of SCAD must be chosen appropriately.

The four corollaries also apply to non-parametric regression problems, in which $\bX_{i}$ consists of appropriate basis functions.
For example, \cref{cor:MestP} implies consistency of Lasso-assisted variable selection in high-dimensional nonparametric additive models \citep{huang2010} under appropriate conditions.

\subsection{A problem where the classical RSC condition fails}\label{sec:RSC-failure}

We now give a simple example where the RSC condition \eqref{eq:RSC-LW} fails, while our weaker condition remains easy to verify.
Consider an allocation problem with linear payoffs and quadratic cost of the form $\max_{\btheta} \bmu^\top \btheta - \frac 1 2 \btheta^\top C \btheta$, where $\bmu \in \R^{p_n}$ is a vector of expected payoffs and $C \in \R^{p_n \times p_n}$ is a symmetric, positive definite cost matrix.
Such problems arise, for example, in portfolio optimization, where $\bmu$ contains expected asset returns and $C$ reflects market risk and transaction or liquidity costs, or in marketing budget allocation, where $\bmu$ contains expected returns of marketing channels and $C$ captures diminishing returns and cross-channel effects such as cannibalization.

For the counterexample, we consider centered payoffs and set $C=I_{p_n}$, so the population solution is $\btheta^*=\0$.
Let $m_n = \lceil\sqrt n\rceil$, $M_n = \lfloor p_n/m_n\rfloor$, and let $S_1,\ldots,S_{M_n}$ be disjoint subsets of $\{1, \ldots, p_n\}$ of size $m_n$.
For $S\subseteq\{1,\ldots,p_n\}$, let $\bm 1_S=(\ind\{j\in S\})_{j=1}^{p_n}\in\R^{p_n}$ denote its indicator vector.
Suppose that $\bY_i$ has independent standard Gaussian entries and define the symmetric random cost matrices
$X_i = I_{p_n} + \sum_{g=1}^{M_n}\xi_{i,g} \bm 1_{S_g}\bm 1_{S_g}^\top$,
  where $\xi_{i,g}\stackrel{\mathrm{iid}}{\sim}\Ncal(0,1),$
and $\xi_{i,g}$ and $\bY_i$ are mutually independent.
Thus, $\E[X_i]=I_{p_n}$, but the entries within each block $S_g\times S_g$ are subject to a common random shock.
Consider the compact parameter set $\Theta_n = \{\btheta\in\R^{p_n}\colon \|\btheta\|_2 \le 1\}$ and the penalized problem
\begin{align} \label{eq:RSC-failure-model}
  \hbtheta \in \argmax_{\btheta\in\Theta_n}  \left\{\frac{1}{n}\sumin \bY_i^\top\btheta - \frac 1 2 \btheta^\top\left(\frac1n\sumin X_i\right)\btheta - \lambda_n\|\btheta\|_1\right\}.
\end{align}
We consider the associated penalized estimating equation
\begin{align} \label{eq:RSC-failure-equation}
  \Phi_n(\btheta)  =\frac 1 n \sumin (\bY_i - X_i\btheta) \in \lambda_n \partial \|\btheta\|_1.
\end{align}
The random part of each sample Hessian has rank $M_n$, which diverges under the conditions below.

\begin{corollary}
  \label{cor:RSC-failure}
  Consider the setting above with $p_n\ge n$, $\eta_n=\sqrt{\ln p_n/n}$, $\ln p_n = o(\sqrt n)$, and $\lambda_n = C_\lambda\eta_n$ for a sufficiently large constant $C_\lambda$.
  Let $1\le\nu_n = O((n/\ln p_n)^{1/2-\delta})$ for some $\delta\in(0, 1/2]$.
  Then, with probability tending to one, $\0$ is the unique solution to \eqref{eq:RSC-failure-equation} on $\Theta^*(\nu_n).$
\end{corollary}
On the other hand, the RSC condition \eqref{eq:RSC-LW} can only hold with vanishing probability.
\begin{proposition} \label[proposition]{lem:RSC-failure}
  Let $\eta_n = \sqrt{\ln p_n/n}$, $p_n \ge n$, and $\ln p_n = o(\sqrt n)$.
  Then, for any $c, c_1 > 0$, the event
  \begin{align*}
    \left\langle \bu, \Phi_n(\0) - \Phi_n(\bu) \right\rangle
    \ge c\|\bu\|_2^2 - c_1 \|\bu\|_1^2 \eta_n^2,
    \qquad \forall \|\bu\|_1 \le 1,
  \end{align*}
  has probability tending to zero.
\end{proposition}
Consequently, the theory of, e.g., \citet{WWpaper, LohWW, WWbook} does not apply.
The problematic directions have effective sparsity $m_n \asymp \sqrt n \gg  \nu_n$ and thus lie outside the cones covered by \cref{cor:RSC-failure}, showing that the rank-one structure exploited in proofs of \eqref{eq:RSC-LW} is consequential.
 \subsection{Multi-Sample Estimation}
\label{sec:MSample}

Consider a multi-sample estimation problem: 
The data are given by $(k_{1}, \bX_{1}), \dots, (k_{n}, \bX_{n})$, where $k_{i} \in \{ 1, \ldots, K_n\}$ indicates to which of the $K_n$ samples $\bX_{i}$ belongs. 
Assume for simplicity that $p_n = K_n$ and that each $\theta_k$ is estimated using only the $k$-th sample.
Write $n_{k} = \sumin \ind\{k_{i} = k\}$ for the sample size of the $k$-th sample and $\phi'_k(\bX; \theta_k)= \partial \phi_k(\bX; \theta_k)/\partial \theta_k$.
The estimation function, its Jacobian, and covariance matrix are
\begin{align*}
\phi(k_{i}, \bX_{i}; \btheta)
  &= \left(\ind\{k_{i} = k\}\, \frac{n}{n_k} \phi_k(\bX_{i}; \theta_k) \right)_{k = 1}^{K_n},\\
\nabla_{\btheta} \phi(k_{i}, \bX_{i}; \btheta)
  &= \mathrm{diag} \lf(\ind\{k_{i} = k\}\, \frac{n}{n_k} \phi'_k(\bX_{i}; \theta_k)\ri)_{k = 1}^{K_n}, \quad 
I(\btheta^*)
  = \mathrm{diag}\lf(\frac{n}{n_k} \E[\phi_k(\bX_{i}; \theta^*_k)^2] \ri)_{k = 1}^{K_n}.
\end{align*}
The standardization $n/n_k$ is necessary to ensure that the eigenvalues of $n^{-1} \sumin \E[H_n(\bX_i)]$ are bounded away from $0$.
Consequently, $r_n = O ((\sum_{k=1}^{K_n} n_k^{-1})^{1/2})$ if $\E[\phi_k(\bX_{i}; \theta^*_k)^2]$ is uniformly bounded over $i$ and $k$. This framework can easily be extended to multiple parameters $\btheta_k \in \R^{p_k}$, potentially shared across subsamples. In the following examples, we stick to the single-parameter case for simplicity.

\subsubsection{Example: Distributed Inference}

An interesting application arises in distributed inference. Here, \emph{iid} data are distributed over $K_n$ different locations, and the goal is to estimate a parameter $\theta^* \in \R$ from the distributed data. This is a common setup in federated learning, where data are distributed over different devices, and the goal is to estimate a common model.
Our general setup allows for differing sample sizes and population characteristics between locations. This may happen if, for example, the data is collected over hospitals who may share their estimate but not the data for privacy reasons. The distributed estimates $\hat \theta_{k}$ can be reconciled into a global estimate through averaging $\hat \theta_{ K_n + 1} = K_n^{-1} \sum_{k = 1}^{K_n} \hat \theta_{k}$. To put this in our framework, we stack the individual estimating equations $\phi_k(\bX_{i}; \theta_k) = \psi(\bX_{i}; \theta_k)$ as above, and append the reconciliation function $\phi_{K_n + 1}(\bX_{i}; \theta) = K_n^{-1} \sum_{k = 1}^{K_n} \theta_{k} - \theta_{ K_n + 1}$.

\begin{corollary}\label[corollary]{cor:distributed}
  Let $\Theta_0 \subseteq \R$ and $n_1 = \dots = n_{K_n} = n / K_n$. Suppose that, for $\theta \in \Theta_0$, $\psi(\bx; \theta)$ is uniformly bounded and $\psi'(\bx; \theta) = \partial_{\theta} \psi(\bx; \theta)$ is negative, uniformly bounded away from 0 and $-\infty$ and Lipschitz in $\theta$.
Then, if $K_n^3/n \to 0$, with probability tending to one there is a unique solution $\hbtheta$. Moreover, for all $k = 1, \ldots, K_n$, $\sqrt{n/K_n}(\hat{\theta}_k - \theta_k^*) \to N(0, \sigma_k^2)$ and $\sqrt{n}(\htheta_{ K_n + 1} - \theta^*) \to N(0, \sigma^2)$ with
  \begin{align*}
    \sigma_k^2 = \frac{ \E[\psi(\bX_{i}; \theta^*_k)^2]}{\E[\psi'(\bX_{i}; \theta^*_k)]^2}, \quad \sigma^2 =  \lim_{n \to \infty} \frac{1}{K_n} \sum_{k = 1}^{K_n}  \sigma_k^2.
  \end{align*}
\end{corollary}
While for fixed $K_n=K$, asymptotic normality of $\htheta_{K_n + 1}$ would follow easily from   asymptotic normality of the independent $\htheta_k$, our results provide conditions on the rate of growth of $K_n$ when $K_n \to \infty$.

If the samples have identical distributions, it holds that $\sigma_1 = \cdots = \sigma_{K_n}$, so the averaged estimator is as efficient as one computed from the pooled sample. 
For a fixed number of samples $K$ and in the risk minimization setting, this is in line with the results in \citet{Zhang2013}, who laid the groundwork for further developments in this area, including a diverging dimension of the target parameter \citep{Rosenblatt2016}, a diverging number of samples \citep{Huang2019} and penalized estimation \citep{Battey2018}.
Unlike the references just mentioned, our framework allows for estimation problems beyond risk minimization/optimization and does not require the different samples to be identically distributed.

The independence assumption might be unrealistic, so we now assume dependence induced by a network as described in \cref{ex:network},  and 
\[
 \phi(\bk_i, \bX_i ; \btheta) \coloneqq \frac{1}{m_N} \sum_{j \in I_{i,N}} \left(\ind\{ k_j = k \} \, \frac{N}{N_k} \psi_k(\bZ_j; \theta_k) \right)_{k=1, \ldots, K_n}, \quad \text{where } N_k = \sum_{j=1}^N \ind\{ k_j = k \}.
\]
Note that now, $\nabla \phi(\bk_i, \bX_i ; \btheta)$ and $I(\btheta)$ are not necessarily diagonal matrices, as we allow $\bX_i$ to contain observations from different locations.
Again, we append the reconciliation function $\phi_{K_n + 1}(\bX_{i}; \theta) =K_n^{-1} \sum_{k = 1}^{K_n} \theta_{k} - \theta_{ K_n + 1}$.
For simplicity, we assume that at each location $k$, the marginal distributions of the $\bZ_j$ with $k_j = k$ are identical.

\begin{corollary}\label{cor:network}
  Let $\Theta_0 \subseteq \R$ and $N_1 = \dots = N_{K_n} = N / K_n$. Suppose that, for $\theta \in \Theta_0$, $\psi(\bz; \theta)$ is uniformly bounded and $\psi'(\bz; \theta) = \partial_{\theta} \psi(\bz; \theta)$ is negative, uniformly bounded away from 0 and $-\infty$ and Lipschitz in $\theta$, and that the conditions of \cref{ex:network} hold. Suppose further that \eqref{eq:cov-decay} holds for the projection matrices associated with the asymptotic normality statements below.
  Let $d_n  \coloneqq m_N K_n^{-2} \sum_{k=1}^{K_n} I(\btheta^*)_{kk}$,  assume the common growth conditions $K_n^3/n = o(1)$, $K_n N (n - L_N)^2/(d_n n^2) =o(1)$, $(n- L _N)^5/n = o(1)$, and define
  \begin{align*}
  \sigma_n^2 &\coloneq \frac{1}{N} \sum_{j=1}^N \sum_{j'=1}^N  \frac{ \E[\psi(\bZ_j; \theta^*_{k_j}) \psi(\bZ_{j'}; \theta^*_{k_{j'}}) ]}{\E[\psi'(\bZ_j; \theta_{k_j}^*)]\E[\psi'(\bZ_{j'}; \theta_{k_{j'}}^*)]}, \\
   \sigma_{n,k}^2 & \coloneqq \frac{K_n}{N} \sum_{j=1}^N \sum_{j'=1}^N \ind\{ k_j = k_{j'} = k \} \frac{\E[\psi(\bZ_j; \theta_k^*) \psi(\bZ_{j'}; \theta_k^*)]}{\E[\psi'(\bZ_j; \theta^*_k)]^2}.
  \end{align*}

  For the global estimator, suppose additionally that $N(n-L_N)^2/(n^2 \sigma_n^2) = O(1)$ and
  \begin{align*}
   \max_{1 \le i \le n} \E\Big[ \big( (\sigma_n \sqrt{m_N})^{-1} \sum_{j \in I_{i,N}} \psi(\bZ_j; \theta_{k_j}^*) \big)^4 \Big]  = o(n),
  \end{align*}
  then $\sqrt{N}\, \sigma_n^{-1} (\htheta_{ K_n + 1} - \theta^*) \to N(0, 1)$.
  For a given location $k = 1, \ldots, K_n$, suppose instead that $K_n^3 d_n/(n \sigma_{n,k}^2) = o(1)$, $K_nN(n-L_N)^2/(n^2 \sigma_{n,k}^2) = O(1)$, $d_n K_n^{3/2}/(\sqrt{N} \sigma_{n,k})  = o(1)$ and
    \begin{align*}
   \max_{1 \le i \le n} \E\Big[ \big( (\sigma_{n,k} \sqrt{m_N})^{-1} \sum_{j \in I_{i,N}} \ind\{k_j = k \} \psi(\bZ_j; \theta_k^*) \big)^4 \Big]  = o(n/K_n^2),
  \end{align*}
  then $\sqrt{N/K_n}\, \sigma_{n,k}^{-1} (\htheta_k - \theta^*_k) \to N(0, 1)$.
\end{corollary}
Compared to the \emph{iid} setting in \cref{cor:distributed}, the conditions are more complex since we impose very few assumptions on the dependence structure, which determines $n, L_N, d_n, \sigma_n$ and $\sigma_{n,k}$. 
Examining the best- and worst-case behavior of $d_n$, $\sigma_n$, and $\sigma_{n,k}$ gives a clearer picture:
If the dependence within the $\bX_i$ is sufficiently weak, then $d_n = O(1)$ and $\sigma_n^2 \simeq 1, \sigma_{n,k}^2  \simeq 1$, and we obtain asymptotic normality of $\htheta_{ K_n + 1}$ and $\htheta_k$ at $\sqrt{N}$ and $\sqrt{N/K_n}$ rate.
The fourth moment condition is plausible since $\E[(m_N^{-1/2} \sum_{j \in I_{i,N}} \psi(\bZ_j; \theta_{k_j}^*) )^4] = O(1)$ if the dependence is weak enough.
Among the additional conditions for asymptotic normality of $\htheta_k$, only $K_nN(n-L_N)^2/(n^2 \sigma_{n,k}^2) = O(1)$ remains non-trivial.
If the dependence within the $\bX_i$ is strong, we have $d_n = O(m_N)$ and $\sigma_n^2, \sigma_{n,k}^2 \simeq m_N$.
The slower rates ($\sqrt{n}$ and $\sqrt{n/K_n}$) require fewer assumptions:
The fourth moment conditions become trivial, since the expectations are bounded, and all remaining conditions are implied by $K_n^3/n = o(1)$ and $(n-L_N)^5/n = o(1)$.

\subsection{Stepwise Estimation}
\label{sec:Stepw}

Another setting that shows the flexibility of our results is stepwise estimation. Assume the parameter vector can be grouped  as $\btheta = (\btheta_1, \dots, \btheta_{K_n})$. The parameters are estimated sequentially using the estimates $\hbtheta_{1}, \ldots, \hbtheta_{k - 1}$ from previous iterations: 
\[
\hbtheta_k = \argmax_{\theta_k} \sumin f_k(\bX_{i}; \btheta_k, \hbtheta_1, \ldots, \hbtheta_{k-1}), \quad \btheta_k^* = \argmax_{\btheta_k} \E\lf[ f_k(\bX; \btheta_k,\btheta^*_1, \ldots, \btheta^*_{k-1}) \ri],
\]
for some functions $f_k$.
Denote $\phi_k(\bX_{i}; \btheta_k, \btheta_1, \ldots, \btheta_{k-1}) = \nabla_{\btheta_k}f_k(\bX_{i}; \btheta_k, \btheta_1, \ldots, \btheta_{k-1})$.
Then, the sequential estimator $\hbtheta$ can be expressed as the solution of $ \sumin \phi(\bX_{i};\hbtheta)  = \0$ and $\btheta^*$ is the solution of $\E[ \phi(\bX; \btheta^*) ] = \0$ with
\[
 \phi(\bX_{i}; \btheta)  = \left(
 \phi_1(\bX_{i}; \btheta_1) ,\ \phi_2(\bX_{i}; \btheta_2, \btheta_1) , \ldots, \ \phi_{K_n}(\bX_{i}; \btheta_{K_n}, \btheta_1, \ldots, \btheta_{K_n-1}) \right)^\top.
\]
By treating the entire estimation path as a single parameter vector, this representation allows our general results to accommodate a diverging number of steps $K_n$.

\subsubsection{Example: Causal Inference}

As a concrete example, suppose we want to estimate the causal effect of some covariates $\bZ$ on an outcome $Y$ in the presence of confounders $\bC$ from \emph{iid} observational data. Part of the population has received a treatment, which we indicate by the binary treatment indicator $T$. Under the usual conditions for no unmeasured confounding, the \emph{conditional average treatment effect (CATE)} can be defined as
\begin{align*}
  \mathrm{CATE}(\bz) = \E\lf[\frac{YT}{\Pr(T = 1 \mid \bW)} - \frac{Y(1 - T)}{\Pr(T = 0 \mid \bW)} \mid \bZ = \bz\ri],
\end{align*}
where $\bW = (\bZ, \bC)$; see e.g.~\citet{causal}.
We model the treatment probabilities and CATE by $\Pr(T = 1 \mid \bw) = \sigma( \bw^\top \btheta_1)$ and $\mathrm{CATE}(\bz) = \bz^\top \btheta_2$, where $\sigma$ is an appropriate link function.
The parameters can be estimated by first estimating $\btheta_1$ using maximum likelihood, and then estimating $\btheta_2$ by the plug-in least squares estimator
\begin{align*}
  \hbtheta_2 = \arg\min_{\btheta_2} \sum_{i = 1}^n \lf[\frac{Y_{i} T_{i}}{\sigma(\bW_{i}^\top \hbtheta_1)} - \frac{Y_{i}(1 - T_{i})}{1 - \sigma(\bW_{i}^\top \hbtheta_1)} - \bZ_{i}^\top \btheta_2\ri]^2.
\end{align*}
This stepwise procedure can be reformulated as solving the estimating equation
\begin{align*}
  \sumin\phi(Y_{i}, T_{i}, \bZ_{i}, \bW_{i}; \btheta) & = \sumin \begin{pmatrix}
    \nabla_{\btheta_1} \lf[T_{i} \ln \sigma(\bW_{i}^\top \btheta_1) + (1 - T_{i}) \ln[1 - \sigma(\bW_{i}^\top \btheta_1)]\ri]\\
    -\lf[\frac{Y_{i} T_{i}}{\sigma(\bW_{i}^\top \btheta_1)} - \frac{Y_{i}(1 - T_{i})}{1 - \sigma(\bW_{i}^\top \btheta_1)} - \bZ_{i}^\top \btheta_2\ri] \bZ_{i}
  \end{pmatrix} = \0 .
\end{align*}

\begin{corollary} \label[corollary]{cor:IPW}
  Suppose that, for some $\varepsilon > 0$,
    $\varepsilon \le \sigma(\bW^\top \btheta_1) \le 1-\varepsilon$
  almost surely, uniformly over $\btheta \in \Theta_n$, and that $\sigma$ is twice continuously differentiable with uniformly bounded derivatives.
  Suppose further that $|Y| \le 1$, $\max_k\E[\phi(Y, t, \bZ, \bW; \btheta^*)^4_k] = O(1)$ and $\bW \in \R^{p_n}$ satisfies the design conditions from \eqref{eq:design-cond}.
  Let $\bar \sigma = 1 - \sigma$ and define
  \begin{align*}
    \alpha_1(T, \bW) &= \sup_{\btheta \in \Theta_n} \lf[T(\ln \sigma)''(\bW^\top \btheta_1) + (1 - T) (\ln \bar \sigma)''(\bW^\top \btheta_1)  \ri], \\
    \alpha_2(T, Y, \bW) &= \sup_{\btheta \in \Theta_n} \lf|\frac{TY \sigma'(\bW^\top \btheta_1)}{2\sigma(\bW^\top \btheta_1)^2}\ri| + \lf|\frac{(1 - T)Y \sigma'(\bW^\top \btheta_1)}{2 \bar \sigma(\bW^\top \btheta_1)^2}\ri| - 1.
  \end{align*}
  Suppose there is $c > 0$ such that $\max\{\lambda_{\max}\lf(\E\lf[\alpha_1(T, \bW) \bW \bW^\top  \ri] \ri), \lambda_{\max}\lf(\E\lf[\alpha_2(T, Y, \bW) \bZ \bZ^\top  \ri] \ri) \} \le -c$.

 Then, if $p_n \ln p_n / n \to 0$ and $\rho^{-1}(n)\ln p_n/n \to 0$ (with $\rho$ as defined in \eqref{eq:design-cond}), the estimating equation has a unique solution $\hbtheta$ on $\Theta_n$ with $\| \hbtheta - \btheta^*\| = O_p(\sqrt{p_n / n})$.

  If $ p_n^2  / n \to 0$ and $p_n^{3/2 \, k}/(n^{2 k - 2}) \to 0$ (with $k$ as defined in \eqref{eq:design-cond}), $\hbtheta$ is asymptotically normal.
\end{corollary}
The matrices inside the expectations of the eigenvalue condition are the blocks of a block-diagonal matrix $H_n$ constructed in the proof. The eigenvalue condition is easiest to verify if both $\ln \sigma$ and $\ln(1-\sigma)$ are concave, as for logistic and probit links.
For these links, the bounded-away condition above is a uniform-overlap assumption on the realized linear predictors; for example, it holds when $\bW^\top \btheta_1$ remains in a fixed compact set uniformly over $\btheta \in \Theta_n$.

If the number of covariates or confounders is large, we may want to add a sparsity penalty $p_{\lambda_n}(\btheta)$.
For simplicity, suppose that the parameters are reordered such that $\btheta^* = (\btheta_{(1)}^*, \0)$.
The following corollary guarantees that the SCAD-penalized estimator $\hbtheta$ is consistent, unique and asymptotically normal.
\begin{corollary}
\label[corollary]{cor:causalSCAD}
  Suppose that the regularity conditions from \cref{cor:IPW} hold, $\bW$ satisfies the design conditions \eqref{eq:design-cond-pen} (with $|\psi|$ uniformly bounded), $\sqrt{n / s_n}\min_{1 \le k \le s_n} |\theta_{k}^*|  \to \infty$, and that the SCAD penalty is used with $a  > 1 +  \frac{1}{\min\{ c, K^2\}}$, with $c$ from \cref{cor:IPW} and $K \in (0, \infty)$ such that
  \begin{align*}
    |Y \sigma'(\bW^\top \btheta_1) /( 2  \sigma(\bW^\top \btheta_1)^2)| \le K , \quad |Y \sigma'(\bW^\top \btheta_1) / (2 \bar \sigma(\bW^\top \btheta_1)^2)| \le K, \quad \|\E[\bW \bW^\top]\| \le K.
  \end{align*} 
  Suppose that $(\min_{1 \le k \le s_n} |\theta_k^* | - \sqrt{s_n / n})/\lambda_n \to \infty$, $\sqrt{n/s_n} \min_{1 \le k \le s_n} |\theta_k^* | \to \infty$, and
  \begin{align*}
    \lambda_n \ge 8 \sigma_n \sqrt{\frac{ \ln p_n}{n}} \max \left\{ 1, \sup_{\btheta \in \Theta'_n} \lf\| \bar  J(\btheta^*, \btheta)_{(2, 1)} \; \bar J(\btheta^*, \btheta)^{-1}_{(1)} \ri\|_{\infty} \right\} .
  \end{align*}
  Suppose that $s_n$ and $p_n$ satisfy the conditions in \eqref{eq:MestP-dim-cond}.
  Then, with probability tending to 1, the penalized equation has a unique solution $\hbtheta$ on $\Theta_n$ with $\hbtheta_{(2)} = \0$ and $\| \hbtheta - \btheta^*\| = O_p(\sqrt{s_n/n})$. Additionally, $\hbtheta_{(1)}$ is asymptotically normal and as efficient as the oracle solution.
\end{corollary}

\subsubsection{Example: Stochastic Optimization}

Our results also apply when $K_n \to \infty$, as in iterative procedures such as gradient descent or boosting.
Suppose we want to learn $\theta_\infty^* \in \R$ solving $\E[f(\bX; \theta_\infty^*)] = 0$.
Given an initial value $\theta_0^* \in \R$, define $\theta_k^* = \theta_{k - 1}^* - \alpha \E[f(\bX; \theta_{k - 1}^*)]$.
Under appropriate conditions on the learning rate $\alpha$ and smoothness of $f$, the sequence $\theta_k^*$ converges geometrically to $\theta_\infty^*$.

Let $\bX_{1}, \dots, \bX_{n}$ be \emph{iid} samples from the distribution of $\bX$ and let $\Bcal_1, \ldots, \Bcal_{K_n}$ partition $\{1, \dots, n\}$ into sets of size $n/K_n$.
Define $\htheta_0 = \theta_0^*$ and, for $1 \le k \le K_n$,
$
\htheta_k = \htheta_{k - 1} - \alpha \frac{K_n}{n} \sum_{i \in \Bcal_k} f(\bX_{i}; \htheta_{k - 1}),
$
and write the entire iteration path $\hbtheta$ as the solution of $n^{-1} \sumin \phi(\bX_{i}; \hbtheta) = \0$ with
\begin{align}
  \phi(\bX_i; \btheta) =\begin{pmatrix}
    K_n\ind_{i \in \Bcal_1} [\theta^*_{0} - \theta_1 -  \alpha f(\bX_{i}; \theta^*_0)]  \\
    \vdots \\
    K_n\ind_{i \in \Bcal_{K_n}} [\theta_{K_n - 1} - \theta_{K_n} - \alpha f(\bX_{i}; \theta_{K_n - 1})]
  \end{pmatrix},
\end{align}
and similarly for the population version.
The following is one result under simple conditions.

\begin{corollary} \label[corollary]{cor:GD}
  Let $f'(\bx; \theta) = \partial_{\theta} f(\bx; \theta)$.
  Suppose that $K_n^3 / n \to 0$, $\|\btheta^*\|_\infty = O(1)$, $\sup_{\theta \in \Theta_n}  \E[f(\bX; \theta)^4] \allowbreak= O(1)$, $f' \in [\kappa, L]$  and $|f''| \le L$ for some $\kappa, L \in (0, \infty)$, and that $0 < \alpha \le 1/L$.

  Then $\| \hbtheta - \btheta^*\| = O_p\lf(K_n/\sqrt{n}\ri)$,
  and $\hbtheta$ is unique with probability $\to 1$, and for any $ A_n \in \R^{q \times K_n}$ with $\|  A_n\| = O(1)$, we have
\begin{align*}
  \sqrt{n / K_n} A_n (\hbtheta_{} - \btheta_{}^* )
 \to_d \Ncal\lf(0, \lim_{n \to \infty}  A_n \Sigma_n  A_n^\top\ri),
\end{align*}
where $\Sigma_n$ is symmetric and, with the convention $\prod_{j = i}^{i - 1} a_j = 1$, it holds that for $i \le j$, 
\begin{align*}
  \Sigma_{ i, j} &= \alpha^2   \sum_{k = 1}^{i} \var[f(\bX; \theta_{k-1}^*)] \lf[\prod_{m = k}^{i - 1} (1 - \alpha \E[f'(\bX_{}; \theta_{m}^*)])\ri]^2 \lf[\prod_{m = i}^{j - 1} (1 - \alpha \E[f'(\bX_{}; \theta_{m}^*)])\ri].
\end{align*}
\end{corollary}
The corollary couples the sample and population iteration paths, showing that they are globally close and that finite-dimensional linear summaries converge to a Gaussian limit.
For example, $ A_n = (0, \dots, 0, 1)$ gives $\sqrt{n/K_n}$-convergence of the final iterate, while $ A_n = (1 / \sqrt{K_n}, \dots, 1 /  \sqrt{K_n})$ gives $\sqrt{n}$-convergence of the averaged iterate.
Another choice is $A_n = (\be_{\lceil K_n/q \rceil}, \be_{\lceil 2K_n/q \rceil}, \dots, \be_{K_n})^\top$, where $\be_k$ is the $k$th standard unit vector.
This discretizes the process $\htheta(t) = \htheta_{\lceil t K_n \rceil}$, $t \in [0, 1]$.
For $n, K_n \to \infty$, one can verify that $\sqrt{n/K_n} A_n(\hbtheta-\btheta^*) \to_d \Ncal(0,V)$ with diagonal $V$, suggesting white-noise fluctuations around the population path.
The assumptions can be relaxed to allow, e.g., a slowly decaying learning rate or probabilistic bounds on $f'$; other iterative algorithms can be handled similarly.

Together, these examples illustrate how the general theory accommodates various estimation procedures beyond standard settings, including dependent data and stepwise procedures with a diverging number of steps. Combined with the broad penalty conditions developed above, this provides a unified route to consistency and asymptotic normality in high-dimensional estimating equations.

\newpage
\bibliography{bibliography}
\bibliographystyle{apalike}
\newpage

\appendix

\section{Penalties covered by our framework}\label{sec:app_penalties}

\begin{exmp}[Lasso]
  The $\ell_1$ penalty \citep{TibLasso} $p_{\lambda_n}(\btheta) = \lambda_n \| \btheta \|_1$ is not differentiable at $0$ and the generalized gradient is given by $ \partial p_{\lambda_n}(\btheta)_k = \lambda_n \sign(\theta_{k})$, where $\sign(0)$ is allowed to be any number in $[-1, 1]$.
\end{exmp}

\begin{exmp}[Elastic Net]
  The elastic net penalty \citep{ElN} is given by $p_{\lambda_{n, 1}, \lambda_{n, 2}}(\btheta) = \lambda_{n, 1} \| \btheta \|_1 + \lambda_{n, 2}  \| \btheta \|_2^2$.
  It holds that $\partial p_{\blambda_n}(\btheta)_k = \lambda_{n, 1} \sign(\theta_{k}) + 2 \lambda_{n, 2} \theta_{k}$.
  For $\lambda_{n, 1} > 0$, the elastic net can induce sparsity.
\end{exmp}

\begin{exmp}[$\ell_q$ penalty]
  The $\ell_q$ penalty \citep{Frank93} $p_{\lambda_n}(\btheta)= \lambda_n \| \btheta \|_q^q$ 
with $q \in (0, 1]$ is not differentiable at 0 and can therefore induce sparsity. 
   The $\ell_q$ penalty is also not Lipschitz around $0$ for $q < 1$, but the definition of the generalized gradient can be extended to functions that are not locally Lipschitz, see \citet[Section 2.4]{clarke1990optimization}. 
   This leads to generalized gradients of the form $\partial p_{ \lambda_n}(\btheta)_k = \{ \lambda_n q \sign(\theta_{k}) |\theta_{k}|^{q - 1}\}$ for $\theta_{k} \neq 0$ and $\partial p_{ \lambda_n}(\btheta)_k = \R$ for $\theta_{k} = 0$.
  The $\ell_q$ penalty with $q < 1$ is special in the sense that $\btheta = \0$ is always a solution to the penalized estimating equation.
\end{exmp}

\begin{exmp}[SCAD]
  The SCAD penalty (\emph{smoothly clipped absolute deviation}) \citep{Fan1997, Fan01} was designed to produce a penalized estimator that is unbiased for large parameters.
  For a single parameter, the SCAD penalty and its derivative are defined as
  \begin{align*}
    p_{\lambda_n}(\theta) = \begin{cases}
                            \lambda_n|\theta|                                             & \hspace*{-12pt} \text{if } |\theta| \le \lambda_n,             \\
                            \frac{2a \lambda_n |\theta| - \theta^2 -\lambda_n^2}{ 2(a - 1)} & \hspace*{-12pt} \text{if } \lambda_n < |\theta| \le a \lambda_n, \\
                            \frac{(a + 1) \lambda_n^2}{2}                                 & \hspace*{-12pt} \text{if } |\theta| > a \lambda_n,
                          \end{cases}
    \
    p'_{\lambda_n}(\theta) = \begin{cases}
                             \lambda_n \sign(\theta)                                  & \hspace*{-12pt}\text{if } |\theta| \le \lambda_n,             \\
                             \frac{\sign(\theta)  (a \lambda_n - |\theta|)}{ (a - 1)} & \hspace*{-12pt}\text{if } \lambda_n < |\theta| \le a \lambda_n, \\
                             0                                                      & \hspace*{-12pt}\text{if } |\theta| > a \lambda_n,
                           \end{cases}
  \end{align*}
  for some fixed $a > 2$.
  For multiple parameters, the SCAD penalty is used componentwise as $p_{\lambda_n}(\btheta) = \sum_{k=1}^{p_n} p_{\lambda_n}(\theta_{k})$.
  Around the origin, SCAD coincides with the Lasso penalty, so it can induce sparsity.
  Since the derivative $p'_{\lambda_n}(\theta)$ is zero for all $|\theta| \ge a \lambda_n$, it leads to an unbiased estimating equation for large parameters.
\end{exmp}

\begin{exmp}[MCP]
  A penalty that is unbiased for large coefficients (i.e., $p'_{\lambda_n}(| \theta |)  = 0$ for $|\theta| \ge a \lambda_n$ with some $a > 0$) and induces sparsity (e.g., $\lim_{\theta \downarrow 0} p'_{\lambda_n}( \theta ) = \lambda_n$) must be nonconvex.
  The \emph{minimax concave penalty} (MCP) \citep{Zhang10} is the ``most convex'' penalty among the penalties satisfying unbiasedness and sparsity, i.e., it minimizes the maximum concavity. For a single parameter, it is given by
  \begin{align*}
    p_{\lambda_n}(\theta) = \ind\{|\theta| \le a \lambda_n\} \lf(\lambda_n |\theta| - \frac{\theta^2 }{2 a}\ri) + \ind\{|\theta| > a \lambda_n\} \frac{a \lambda_n^2 }{2},
  \end{align*}
  with derivative
  \[
    p'_{\lambda_n}( \theta ) = \ind\{|\theta| \le a \lambda_n\} \sign(\theta) \frac{(a \lambda_n - |\theta|)}{a}
  \]
  for some fixed $a > 0$ \citep{FanLv2010}. The penalty shares a similar behavior to SCAD in that it coincides with the Lasso around $\theta = 0$ and yields an unbiased estimating equations for $|\theta|$ large.
\end{exmp}

\begin{exmp}[Fusion penalty]
  The aim of a \emph{fusion penalty} is to reduce the number of different coefficients, which is of particular interest for categorical data \citep{FLasso}.
  This is achieved by penalizing differences between coefficients, e.g., $p_{\lambda_n}(\btheta) = \lambda_n\sum_{k=1}^{p_n - 1} | \theta_{ k+1} - \theta_{ k}|$.
  This penalty fits within our framework via reparametrization, i.e., by defining $\beta_{1} = \theta_{1}$ and $\beta_{ k} = \theta_{k} - \theta_{k-1}$ and adapting the estimation function $\phi$ accordingly.
  Then, zero-entries of $\bbeta$ correspond to parameters being ``fused''.
\end{exmp}

\begin{exmp}[Group Lasso, non-overlapping groups] \label{ex:GroupL}
  The Group Lasso penalty \citep{GroupL} is given by $p_{\blambda_n}(\btheta) =  \sum_{g= 1}^{K_n} \lambda_{n, g}  \| \btheta_{ G_g} \|_2$ with a partition $G_1, \ldots, G_{K_n}$ of $\{ 1, \ldots, p_n \}$.
  In this case, $\btheta_{(2)}$ contains all $\btheta_{G_g}$ with $\btheta^*_{ G_g} = \0$.
  The generalized gradient is given by 
  \begin{align*}
    \partial p_{\blambda_n}(\btheta)_{G_g} = \ind\{ \btheta_{ G_g}  \neq \0\} \, \lambda_{n,g} \frac{\btheta_{ G_g}}{\| \btheta_{ G_g} \|_2} +  \ind\{ \btheta_{ G_g}  = \0\} \, B(\0_{ | G_g |}, \lambda_{n,g}) \quad \text{for all } g  =1, \ldots, K_n.
  \end{align*}
  Note that the Group Lasso contains the $\ell_2$ norm penalty $p_{\lambda_n}(\btheta) = \| \btheta \|_2$ as a special case with one single group $G_1 = \{ 1, \ldots, p_n\}$.
  In this case, the penalty only induces sparsity if $\btheta^* = \0$.
\end{exmp}

\begin{exmp}[Group Lasso, overlapping groups] \label{ex:GroupLO}
  One may also use the Group Lasso with overlapping groups $G_1, \ldots, G_{K_n} \subseteq \{ 1, \ldots, p_n \}$.
  In this case, $\btheta_{(2)}$ contains all $\btheta_{k}$ belonging to at least one group $G_g$ with $\btheta^*_{ G_g} = \0$.
  The generalized gradient is the sum of valid generalized gradients of Group Lasso with non-overlapping groups, i.e,
  \[
  \bz \in \partial p_{\blambda_n}(\btheta) \iff \bz = \sum_{g=1}^{K_n} \bz^{(g)}, \text{ where } \bz^{(g)} \in \partial f^{(g)}(\btheta) \text{ for all } g =1, \ldots, K_n,
  \]
  where $f^{(g)}(\btheta) \coloneq \lambda_{n,g} \| \btheta_{G_g} \|_2$.
  This generalized gradient also satisfies Assumption \ref{A:Penalty2}.
\end{exmp}

\begin{exmp}[Group SCAD]\label{ex:GroupSCAD}
  Another group penalty of the form $p_{\blambda_n} = \sum_{g=1}^{K_n} p_{\lambda_{n,k}}(\| \btheta_{G_g} \|_2)$ is the Group SCAD \citep{GroupSCAD,GroupSCAD2}, where $p_{\lambda_{n,k}}(\theta)$ denotes the SCAD penalty defined above.
  For $\| \btheta_{G_g} \|$ small, Group SCAD coincides with Group Lasso, while for $\| \btheta_{G_g} \| > a \lambda_{n,g}$, the generalized gradient is $\bnull$, leading to unbiased estimates.

  A similar extension of MCP to Group MCP \citep{Breheny2009} is also possible.
\end{exmp}
Except for Group Lasso and Group SCAD, all presented penalties are coordinate-separable and can also be used with a vector of tuning parameters $\blambda_n$, i.e., parameter specific $\lambda_{n, k}$.
 
\section{Proofs of Main Results}\label{sec:ProofsTh}
To simplify the notation in the following proofs and results, we shall use the following notation:
\begin{align*}
  \P_n f = \frac 1 n \sumin f(\bX_i), \qquad P f = \frac 1n \sumin \E[f(\bX_i)].
\end{align*}

\subsection{Proof of \autoref{theorem1}}

We first show that the sets $\Theta_n$ contain a solution of the estimating equation \eqref{eq:estim1} with probability tending to 1.
From an extension of the intermediate value theorem in \citet[Theorem 2.3]{IVT}, it follows that if
\[
  \sup_{\| \bu \| = 1} \langle r_n C \bu, \P_n \phi(\btheta^* + r_n C \bu) \rangle \le 0
\]
holds, there is a solution $\hbtheta$ of $\P_n \phi(\btheta)  = \0$ that satisfies $\lVert \hbtheta - \btheta^* \rVert \le r_n C$.
We show that by choosing $C$ large enough, the probability that
\begin{equation}
  \label{eq:th1proof}
  (r_n C)^{-1} \sup_{\| \bu \| = 1} \langle \bu, \P_n \phi(\btheta^* + r_n C \bu) \rangle \le -c < 0
\end{equation}
holds for some $c > 0$ becomes arbitrarily close to $1$.
We have
\begin{align*}
   & \quad  (r_n C)^{-1} \sup_{\| \bu \| = 1}\langle \bu, \P_n \phi(\btheta^* + r_n C \bu) \rangle \\
   & \le (r_n C)^{-1} \sup_{\| \bu \| = 1} \langle \bu, \P_n \phi(\btheta^* ) \rangle
  + (r_n C)^{-1} \sup_{\| \bu \| = 1}  \langle \bu, \P_n [\phi(\btheta^* + r_n C \bu ) - \phi(\btheta^*)] \rangle.
\end{align*}
The first term is of order $O_p(1/C)$, since
\begin{equation}\label{eq:Th1dep}
 \sup_{\| \bu \| = 1} | \langle \bu, \P_n \phi(\btheta^* ) \rangle | = \|\P_n \phi(\btheta^* )  \| = O_p(n^{-1/2} \sqrt{\mathrm{tr}(I(\btheta^*))}) = O_p(r_n)
\end{equation}
  
by \cref{lem:Cons1}.
Choosing $C$ large enough, it suffices that the second term remains below some $-c$ with probability going to $1$.
It holds that
\begin{align}\label{eq:Th1dep2}
  (r_n C)^{-1}  \langle \bu, \P_n [\phi(\btheta^* + r_n C \bu ) - \phi(\btheta^*)] \rangle
   & \le \P_n [\bu^\top   H_n \bu ]   \\                                            
   & =  P [\bu^\top   H_n \bu ] + (\P_n - P) [\bu^\top   H_n \bu ], \nonumber
\end{align}
with $H_n$ as defined in assumption \ref{A:Cons1}\ref{eq:Hn-def}. Hence, 
$$
  (r_n C)^{-1} \sup_{\| \bu \| = 1}  \langle \bu, \P_n [\phi(\btheta^* + r_n C \bu ) - \phi(\btheta^*)] \rangle
  \le \lambda_{\max}\lf( \frac 1 n\sum_{i = 1}^n\E[H_n(\bX_i)]\ri) + \|  (\P_n - P) H_n\|.
$$
By assumption \ref{A:Cons1}\ref{eq:Hn-identifiable}, we have $ \lambda_{\max}(n^{-1}\sum_{i = 1}^n\E[H_n(\bX_i)]) \le -c$ for some $c > 0$ and large enough $n$, and \cref{lem:Hn-convergence} gives $\|  (\P_n - P) H_n\| = o_p(1)$. This proves \eqref{eq:th1proof}.

We now show that every solution in $\Theta_n$ must satisfy $\|\hbtheta - \btheta^*\| \le Cr_n$ for some $C < \infty$, with probability tending to 1. 
Suppose this is not the case and define $C_n^{-1} = r_n /  \|\hbtheta - \btheta^*\| = o_p(1)$.
Then, we can write $\hbtheta = \btheta^* + r_n C_n \hbu$ with $\| \hbu \| = 1$.
It holds that
\begin{align*}
 0 & = (r_n C_n)^{-1}  \langle \hbu, \P_n \phi(\btheta^* + r_n C_n \hbu ) \rangle  \le \sup_{\|\bu\| = 1} (r_n C_n)^{-1} \langle \bu, \P_n \phi(\btheta^* + r_n C_n \bu ) \rangle\\
 & \le O_p(C_n^{-1}) -c + o_p(1) = -c + o_p(1),
\end{align*}
where the second inequality follows from the above arguments.
Hence, $\hbtheta$ cannot be a solution with probability tending to 1.
 \subsection{Proof of \autoref{theorem1-uniqueness}}

The claim is trivial when no solution exists. Otherwise, let $\hbtheta$ be any solution to the estimating equation \eqref{eq:est_eq}. By \cref{theorem1}, we may assume that $\|\hbtheta - \btheta^*\| \le r_n C$ for some $C \in (0, \infty)$ and $n$ large enough.
Suppose there is another solution $\hbtheta + \bu$.
By the strengthened assumption \eqref{eq:Hn-def-unique},
\begin{align*}
  \langle \bu, \P_n \phi(\hbtheta + \bu) \rangle
   =   \langle \bu, \P_n \phi(\hbtheta + \bu) \rangle -   \langle \bu, \P_n \phi(\hbtheta) \rangle
   & \le  \P_n[\bu^\top H_n \bu  ]                                          \\
   & =   P[\bu^\top H_n \bu ] + (\P_n - P) [\bu^\top H_n \bu   ]              \\
   & \le\|\bu\|^2 (-c +  o_p(1)) ,
\end{align*}
uniformly on the set $\{\bu \colon \|\hbtheta + \bu - \btheta^*\| \le r_n C\}$ using \ref{A:Cons1} and \cref{lem:Hn-convergence}. The right-hand side is strictly negative on the subset where $\bu \neq \bnull$ with probability tending to 1, so it must hold that $\bu = \0$.

 \subsection{Proof of \autoref{theorem2}}
We have
\begin{align*}
\0 & = \P_n \phi(\hbtheta)
= \P_n \phi(\btheta^*) + \P_n [ \phi(\hbtheta) - \phi(\btheta^*)]\\
& =  \P_n \phi(\btheta^*) + P [ \phi(\hbtheta) - \phi(\btheta^*)] + (\P_n - P) [ \phi(\hbtheta) - \phi(\btheta^*)]\\
& =  \P_n \phi(\btheta^*) + \bar J (\btheta^*, \hbtheta) (\hbtheta - \btheta^*)   + (\P_n - P) [ \phi(\hbtheta) - \phi(\btheta^*)], 
\end{align*}
where $\bar J (\btheta^*, \btheta) \coloneqq \int_0^1 J(\btheta^* + t(\btheta - \btheta^*)) d t$.
Rearranging and adding terms gives
\[
- J(\btheta^*)(\hbtheta - \btheta^*)  = \P_n \phi(\btheta^*) + [\bar J (\btheta^*, \hbtheta) - J(\btheta^*)] (\hbtheta - \btheta^*)   + (\P_n - P) [ \phi(\hbtheta) - \phi(\btheta^*)]
\]
and
\begin{align}\label{eq:Th3dep}
- \sqrt{n} A_n J(\btheta^*)(\hbtheta - \btheta^*) 
= & \sqrt{n} A_n \P_n \phi(\btheta^*) \nonumber \\
& + \sqrt{n} A_n  [\bar J (\btheta^*, \hbtheta) - J(\btheta^*)] (\hbtheta - \btheta^*)  \\
& + \sqrt{n} A_n (\P_n - P) [\phi(\hbtheta ) - \phi(\btheta^*)].\nonumber
\end{align}
The second and the third term are negligible, since
\[
\sqrt{n} A_n [\bar J (\btheta^*, \hbtheta) - J(\btheta^*)] (\hbtheta - \btheta^*) 
= o_p\lf(\sqrt{n}\, \frac{1}{\sqrt{n}r_n}\,r_n \ri) = o_p(1)
\]
by assumption \ref{A:Asymp}, and \cref{lem:AsN1} yields 
\[
 \sqrt{n} A_n(\P_n - P)  [\phi(\hbtheta ) - \phi(\btheta^*)] =o_p(1).
\]
It remains to show a central limit theorem for 
\[
\sqrt{n} A_n \P_n \phi(\btheta^*)  = \sumin  \frac{1}{\sqrt{n}} A_n \phi_{i}(\btheta^*) \coloneq \sumin \bY_{i}.
\]
Since
\begin{align*}
\sumin \E \left[ \lVert \bY_{i} \rVert^2 \mathbbm{1}\{\lVert \bY_{i} \rVert > \eps \}  \right] 
& \le \sumin \E \left[ \lVert \bY_{i} \rVert^2 \mathbbm{1}\{\lVert \bY_{i} \rVert > \eps \} \| \bY_{i} \|^2/\eps^2  \right] \leq \sumin \E\left[ \lVert \bY_{i} \rVert^4 \right]/\eps^2,
\end{align*}
and $ \E[ \lVert \bY_{i} \rVert^4] = n^{-2}\E[ \| A_n \phi_{i}(\btheta^*) \|^4 ] = o(n^{-1})$ for all $i = 1, \ldots, n $, by \ref{A:Cons2}, we have
\[
\sumin \mathbb{E} \left[ \lVert \bY_{i} \rVert^2 \mathbbm{1}\{\lVert \bY_{i} \rVert > \eps \}  \right] \to 0 \text{ for every } \eps > 0.
\]
Since $\sumin \mathbb{E}[\bY_{i}]=\boldsymbol{0}$ for all $i=1, \ldots, n$ and 
\begin{align*}
\sum_{i=1}^n \mathrm{Cov}(\bY_{i}) 
 = \frac 1 n \sumin  \mathrm{Cov}[A_n \phi_{i}(\btheta^*)] 
 =  \frac 1 n A_n\sumin  \mathrm{Cov}[ \phi_{i}(\btheta^*)]  A_n^\top 
=  A_n  I(\btheta^*)  A_n^\top
\to \Sigma,
\end{align*}
the conditions of the Lindeberg-Feller central limit theorem \citep[Section 2.8]{vdV2} are satisfied, and we obtain 
\[
\sqrt{n} A_n J(\btheta^*)  (\hbtheta - \btheta^*) \to_d \Ncal(\0, \Sigma).
\]
 
\subsection{Proof of \autoref{theorem3}}

Let $\hbtheta \in \Theta_n$ be any solution of \eqref{eq:Def_M}, which we write as $\hbtheta = \btheta^* + \bu$.
Since $\hbtheta$ is a solution of \eqref{eq:Def_M}, there is a $\bz \in \partial p_{\blambda_n}(\btheta^* + \bu)$ such that
\begin{align}
  0 & = \langle \bu, \P_n \phi(\btheta^* + \bu) - \bz \rangle    \le  \langle \bu, \P_n  \phi(\btheta^*) \rangle +  \langle \bu,  \P_n H_n \bu \rangle - \langle \bu,\bz \rangle.   \label{eq:penalized-rate-proof} 
\end{align}
Using H\"older's inequality, \cref{lem:eta_n} and \ref{A:phi-moments}, the first term in \eqref{eq:penalized-rate-proof} can be bounded by
\begin{align*}
  \langle \bu, \P_n  \phi(\btheta^*) \rangle \le \|\bu\|_1  \eta_n \le \sqrt{\nu_n}  \|\bu\|_2 \eta_n
\end{align*}
with probability tending to 1.
For the second term in \eqref{eq:penalized-rate-proof}, \ref{A:phi-H-penalty} and H\"older's inequality yield
\begin{align*}
  \langle \bu,  \P_n H_n \bu \rangle &=  \langle \bu,  P H_n \bu \rangle +  \langle \bu,  (\P_n - P) H_n \bu \rangle \\
  &\le -c\|\bu\|^2 + \nu_n \|\bu\|^2 \max_{1 \le j, k \le p_n} |  (\P_n - P) H_{n, j, k} |  = -c\|\bu\|^2 + \|\bu\|^2 o_p(1),
\end{align*}
where the last step follows from  \cref{lem:Hn-convergence-pen}.
For the third term in \eqref{eq:penalized-rate-proof}, H\"older's inequality gives
\begin{align*}
  -\langle \bu, \bz \rangle 
  &\le \sqrt{\nu_n} \|\bu\|_2 \bar b_n.
\end{align*}
Altogether, we have shown
\begin{align*}
  0 \le  \|\bu\|_2  \sqrt{\nu_n}  (\eta_n + \bar b_n) - \|\bu\|_2^2[c + o_p(1)].
\end{align*}
Rearranging terms gives
\begin{align*}
  \|\hbtheta - \btheta^*\|_2 = \|\bu\|_2 \le  \frac{\sqrt{\nu_n} (\eta_n + \bar b_n)}{c + o_p(1)} = O_p\left(\sqrt{\nu_n}(\eta_n + \bar b_n)\right).
\end{align*}

\subsection{Proof of \autoref{theorem4}}

The proof is split in two steps:
\begin{enumerate}
  \item Show that there is a solution $\hbtheta_{(1)}$ to 
  \[
  \Phi_n((\btheta_{(1)}, \0))_{(1)} \in \partial p_{\blambda_n}((\btheta_{(1)}, \0))_{(1)} \in \R^{s_n}
  \]
  with $\| \hbtheta_{(1)} - \btheta_{(1)}^*\| = O_p(\tilde r_n)$.
  \item Show that $\hbtheta = (\hbtheta_{(1)}, \0)$ is also a valid solution to 
  \[
  \Phi_n(\hbtheta)_{(2)} \in \partial p_{\blambda_n}(\hbtheta)_{(2)}.
  \]
  Since $\partial p_{\blambda_n}(\btheta)_{G_g} \supseteq B(\0_{ | G_g |}, \lambda_{n,g})$ for all $g \in I_{(2)}$ by \ref{A:Penalty2}, 
  it suffices to verify 
  \[
  \max_{g \in I_{(2)}} \lambda_{n,g}^{-1} \ \| \Phi_n(\hbtheta)_{G_g}  \|_2  \le 1.
  \]
\end{enumerate}
Together this implies that $\hbtheta$ is a valid solution to the full problem \eqref{eq:Def_M}. 

\myparagraph{Step 1} 
Similar as in the proof of \cref{theorem1}, it suffices to show that
\begin{equation}
  \label{eq:pr_th3}
  (\tilde r_n C)^{-1} \sup_{\| \bu \| = 1, \bu_{(2)} = \0} \langle  \bu_{(1)}, \Phi_n (\btheta^* + \tilde r_n C \bu)_{(1)} - \nabla_{\btheta_{(1)}} p_{\blambda_n}(\btheta^* + \tilde r_n C\bu) \rangle \le -c < 0,
\end{equation}
for large enough $C$ and high probability.
The left-hand side of \eqref{eq:pr_th3} is upper bounded by 
\begin{align}
\begin{split}
\label{eq:pr_th3_2}
&  (\tilde r_n C)^{-1} \sup_{\| \bu \| = 1, \bu_{(2)} = \0}  \langle  \bu_{(1)}, \P_n \phi(\btheta^* + \tilde r_n C \bu)_{(1)} \rangle  \\
 & - (\tilde r_n C)^{-1} \inf_{\| \bu \| = 1, \bu_{(2)} = \0} \langle  \bu_{(1)}, \nabla_{\btheta_{(1)}} p_{\blambda_n}(\btheta^* + \tilde r_n C\bu) \rangle  
 \end{split}
\end{align}
For the first term, one can proceed similarly to the proof of \cref{theorem1} and show that, using \ref{A:Cons1} and choosing $C$ large enough, 
\[
 (\tilde r_n C)^{-1} \sup_{\| \bu \| = 1, \bu_{(2)} = \0} \langle \bu_{(1)}, \P_n \phi(\btheta^* + \tilde r_n C \bu)_{(1)} \rangle \le -c < 0
\]
holds with arbitrarily high probability. It remains to show that the second term in \eqref{eq:pr_th3_2} is sufficiently small.
A Taylor expansion yields
\begin{align*}
&\quad (\tilde r_n C)^{-1} \langle  \bu_{(1)}, \nabla_{\btheta_{(1)}} p_{\blambda_n}(\btheta^* + \tilde r_n C\bu) \rangle \\
& = (\tilde r_n C)^{-1} \langle  \bu_{(1)}, \nabla_{\btheta_{(1)}} p_{\blambda_n}(\btheta^*) \rangle 
+  \langle  \bu_{(1)}, \bar p_{\blambda_n}(\btheta^*, \btheta^* + \tilde r_n C\bu) \,  \bu_{(1)} \rangle ,
\end{align*}
where $\bar p_{\blambda_n}(\btheta^*, \btheta) \coloneqq \int_0^1 \nabla^2_{\btheta_{(1)}} p_{\blambda_n}(\btheta) \rvert_{\btheta = \btheta^* + t(\btheta - \btheta^*)} dt$.
Using the Cauchy-Schwarz inequality, we obtain
\begin{align*}
 (\tilde r_n C)^{-1}  \sup_{\| \bu \| = 1, \bu_{(2)}= \0} \left| \langle  \bu_{(1)}, \nabla_{\btheta_{(1)}} p_{\blambda_n}(\btheta^*) \rangle  \right|
&\le  (\tilde r_n C)^{-1}  \sup_{\| \bu \| = 1, \bu_{(2)}= \0} \|  \nabla_{\btheta_{(1)}} p_{\blambda_n}(\btheta^*)\| \| \bu \| \\
& \le (\tilde r_n C)^{-1}  \sqrt{s_n} \|  \nabla_{\btheta_{(1)}} p_{\blambda_n}(\btheta^*)\|_{\infty} \\
&= ( \tilde r_n C)^{-1}  \sqrt{s_n} b_n^* =  O(1/C),
\end{align*}
which becomes negligible by choosing $C$ large enough.
For the second term, \ref{A:Penalty1} gives 
\begin{align*}
  \sup_{\| \bu \| = 1, \bu_{(2)} = \0} \langle  \bu_{(1)},  \bar p_{\blambda_n}(\btheta^*, \btheta^* + \tilde r_n C\bu) \,  \bu_{(1)} \rangle  = o(1),
\end{align*}
which proves \eqref{eq:pr_th3}. 

\myparagraph{Step 2}
\cref{lem:sparsity} yields
\begin{align}\label{eq:step2Th5}
& \max_{g \in I_{(2)}} \lambda_{n,g}^{-1} \ \| \Phi_n(\hbtheta)_{G_g}  \|_2  \\
 & \, \le \quad    \max_{g \in I_{(2)}} \lambda_{n,g}^{-1} \ \| \bar J(\btheta^*, \hbtheta)_{G_g,(1)} \;\bar J(\btheta^*, \hbtheta)^{-1}_{(1)}  \;  \nabla_{\btheta_{(1)}} p_{\blambda_n}(\hbtheta) \|_2  +  \max_{g \in I_{(2)}} \lambda_{n,g}^{-1} \ \|   \P_n \phi(\btheta^*)_{G_g} \|_2 \nonumber \\
 &\quad \,  +   \max_{g \in I_{(2)}} \lambda_{n,g}^{-1} \ \| \underbrace{ \bar J(\btheta^*, \hbtheta)_{G_g,(1)} \;\bar J(\btheta^*, \hbtheta)^{-1}_{(1)}  \; \P_n \phi(\btheta^*)_{(1)} }_{\eqcolon \bv_{g}}  \|_2   + o_p(1) .\nonumber
\end{align}
The first term is smaller or equal than some $\alpha \in [0,1)$ by \ref{A:lambda}.
\cref{lem:eta_n} implies $\| \P_n \phi(\btheta^*)  \|_{\infty} \le \eta_n$ with probability tending to 1.
On this event, the condition on $\lambda_{n,g}$ for $g \in I_{(2)}$ in \ref{A:lambda2} gives, for the second term,
\[
\max_{g \in I_{(2)}} \lambda_{n,g}^{-1} \ \| \P_n \phi(\btheta^*)_{G_g}  \|_2 \le  \max_{g \in I_{(2)}} \frac{\sqrt{| G_g | }}{\lambda_{n,g}}   \ \| \P_n \phi(\btheta^*)  \|_\infty
\le \frac{1-\alpha}{4} \frac{\| \P_n \phi(\btheta^*)  \|_\infty}{\eta_n} \le \frac{1-\alpha}{4}.
\]
For the third term, we obtain using $\| A \bx \|_2 = (\| A \bx \|_2 / \| \bx \|_{\infty}) \| \bx \|_{\infty} \le  \| \bx \|_{\infty} \sup_{\| \bx \|_\infty \le 1 } \| A \bx \|_2$,
\begin{align*}
  & \max_{g \in I_{(2)}} \lambda_{n,g}^{-1} \ \| \bar J(\btheta^*, \hbtheta)_{G_g,(1)} \;\bar J(\btheta^*, \hbtheta)^{-1}_{(1)} \; \P_n \phi(\btheta^*)_{(1)} \|_2 \\
  & \le \max_{g \in I_{(2)}} \lambda_{n,g}^{-1} \ \sup_{\| \bx \|_\infty \le 1}\| \bar J(\btheta^*, \hbtheta)_{G_g,(1)} \;\bar J(\btheta^*, \hbtheta)^{-1}_{(1)}  \bx \|_2 \ \| \P_n \phi(\btheta^*)_{(1)} \|_\infty \\
  & \le \frac{1-\alpha}{4}\ \frac{\| \P_n \phi(\btheta^*)  \|_\infty}{\eta_n}  \max_{g \in I_{(2)}}\ \frac{\sup_{\| \bx \|_\infty \le 1}\| \bar J(\btheta^*, \hbtheta)_{G_g,(1)} \;\bar J(\btheta^*, \hbtheta)^{-1}_{(1)}  \bx \|_2}{ J_{n,g}  } \le \frac{1-\alpha}{4},
\end{align*}
again using the condition on $\lambda_{n,g}$ in \ref{A:lambda2}
Together, we have shown
\[
\max_{g \in I_{(2)}} \lambda_{n,g}^{-1} \ \| \Phi_n(\hbtheta)_{G_g}  \|_2 \le \alpha +  \frzwei (1-\alpha) + o_p(1) = \frzwei (1+\alpha) + o_p(1) < 1
\]
with probability going to $1$.

\subsection{Proof of \autoref{theorem5}}

Suppose there is a solution $\hbtheta \in \Theta_n$ and a further solution $\wt \btheta =\hbtheta + \wt \bu \in \Theta_n$.
From \cref{theorem3} we know that $ \lVert \hbtheta - \btheta^* \rVert = O_p\left(\sqrt{\nu_n}(\eta_n + \bar b_{n}) \right)$ and $ \lVert \wt \btheta - \btheta^* \rVert = O_p\left(\sqrt{\nu_n}(\eta_n + \bar b_{n}) \right)$.
Similar to the proof of \cref{theorem3}, with some $\bz \in \partial p_{\blambda_n}( \btheta)$ and some $\tilde \bz \in \partial p_{\blambda_n}(\wt \btheta)$, it must hold that
\begin{align*}
  0
   & = \langle \wt \bu_{}, \P_n \phi(\wt \btheta) - \tilde \bz \rangle     = \langle \wt \bu_{}, \P_n [\phi(\wt \btheta) -\phi(\hbtheta)] \rangle -  \langle \wt \bu, \tilde \bz - \bz \rangle \\
   & \le -c \|\wt \bu\|_2^2 + \nu_n \|\wt \bu\|_2^2 \max_{1 \le j, k \le p_n} |  (\P_n - P) H_{ j, k} | +  \mu_n \|\wt \bu\|_2^2      \le  \|\wt \bu\|_2^2[-c +  \mu_n + o_p(1)],
\end{align*}
where we used \ref{A:phi-H-penalty} and \ref{A:Penalty3} in the first inequality, and \cref{lem:Hn-convergence-pen} in the second.
Since $\mu_n < c$ asymptotically, it must hold that $\|\wt \bu\|_2 = 0$ or, equivalently, $\wt \btheta = \hbtheta$ with probability tending to 1.
 \subsection{Proof of \autoref{theorem6}}

Similar to the proof of \cref{theorem2}, we obtain, with $\bar J (\btheta^*, \btheta)_{(1)} \coloneqq \int_0^1 J(\btheta^* + t(\btheta - \btheta^*))_{(1)} dt$,
\begin{align*}
\0  &=  \P_n \phi(\hbtheta)_{(1)} - \nabla_{\btheta_{(1)}} p_{\blambda_n}(\hbtheta) \\
 &=  \P_n \phi(\btheta^*)_{(1)} + \bar J (\btheta^*, \hbtheta)_{(1)} (\hbtheta_{(1)} - \btheta_{(1)}^*)   + (\P_n - P) [ \phi(\hbtheta)_{(1)} - \phi(\btheta^*)_{(1)}] \\
 & \quad  -  \nabla_{\btheta_{(1)}} p_{\blambda_n}(\btheta^*) - \bar p_{\blambda_n}(\btheta^*, \hbtheta) (\hbtheta_{(1)} - \btheta_{(1)}^*) ,
\end{align*}
where $\bar p_{\blambda_n}(\btheta^*, \btheta) = \int_0^1 \nabla^2_{\btheta_{(1)}} p_{\blambda_n}(\btheta) \rvert_{\btheta = \btheta^* + t(\btheta - \btheta^*)} dt$.
It then holds that
\begin{align*}
 \sqrt{n} A_n \left[-  J(\btheta^*)_{(1)}(\hbtheta_{(1)} - \btheta_{(1)}^*) +    \nabla_{\btheta_{(1)}} p_{\blambda_n}(\btheta^*) \right] 
 & =  \sqrt{n} A_n  \P_n \phi(\btheta^*)_{(1)}\\
 & \quad  + \sqrt{n} A_n  [ \bar J (\btheta^*, \hbtheta)_{(1)}  -J(\btheta^*)_{(1)}](\hbtheta_{(1)} - \btheta_{(1)}^*)  \\
&  \quad +\sqrt{n} A_n  (\P_n - P) [ \phi(\hbtheta)_{(1)} - \phi(\btheta^*)_{(1)}] \\
 &  \quad  - \sqrt{n} A_n \bar p_{\blambda_n}(\btheta^*, \hbtheta)(\hbtheta_{(1)} - \btheta_{(1)}^*).
\end{align*}
Adapting the proof of \cref{theorem2}, one can show a central limit theorem for the first term and that the second and third term are of order $o_p(1)$ by \ref{A:Asymp} and \cref{lem:AsN1}.
For the last term, we have
\[
 \sqrt{n} A_n  \bar p_{\blambda_n}(\btheta^*, \hbtheta)(\hbtheta_{(1)} - \btheta_{(1)}^*) 
 = o_p(1)
\]
by assumption \ref{A:Penalty4} and $\| A_n \| = O(1)$, which concludes the proof.
   
  \subsection{Proof of \autoref{th:dependence}}
  
  \textbf{\cref{theorem1}}:
  In \eqref{eq:Th1dep}, replace $\| \P_n \phi(\btheta^*) \|$ by
  \[
    \| \P_n \phi(\bX_i^*; \btheta^*) \| + \sqrt{p_n} \max_{1 \le k \le p_n} | \P_n [ \phi(\btheta^*)_k - \phi(\bX_i^*; \btheta^*)_k  ] | = \| \P_n \phi(\bX_i^*; \btheta^*) \|  + o_p(r_n),
  \]
  which follows from \cref{lem:dep2} with $\eps_N = r_n n / \sqrt{p_n}$, $\Fcal = \{ \phi(\bx; \btheta^*)_k, k =1,\ldots, p_n\}$ and \eqref{eq:assumpDep}.
  In \eqref{eq:Th1dep2}, replace $\P_n [\bu^\top H_n(\bX_i) \bu]$ by 
\[
  \P_n [\bu^\top H_n(\bX_i^*) \bu] + \P_n [\bu^\top (H_n(\bX_i) - H_n(\bX_i^*) ) \bu] = \P_n [\bu^\top H_n(\bX_i^*) \bu] + o_p(1),
  \]
  which follow from \cref{lem:dep2} with $\eps_N = C n $ with some $C < \infty$ such that $C n / (n - L_N) \ge B_n$ (which exists due to the additional assumption on $B_n$), $\Fcal = \{ \bu^\top H_n(\bx) \bu : \| \bu \| \le 1 \}$ and \ref{A:Cons1}\ref{eq:Hn-bounds}.
  \cref{theorem1-uniqueness} follows analogously.

  \textbf{\cref{theorem2}}: In \eqref{eq:Th3dep}, replace $\sqrt{n} A_n \P_n \phi(\btheta^*)$ by
  \[
  \sqrt{n} A_n \P_n \phi(\bX_i^*; \btheta^*) + \sqrt{n} A_n \P_n [\phi(\btheta^*) - \phi(\bX_i^*; \btheta^*)]
  = \sqrt{n} A_n \P_n \phi(\bX_i^*; \btheta^*) + o_p(1),
  \]
  which follows from \cref{lem:dep2} with $\eps_N = \sqrt{n}$, $\Fcal = \{  \langle \bu, A_n \phi(\bx; \btheta^*) \rangle : \| \bu \| \le 1 \}$ and
  \begin{align*}
   \sum_{i=L_N + 1}^{n} \Pr\left(\| A_n \phi_i(\btheta^*) \| >  \sqrt{n} /(n - L_N)\right) & \le \sum_{i=L_N + 1}^{n} \frac{(n - L_N)^4 \E [\| A_n \phi_i(\btheta^*) \|^4] }{n^2} \\ & \overset{\ref{A:Cons2}}{=} o((n - L_N)^5/n ) = o(1).
  \end{align*}
  $\sqrt{n} A_n \P_n [\phi(\hbtheta) - \phi(\btheta^*)]$ in \eqref{eq:Th3dep} is replaced by
  \[
  \sqrt{n} A_n \P_n [\phi(\bX_i^*; \hbtheta) - \phi(\bX_i^*; \btheta^*)] + \sqrt{n} A_n \P_n [\phi(\hbtheta) - \phi(\btheta^*) - (\phi(\bX_i^*; \hbtheta) - \phi(\bX_i^*; \btheta^*) )].
  \]
  In the second term, we may replace $\phi(\bX; \btheta), \phi(\bX^*;\btheta)$ by $\tilde \phi(\bX; \btheta), \tilde \phi(\bX^*;\btheta)$ from \ref{A:Asymp}, since $\E[\phi(\bX; \btheta)] = \E[\phi(\bX^*; \btheta)]$.
  Now this term is $o_p(1)$ by \cref{lem:dep2} with $\eps_N = C \sqrt{n}$ with some $C < \infty$ such that $C \sqrt{n}/(n-L_N) \ge D_n$ (which exists due to the additional assumption on $D_n$), $\Fcal = \{ \langle \bu, A_n [\tilde \phi(\bx; \btheta) - \tilde \phi(\bx; \btheta^*)] \rangle : \| \bu \| \le 1, \btheta \in \Theta_n\}$ and, using the second condition in \ref{A:Asymp} with any fixed $\tilde C < \infty$,
  \begin{align*}    
 &  \sum_{i=L_N + 1}^{n} \Pr\left(\sup_{\btheta \in \Theta_n} \| A_n [\tilde\phi_i(\btheta) - \tilde\phi_i(\btheta^*)] \| > C \sqrt{n}/(n - L_N) \right)
  \\ &  \le \sum_{i=L_N + 1}^{n} \Pr\left( \sup_{\| \bu \| \le r_n \tilde C} \| A_n [\tilde\phi_i(\btheta + \bu) - \tilde \phi_i(\btheta^*)]  \| > D_n\right) = o(1).
  \end{align*}

  \textbf{\cref{theorem3}}: In \eqref{eq:penalized-rate-proof}, replace $ \langle \bu, \P_n  \phi(\btheta^*) \rangle$ by 
  \[
   \langle \bu, \P_n  \phi(\bX_i^*; \btheta^*) \rangle +  \langle \bu, \P_n [\phi(\btheta^*) - \phi(\bX_i^*; \btheta^*)]  \rangle 
  \le \langle \bu, \P_n  \phi(\bX_i^*; \btheta^*) \rangle + \sqrt{\nu_n} \| \bu \|_2\, o_p(\eta_n),
  \]
  which follows from \cref{lem:dep2} with $\eps_N =  \eta_n n$, $\Fcal = \{ \phi(\bx; \btheta^*)_k : k = 1, \ldots, p_n \}$ and \eqref{eq:assumpDep}. 
  $\langle \bu,  \P_n H_n \bu \rangle$ in \eqref{eq:penalized-rate-proof} is replaced by
  \[
    \langle \bu,  \P_n H_n(\bX_i^*) \bu \rangle +  \langle \bu,  \P_n [H_n(\bX_i) - H_n(\bX_i^*)] \bu \rangle
     \le  \langle \bu,  \P_n H_n(\bX_i^*) \bu \rangle +  \| \bu \|_2^2 \, o_p(1),
  \]
  which follows from \cref{lem:dep2} with $\eps_N = C n / \nu_n$ with some $C < \infty$ such that $C n / (\nu_n (n - L_N)) \ge \tilde B_n$ (which exists due to the additional assumption on $\tilde B_n$) and $\Fcal = \{ H_n(\bx)_{j,k} : 1 \le j,k \le p_n \}$ and, using the second condition in \ref{A:phi-H-penalty}\ref{eq:Hn-bounds2},
  \[
  \sum_{i=L_N + 1}^{n} \Pr\left( \max_{1 \le j,k \le p_n} | H_n(\bX_i) | >  \frac{C n }{\nu_n (n - L_N)} \right)
  \le \sum_{i=L_N + 1}^{n} \Pr\left( \max_{1 \le j,k \le p_n} | H_n(\bX_i) | >  \tilde B_n \right) = o(1).
  \]
  \cref{theorem5} follows analogously.

  \textbf{\cref{theorem4}}: Recall that $\lambda_{n,g}^{-1} \le (1-\alpha)/(4  \eta_n \sqrt{ |G_g|})$ for all $g \in I_{(2)}$ by \ref{A:lambda2}.
  In \eqref{eq:step2Th5}, replace $\max_{g \in I_{(2)}} \lambda_{n,g}^{-1} \| \Phi_n (\hbtheta)_{G_g} \|_2$ by 
  \begin{align*}
    & \max_{g \in I_{(2)}} \lambda_{n,g}^{-1} \| \P_n \phi (\bX_i^*; \hbtheta)_{G_g} \|_2
    + \max_{g \in I_{(2)}} \lambda_{n,g}^{-1} \sqrt{| G_g |} \max_{k \in G_g} | \P_n [\phi(\hbtheta)_k - \phi(\bX_i^*; \hbtheta)_k] | \\
    & \le \max_{g \in I_{(2)}} \lambda_{n,g}^{-1} \| \P_n \phi (\bX_i^*; \hbtheta)_{G_g} \|_2 + \frac{1-\alpha}{4 \eta_n} \max_{k = s_n+1, \ldots, p_n }| \P_n [\phi(\hbtheta)_k - \phi(\bX_i^*; \hbtheta)_k] | \\
    & =  \max_{g \in I_{(2)}} \lambda_{n,g}^{-1} \| \P_n \phi (\bX_i^*; \hbtheta)_{G_g} \|_2 + o_p(1),
  \end{align*}
  which follows from \cref{lem:dep2} with $\eps_N = \eta_n n$, $\Fcal = \{ \phi(\bx; \btheta)_k : k = s_n + 1, \ldots, p_n, \btheta \in \Theta'_n \}$ and \eqref{eq:assumpDep}. 
\subsection{Proofs for \texorpdfstring{\cref{sec:RSC-failure}}{the RSC counterexample}}

\label{sec:RSC-failure-proof}

In the following, write $\bar\bY = n^{-1} \sum_{i=1}^n \bY_i$ and $\bar\xi_g = n^{-1} \sum_{i=1}^n \xi_{i,g}$.

\begin{proof}[Proof of \cref{cor:RSC-failure}]
  We first verify condition \ref{A:phi-H-penalty} for the choice $H_n(X) = -X$ on the cone $\Theta^*(\nu_n)$. Condition (i) holds with equality and (ii) follows from $\E[X_i] = I_{p_n}$.
  Condition (iii) holds  because 
  \begin{align*}
    \max_{j, k} \frac 1 n \sumin \E[H_n(X_i)_{j, k}^2] = O(1),
  \end{align*}
  and, setting $\tilde B_n = \ln (n M_n)$, it holds
  \begin{align*}
    \Pr\left(\max_{1\le j, k \le p_n}|H_n(X_i)_{j, k}| \ge \tilde B_n\right) \le  \Pr\left(1 + \max_{1\le g\le M_n}|\xi_g| \ge \tilde B_n\right)  = o(1/n),
  \end{align*}
  using Gaussianity of $X_i$.
Similarly, \ref{A:phi-moments} holds with $\eta_n = O(\sqrt{\ln p_n / n})$ by  the Gaussianity of $\bY_i$.
Hence, \cref{theorem3} applies on $\Theta_n\cap\Theta^*(\nu_n)$ and, since $\bar b_n=\lambda_n=O(\eta_n)$ for the Lasso penalty, shows that any solution in this set satisfies
\[
  \|\hbtheta\|_2=O_p\left(\sqrt{\nu_n}\eta_n\right).
\]
Moreover, \ref{A:Penalty3} holds with $\mu_n=0$, and \eqref{eq:Hn-def-unique-pen} holds with equality.
Thus, \cref{theorem5} implies that there is at most one solution in $\Theta_n\cap\Theta^*(\nu_n)$ with probability tending to one. Finally, for a sufficiently large $C_\lambda$,
\begin{align*}
  \|\Phi_n(\bnull)\|_\infty = \|\bar \bY_n \|_\infty < \lambda_n, 
\end{align*}
with probability tending to one, so $\0$ indeed solves \eqref{eq:RSC-failure-equation}.
\end{proof}

\begin{proof}[Proof of \cref{lem:RSC-failure}]
For $g=1,\ldots,M_n$, define $\bu_g=m_n^{-1}\bm 1_{S_g}$.
Then $\|\bu_g\|_1=1$, $\|\bu_g\|_2^2=m_n^{-1}$, and
\begin{align*}
  \left\langle \bu_g, \Phi_n(\0) - \Phi_n(\bu_g) \right\rangle
  = \bu_g^\top \left(\frac 1 n \sumin X_i\right) \bu_g
  = \frac1{m_n} + \bar\xi_g.
\end{align*}
Consequently, the RSC event
  \begin{align*}
    \left\langle \bu, \Phi_n(\0) - \Phi_n(\bu) \right\rangle
    \ge c\|\bu\|_2^2 - c_1 \|\bu\|_1^2 \eta_n^2,
    \qquad \forall \|\bu\|_1 \le 1,
  \end{align*}
implies
\begin{align*}
    \min_{1\le g\le M_n} \bar \xi_g    \ge \frac{c-1}{m_n}-c_1\frac{\ln p_n}{n}.
\end{align*}
The variables $\sqrt n \bar\xi_g$ are independent standard Gaussian random variables, while
\[
  a_n \coloneqq
  \sqrt n\left\{\frac{c-1}{m_n}-c_1\frac{\ln p_n}{n}\right\}
  \longrightarrow c - 1.
\]
Writing $\Psi$ for the standard Gaussian distribution function, it holds
\[
  \Pr\left(\sqrt{n} \min_{1\le g\le M_n} \bar \xi_g \ge a_n \right)
  =\{1-\Psi(a_n)\}^{M_n} \to 0,
\]
as $a_n$ remains bounded for large $n$ and $M_n\to\infty$.
\end{proof}
 
\section{Proofs of Corollaries}
\subsection{Proof of \autoref{cor:Mest}}

\myparagraph{Consistency}
For \ref{A:Cons1}, we can choose
\[
H_n(\bX_i) = - \inf_{\btheta \in \Theta_n} | \psi'(Y_{i}, \bX_{i}^\top \btheta)| \bX_{i}  \bX_{i}^\top.
\]
Then, \ref{A:Cons1}\ref{eq:Hn-def} is fulfilled by the definition of $H_n$ and \ref{eq:Hn-identifiable} follows from \eqref{eq:Mest-eigcond}.
For \ref{eq:Hn-bounds}, $ \inf_{\btheta \in \Theta_n}|\psi'(Y_{i}, \bX_{i}^\top \btheta)|$ is negligible as this term is bounded, so it remains to verify
\[
\sup_{\| \bu \| = 1} \E\left[ (\bu^\top \bX  \bX^\top \bu)^2 \right] = o(n), \quad
 \lf\|  \E[( \bX  \bX^\top)^2 \ind_{\|\bX\|^2 \le B_n}] \ri\| = o(n/\ln p_n),  
\]
and $n \Pr\lf(\| \bX  \bX^\top \|> B_n \ri) = o(1)$.
The first condition holds since $(\bu^\top \bX  \bX^\top \bu)^2 = (\bu^\top \bX)^4$ and $\E[(\bu^\top \bX)^4] \allowbreak = O(\| \bu \|^4)$. 
The second condition follows from \cref{lem:semi-definite} using $\| \E[\bX \bX^T] \|=O(1)$.
The third condition holds for $B_n = \omega_n n/\ln p_n$ with $\omega_n \to 0$ arbitrarily slowly.
Then, the assumptions on $\rho^{-1}(n)$ and $p_n$ imply $\rho^{-1}(n) = o(B_n)$ and $p_n = o(B_n)$. 
Note that the first condition in \eqref{eq:design-cond} implies $\E[\| \bX \|] = O(\sqrt{p_n})$ since $\E[\| \bX \| ] \le \sqrt{\E[\| \bX \|^2]}$ and 
$\E[\| \bX \|^2] = \text{tr}(\E[\bX \bX^T]) \le p_n \| \E[\bX \bX^T] \| = O(p_n)$.
Furthermore, note that $\rho(x)/x$ increasing implies $\rho(\lambda x) \le \lambda \rho(x)$ for $\lambda \le 1$ (since $\rho(\lambda x) / (\lambda x) \le \rho(x) / x$) and therefore $x_n = o(y_n) \implies \rho(x_n) = o(\rho(y_n))$.
We have
\begin{align*}
  n \Pr\lf(\| \bX  \bX^\top \|> B_n \ri) & = n \Pr\lf(\| \bX \|> \sqrt{B_n} \ri) \\
  & \le n \Pr\lf(\E[\| \bX \|] > \sqrt{B_n}/2 \ri) + n \Pr\lf( (\| \bX \| - \E[\| \bX \|]) > \sqrt{B_n}/2 \ri).
\end{align*}
It holds that $\Pr\lf(\E[\| \bX \|] > \sqrt{B_n}/2 \ri) = 0$ for $n$ large, since $\E[\| \bX \|] = O(\sqrt{p_n})$ and $p_n = o(B_n)$.
For the second term we obtain
\[
n \Pr\lf( (\| \bX \| - \E[\| \bX \|]) > \sqrt{B_n}/2 \ri) \le \frac{n \E[\rho(|\| \bX_{i} \| - \E[\| \bX_{i} \|]|^2)]}{\rho(B_n/4)} = O\left(\frac{n}{\rho(B_n/4)}\right) = o(1)
\]
since $\rho^{-1}(n) = o(B_n)$ implies $n = o(\rho(B_n/4))$.
Now the consistency result follows from \cref{theorem1}.
Since our choice of $H_n$ does not rely on $\btheta^*$, \ref{A:Cons1}\ref{eq:Hn-def} holds with $\btheta^*$ replaced by any $\btheta \in \Theta_n$, and the resulting estimator is unique by \cref{theorem1-uniqueness}.
The rate of convergence is $\sqrt{p_n/n}$, since $\tr(I(\btheta^*)) = O(p_n)$ follows from $\E[\phi_{}(\bX_{i}; \btheta^*)_k^4] = O(1)$.

\myparagraph{Asymptotic normality}
It suffices to verify \ref{A:Asymp} for each row $\ba_l, l = 1, \ldots, q$ of $A_n$, as $A_n \phi_{i}(\btheta)$ is a finite dimensional vector.
Note that $\| A_n \|=O(1)$ implies $\| \ba_l \| =O(1)$ for all $r$, since $\| \ba_l \| = \sup_{\| \bx \| = 1} |\ba_l^T \bx| \le \sup_{\| \bx \| = 1} \| A_n\bx \| = \| A_n \|$.
Using a Taylor expansion and boundedness of $\psi'(Y, \bX^\top \btheta)$, it suffices to use $\ba_l^\top \bX  \bX^\top \btu$ with $\btu \coloneq \bu - \bu'$ instead of $\ba_l^\top [\phi_{i}(\btheta^* +  \bu) - \phi_{i}(\btheta^* +  \bu')]$.
We obtain
\[
\E[ | \ba_l^\top \bX  \bX^\top \btu|^2	] 
= \E[ | \ba_l^\top \bX |^2 | \bX^\top \btu|^2] 
\le \sqrt{\E[|\ba_l^\top \bX|^4] \E[|\btu^\top \bX|^4]} = O(  \|\btu\|^2),
\]
verifying the first condition in \ref{A:Asymp} (with $p_n^2/n = o(1)$).
Similarly, one can show $\E[ | \ba_l^\top \bX  \bX^\top \btu|^k	] = O(  \|\btu\|^k)$.
For the second condition, set $D_n =  n \omega_n /\sqrt{p_n^3}$ with $\omega_n \to 0$ arbitrarily slowly, which satisfies $D_n = o(\sqrt{n}/(r_n p_n))$.
Markov's inequality gives
\begin{align*}
 &\quad \sumin \Pr\lf(\sup_{\|\bu\| \le r_n C}\frac{| \ba_l^\top \bX_{i} \bX_{i}^\top \btu| }{\|\btu\|}>  D_n \ri) 
 \le \frac{n \E[ | \ba_l^\top \bX  \bX^\top \btu|^k	]}{\|\btu\|^k D_n^k}
  = O\left(n^{1-k} p_n^{3/2 \, k} \omega_n^{-k} \right) = o(1),
 \end{align*}
 since $p_n^{3k}/(n^{2k - 2}) \to 0$.
Uniform boundedness of $\psi'$ implies that $J(\btheta) = \E[\nabla_{\btheta} \phi(\btheta)]$, so $J(\btheta) = \E[\psi'(Y, \bX^\top \btheta) \bX \bX^\top]$.
Using that $\psi'$ is Lipschitz, it holds that
\begin{align*}
  \| \ba_l^\top [\bar J(\btheta^*, \btheta^* +   \bu) - J(\btheta^*)] \| 
  & = \| \ba_l^\top \int_0^1 \E[\psi'(Y, \bX^\top [\btheta^* + t \bu])\bX \bX^\top - \psi'(Y, \bX^\top \btheta^*) \bX \bX^\top] dt \| \\
  & \le  \sup_{\| \btu\| = 1} \int_0^1 \E[|\ba_l^\top   \bX  \btu^\top \bX [\psi'(Y, \bX^\top [\btheta^* + t \bu]) - \psi'(Y, \bX^\top \btheta^*)] | ] dt \\
  & \lesssim \sup_{\| \btu\| = 1} \int_0^1 t dt \, \E[|\ba_l^\top   \bX \btu^\top \bX \bu^\top \bX |] \\ 
  &\le \sup_{\| \btu\| = 1} |\E[|\ba_l^\top   \bX|^3]^{1/3} \E[|\btu^\top \bX|^3]^{1/3} \E[|\bu^\top \bX|^3]^{1/3} | \\
  &= O(\|\ba_l\|\|\bu\|) = O(\sqrt{p_n/n}) = o(1/\sqrt{p_n}),
\end{align*}
since $p_n^2/n = o(1)$, verifying the third condition in \ref{A:Asymp}.

Finally, since  $\| A_n \phi_{i}(\btheta) \| ^4  \le \| A_n \|^4 \| \phi_{i}(\btheta) \| ^4$,
\[
 \| \phi_{i}(\btheta)  \|^4  = \lf(\sum_{k=1}^{p_n} \phi_{i}(\btheta)_k^2 \ri)^2 
 = \sum_{k=1}^{p_n} \sum_{k'=1}^{p_n} \lf( \phi_{i}(\btheta)_k \phi_{i}(\btheta)_{k'}  \ri)^2,
\]
and $\max_k \E[\phi_{}(\bX; \btheta^*)_k^4] = O(1)$, we have $\mathbb{E}\left[ \| A_n \phi_{i}(\btheta^*)\|^4 \right] = O(p_n^2) = o(n)$,
since $p_n^2 / n= o(1)$, verifying \ref{A:Cons2}. We have checked all conditions of \cref{theorem2}, which yields the claim.

\subsection{Proof of \autoref{cor:TS}}
First, recall from \cref{ex:timeseries}, with $0 < \delta < \gamma < 1$ with $\gamma + \kappa \delta > 1$,
\[
m_N = N^\gamma, \quad n = N/m_N = N^{1 - \gamma}, \quad L_N = N/(N^\gamma + N^\delta), \quad n - L_N \le N^{1 + \delta- 2 \gamma}.
\]
The last inequality follows from a Taylor expansion with $f(x) = N/(N^\gamma + x), f'(x) = - N/(N^\gamma + x)^2$ around $N^{\delta}$.
Applying \cref{lem:Yokoyama} with $r=2$ and $\max_k \E[| \varphi(Y_j, \bZ_j; \btheta^*)_k|^{4 + \eps}] = O(1)$ for some $\eps > 0$ and $\kappa > 2(4 + \eps )/\eps$ yields $\max_{i , k} \E[\phi_i(\btheta^*)_k^4]  = O(m_N^2/m_N^4)= O(1/m_N^2)$.

The rate of convergence is given by $r_n = \sqrt{\tr(I(\btheta^*)) / n} = \sqrt{p_n / N}$, since $\max_{i, k} \E[\phi_i(\btheta^*)_k^2] = O(1/m_N)$ and therefore
\[
\tr(I(\btheta^*)) =  \sum_{k=1}^{p_n} \frac 1 n \sumin \E[\phi_i(\btheta^*)_k^2] \le p_n \max_{i, k} \E[\phi_i(\btheta^*)_k^2] = O(p_n/m_N).
\]
\myparagraph{Consistency}
We need to verify \eqref{eq:assumpDep} with $T_n = \| \phi_i(\btheta^*) \|_\infty, E_n = r_n n /(\sqrt{p_n} (n-L_N)), r_n = \sqrt{p_n/N}$ and \ref{A:Cons1} with $B_n = O(n/(n-L_N))$.
For \eqref{eq:assumpDep}, the union bound and Markov's inequality yield, since $r_n n / \sqrt{p_n} = \sqrt{N}/m_N$, 
\begin{align*}
\sum_{i=L_N + 1}^{n} \Pr\left( \| \phi_i(\btheta^*) \|_\infty > \frac{\sqrt{N}}{m_N (n-L_N)} \right) 
& \le \sum_{i=L_N + 1}^{n} p_n \max_{1 \le k \le p_n}\Pr\left( |\phi_i(\btheta^*)_k | > \frac{\sqrt{N}}{m_N (n-L_N)}  \right) 
\\ & \le \frac{(n - L_N)^5 p_n m_N^4 \max_{i , k} \E[\phi_i(\btheta^*)_k^4]}{N^2} \\
& = O\left(\frac{(n - L_N)^5 p_n m_N^4 }{N^2 m_N^2} \right) = O\left(\frac{(n - L_N)^5 p_n  }{n^2} \right) \\
& = O(p_n N^{(5 (1 + \delta - 2 \gamma) - 2 (1 - \gamma))} )= O(p_n N^{(3 +5\delta - 8 \gamma )})= o(1),
\end{align*}
since $p_n = o(N^{(-3 - 5 \delta + 8 \gamma)})$ follows from choosing $\gamma$ arbitrarily close to 1 and $\delta > (1 - \gamma)/\kappa$ arbitrarily close to 0 and $p_n / N \to 0$. 

Since $\psi'$ is non-positive and $\nabla_{\btheta} \psi(Y_j, \bZ_j^\top \btheta)\bZ_j = \psi'(Y_j, \bZ_j^\top \btheta) \bZ_j \bZ_j^\top$, it holds that
\[
\bu^\top [\phi_i(\btheta^* + \bu) - \phi_i(\btheta^*)] \le \bu^\top \left[ - \frac 1 {m_N} \textstyle \sum_{j \in I_{i,N}} \inf_{\btheta \in \Theta_n} | \psi'(Y_j, \bZ_j^\top \btheta)| \bZ_j \bZ_j^\top \right] \bu,
\]
so \ref{A:Cons1}\ref{eq:Hn-def} holds with $H_n(\bX_i) = - m_N^{-1} \sum_{j \in I_{i,N}} \inf_{\btheta \in \Theta_n} | \psi'(Y_j, \bZ_j^\top \btheta)| \bZ_j \bZ_j^\top$.
This choice satisfies \ref{A:Cons1}\ref{eq:Hn-identifiable} by \eqref{eq:Mest-eigcond}.
For \ref{eq:Hn-bounds}, $\inf_{\btheta \in \Theta_n} | \psi'(Y_j, \bZ_j^\top \btheta)|$ is negligible since this term is bounded.
Now the first condition in \ref{A:Cons1}\ref{eq:Hn-bounds} follows from
\begin{align*}
  & \quad \sup_{\| \bu \| = 1} \frac{1}{n}  \sumin    \E\left[ (\bu^\top (m_N^{-1}\textstyle \sum_{j \in I_{i,N}} \bZ_j \bZ_j^\top) \bu)^2 \right] \displaystyle
    = \sup_{\| \bu \| = 1} \frac{1}{n}  \sumin \E[ (m_N^{-1} \textstyle  \sum_{j \in I_{i,N}} (\bu^\top \bZ_j)^2)^2 ]  \\
   & = \sup_{\| \bu \| = 1} \frac{1}{n} m_N^{-2} \sumin \sum_{j \in I_{i,N}} \sum_{j' \in I_{i,N}} \E[(\bu^\top \bZ_j)^2 (\bu^\top \bZ_{j'})^2] = O(1) = o(n),
\end{align*}
since 
\[
\E[(\bu^\top \bZ_j)^2 (\bu^\top \bZ_{j'})^2] \le \sqrt{\E[(\bu^\top \bZ_j)^4] \E[(\bu^\top \bZ_{j'})^4]} = O(\| \bu \|^4) = O(1)
\]
uniformly for all $j, j'$ by \eqref{eq:design-cond}.
Set $B_n = n  \omega_n / \ln p_n$ with $\omega_n \to 0$ arbitrarily slowly.
This $B_n$ satisfies $B_n = O(n/(n - L_N))$ since $n - L_N \to 0$ as we choose $\gamma$ arbitrarily close to 1.
Now the second condition in \ref{A:Cons1}\ref{eq:Hn-bounds} follows from \cref{lem:semi-definite}.
For the third condition, it holds that
\begin{align*}
   & \quad \sumin \Pr (\| m_N^{-1} \sum_{j \in I_{i,N}} \bZ_j \bZ_j^\top \| > B_n ) \\
   & \le  \sumin \Pr (\|  \E[ \bZ \bZ^\top ] \| > B_n/2 )  +  \sumin \Pr (\| m_N^{-1} \textstyle \sum_{j \in I_{i,N}} \bZ_j \bZ_j^\top -  \E[ \bZ \bZ^\top ] \| > B_n/2 ).
\end{align*}
The first term is 0 for $n$ large since $\|  \E[ \bZ \bZ^\top ] \| = O(1)$ and $B_n \to \infty$ because $\ln p_n / n \to 0$.
For the second term, for each $i = 1, \ldots n$, we build $m_N / \tilde m_N$ blocks of size $\tilde m_N$ of consecutive observations, where $\tilde m_N = N^\alpha$ with $\alpha > 1/(1 + \kappa)$.
Let $\tilde I_{k, i, N}, k = 1, \ldots m_N / \tilde m_N$ denote the indices of block $k$ of $I_{i,N}$ and $Y_{k,i} = \tilde m_N^{-1} \sum_{j \in \tilde I_{k,i,N}} \bZ_j \bZ_j^\top -  \E[ \bZ \bZ^\top]$.
Now
\begin{align*}
   & \quad \sumin \Pr (\| m_N^{-1} \textstyle \sum_{j \in I_{i,N}} \bZ_j \bZ_j^\top -  \E[ \bZ \bZ^\top ] \| > B_n/2 ) \\
   & \le  \sumin \Pr (\|  \tilde m_N m_N^{-1} \textstyle \sum_{\text{$k$ even}} Y_{k,i}  \| > B_n/4 ) + \displaystyle  \sumin \Pr (\|   \tilde m_N m_N^{-1} \textstyle \sum_{\text{$k$ odd}} Y_{k,i}  \| > B_n/4 ).
\end{align*}
We show that the first term is $o(1)$, the result for the second follows in the same way.
Similar to the proof of \cref{lem:dep}, we can construct an independent sequence $Y_{2,i}^*,  Y_{4,i}^*, \ldots$ satisfying $Y_{k,i}^* =_d Y_{k,i}$ and $\max_{\text{$k$ even}} \P(Y_{k,i}^* \neq Y_{k,i}) \le \tilde m_N^{- \kappa}$.
Now
\begin{align*}
 & \quad \sumin \Pr (\|  \tilde m_N m_N^{-1} \textstyle \sum_{\text{$k$ even}} Y_{k,i}  \| > B_n/4 )  \\
 & \le \sumin \Pr (\|  \tilde m_N m_N^{-1} \textstyle \sum_{\text{$k$ even}} Y_{k,i}^*  \| > B_n/8 ) 
 + \displaystyle\sumin \Pr (\|  \tilde m_N m_N^{-1} \textstyle \sum_{\text{$k$ even}} Y_{k,i} -  Y_{k,i}^*  \| > B_n/8 ) \\
 & \le \sumin \Pr (\|  2 \tilde m_N m_N^{-1} \textstyle \sum_{\text{$k$ even}} Y_{k,i}^*  \| > B_n/4 ) 
 + \displaystyle\sumin \Pr (\textstyle \bigcup_{\text{$k$ even}} \{ Y_{k,i}^* \neq Y_{k,i}\}) .
\end{align*}
The second term is $\lesssim n m_N / \tilde m_N \tilde m_N^{-\kappa} = N \tilde m_N^{-1 - \kappa} = N^{1 - \alpha (1 +  \kappa)} = o(1)$ by the condition on $\alpha$.
For the first term, note that the definition of $b_n$ implies
\[
\textstyle \max_{i, k} \| Y_{i,k} \| \le \frac{1}{\tilde m_N} \sum_{j \in \tilde I_{k,i,N} } \| \bZ_j \bZ_j^\top - \E[\bZ \bZ^\top] \|
\lesssim \max_j \| \bZ_j \bZ_j^\top \| =  \| \bZ \|^2 \le b_n \quad \text{a.s.}
\]
Denote $X_j = \bZ_j \bZ_j^\top -  \E[ \bZ \bZ^\top]$ and
\begin{align*}
  \sigma_n^2 & \coloneq \max_{i} \frac{2 \tilde m_N}{m_N} \| \sum_{\text{$k$ even}} \E[Y_{k,i}^2]  \|
  \le \max_{i, k} \|  \E[Y_{k,i}^2] \| \le  \frac{1}{\tilde m_N^2}\max_{i, k} \sum_{j \in \tilde I_{k,i,N}} \sum_{j' \in \tilde I_{k,i,N}} \| \E[X_j X_{j'}] \| \\
  & \le \max_j \| \E[X_j^2] \| = \| \E[(\bZ \bZ^\top - \E[\bZ \bZ^\top])^2] \le  \| \E[(\bZ \bZ^\top)^2] \| = \| \E[\|\bZ\|^2 \bZ \bZ^\top] \| \\
  & = \sup_{\| \bu \| = 1}\E[\|\bZ\|^2 \bu^\top \bZ \bZ^\top\bu]
  \le \sup_{\| \bu \| = 1} \sqrt{\E[\| \bZ \|^4] \E[(\bu^\top \bZ)^4]} = O(\sqrt{\E[\| \bZ \|^4]}) = O(p_n),
\end{align*}
since $\max_{1 \le k \le p_n}\E[Z_k^4] = O(1)$.
Now, the Bernstein inequality for random matrices \citep[Theorem 6.17]{WWbook} yields
\begin{align*}
  \sumin \Pr (\|  2 \tilde m_N m_N^{-1} \textstyle \sum_{\text{$k$ even}} Y_{k,i}^*  \| > B_n/4 ) 
  \lesssim  \exp\left(\ln(n p_n) - \frac{m_N B_n^2}{\tilde m_N (\sigma_n^2 + b_n B_n)}\right) = o(1),
\end{align*}
since 
\[
\frac{\ln(n p_n) \tilde m_N p_n}{m_N B_n^2}
= \frac{\ln(n p_n) N^\alpha p_n (\ln (p_n))^2}{N n \omega_n} \to 0, \quad 
\frac{\ln(n p_n) \tilde m_N b_n}{m_N B_n} 
= \frac{\ln(n p_n) N^\alpha b_n \ln (p_n)}{N \omega_n}  \to 0.
\]
Now consistency with $r_n = \sqrt{p_n/N}$ follows from \cref{theorem1}.
Uniqueness of $\hbtheta$ follows from \cref{theorem1-uniqueness}, since $H_n$ does not rely on $\btheta^*$. 

\myparagraph{Asymptotic normality}
We set $\gamma = 1/2$, so $n = m_N = \sqrt{N}$, and $\delta \le 1/10$ small enough such that $\delta > (1 - \gamma)/\kappa$. 
Such $\delta$ exists since $\kappa > 5$.
Now
\[ 
(n - L_N)^5 / n \le N^{(5 + 5\delta - 10 \gamma)}/N^{(1-\gamma)} = N^{( 5 \delta - 1/2)} = O(1).
\]
It remains to prove \ref{A:Asymp} with $D_n = O(\sqrt{n}/(n- L_N))$ and \ref{A:Cons2} with some matrix $B_n = \sqrt{m_N} A_n$ with $\| A_n \|_\infty =O(1)$.
For \ref{A:Cons2}, we use $\| B \bx \|_2 \le \sqrt{q} \| B \bx \|_\infty \le \sqrt{q} \| B \|_\infty \| \bx \|_\infty$ for a $q \times p$ dimensional matrix $B$, so
\begin{align*}
  \max_i \E\left[ \| B_n \phi_{i}(\btheta^*)\|^4 \right] \lesssim m_N^2 \| A_n \|_\infty^4 \max_i \E\left[ \| \phi_{i}(\btheta^*)\|^4_\infty \right]
  \lesssim m_N^2 p_n \max_{i, k} \E\left[  \phi_{i}(\btheta^*)_k^4 \right] = o(n)
\end{align*}
since $p_n = o(n)$.
Similar to the proof of \cref{cor:Mest}, we verify \ref{A:Asymp} for each row $\bb_l, l = 1, \ldots, q$ with $\bb_l = \sqrt{m_N} \ba_l$, $\| \ba_l \| =O(1)$ (which follows from $\| A_n\|_\infty = O(1)$).
We choose $\tilde \phi_i(\btheta) = \phi_i(\btheta) - \E[ \phi_i(\btheta) ]$ to obtain sharp bounds.
We may ignore $\psi'(Y, \bZ^\top \btheta)$ since its absolute value is uniformly bounded and work with $\frac{\sqrt{m_N}}{m_N} \ba_l^\top  \textstyle \sum_{j \in I_{i,N}} ( \bZ_j \bZ_j^\top - \E[\bZ \bZ^\top]  ) \bu$.

It holds that $\E[| \ba^\top \bZ |^{4r + 2 \eps'}] = O(\| \ba \|^{4r + 2 \eps'})$ with $4r + 2 \eps' = \tau$, $r > 1$ by assumption, which implies $\E[| \ba_l^\top \bZ \bZ^\top \bu |^{2r + \eps' }] = O(\| \ba_l \|^{2r + \eps' } \| \bu \|^{2r + \eps' })$.
To apply \cref{lem:Yokoyama} ($\max_i \E[|\ba_l^\top \textstyle \sum_{j \in I_{i,N}} ( \bZ_j \bZ_j^\top - \E[\bZ \bZ^\top]  ) \bu |^{2r}] = O(m_N^r \| \ba_l \|^{2r} \| \bu \|^{2r})$) with the largest possible $r$, we maximize $r$ under the constraint that $\kappa > r(2r + \eps')/\eps'$:
We have $\eps' = \tau/2 - 2r$.
Plugging this into the constraint gives
\[
\kappa > \frac{r(2r + \tau/2 - 2r)}{\tau/2 - 2r} \ \Leftrightarrow \ r \tau / 2 < \kappa \tau/2 - 2 \kappa r\ \Leftrightarrow \ r(\tau + 4 \kappa) < \kappa \tau \ \Leftrightarrow \ r < \frac{\kappa \tau}{\tau + 4 \kappa}.
\]
The first condition in \ref{A:Asymp} follows since $r_n^2 p_n = p_n^2/N \to 0$ and
\[
  \max_{1 \le i \le n }\frac{1}{m_N} \E\Big[ | \ba_l^T \sum_{j \in I_{i,N}} ( \bZ_j \bZ_j^\top - \E[\bZ \bZ^\top] ) \bu|^2\Big] = O(\| \bu \|^2).
\]
For the second condition, set $D_n = \min\{ \sqrt{n}/(n - L_N), \sqrt{n N/p_n^3} \, \omega_n\}$ with $\omega_n \to 0$ arbitrarily slowly.
This choice satisfies both $D_n = O(\sqrt{n}/(n-L_N))$ and $D_n = o(\sqrt{n}/(r_n p_n))$.
It holds that
\begin{align*}
  & \quad \sumin \P\left(\sup_{\| \bu \| \le r_n C} \frac{m_N^{-1/2}|\ba_l^\top  \textstyle \sum_{j \in I_{i,N}} ( \bZ_j \bZ_j^\top - \E[\bZ \bZ^\top] )  \bu |}{\| \bu \|} > D_n\right) \\
  & \le \frac{n \, m_N^{-r} \max_i \E[|\ba_l^\top \textstyle \sum_{j \in I_{i,N}} ( \bZ_j \bZ_j^\top - \E[\bZ \bZ^\top]  ) \bu |^{2r}]}{\| \bu \|^{2r} D_n^{2r}}
  = O(n/D_n^{2r}) \\
  & = O\left(\frac{\max\{ (n - L_N)^{2r}, p_n^{3r}/(N^r \omega_n^{2r}) \}}{n^{r-1}}  \right)
  = O\left(\max\left\{N^{2r \delta - (r-1)/2} , \frac{ p_n^{3r}}{N^{r + (r-1)/2} \omega_n^{2r}}\right\} \right) = o(1)
\end{align*}
since $p_n^{3r} / N^{(1.5 r - 1/2)} = o(1)$ and $r/\kappa - (r-1)/2 < 0$ (and $\delta > 1/(2 \kappa)$), which follows from $r > \kappa / (\kappa - 2)$.
This is implied by $\kappa > 3 \tau/(\tau - 4)$, since
\[
\frac{\kappa}{\kappa - 2} < \frac{\kappa \tau}{\tau + 4 \kappa} \ \Leftrightarrow \ \kappa \tau + 4 \kappa^2 < \kappa^2 \tau - 2 \kappa \tau 
\ \Leftrightarrow \ 3 \kappa \tau  < \kappa^2 (\tau - 4) \ \Leftrightarrow \ \frac{3 \tau}{\tau - 4} < \kappa.
\]
For the third condition in \ref{A:Asymp}, note that $1/(\sqrt{n} r_n) = \sqrt{m_N / p_n}$.
Since $J(\btheta) = \E[\psi'(Y, \bZ^\top \btheta) \bZ \bZ^\top]$, the proof of \cref{cor:Mest} gives 
\[
\| \bb_r^\top [\bar J(\btheta^*, \btheta^* +   \bu) - J(\btheta^*)] \| 
   = O(\|\bb_r\|\|\bu\|) = O(\sqrt{m_N} \| \ba_l\| \sqrt{p_n / N} ) = o(\sqrt{m_N / p_n})
\]
since $p_n^2/N \to 0$.
Now \cref{th:dependence} and \cref{theorem2} imply the claim.

\subsection{Proof of \autoref{cor:MestP}}

We apply  Theorems \ref{theorem4}--\ref{theorem6}. The rate of convergence in \cref{theorem4}  is $\tilde r_n = \sqrt{s_n \ln p_n /n} $, since $b^*_n = \lambda_n$ and $\lambda_n = O(\sqrt{\ln p_n / n})$.
\begin{itemize}
  \item The conditions on the penalty \ref{A:Penalty2}--\ref{A:Penalty4} are always satisfied for the Lasso and \ref{A:Penalty1} holds since $\min_{k : \theta^*_k \neq 0 }| \theta_{k}^*|/ \tilde r_n \to \infty$. 
  \item That the reduced problem \eqref{eq:Def_M_reduced} satisfies \ref{A:Cons1} follows from the proof of \cref{cor:Mest}.
    
  \item For \ref{A:phi-moments}, our assumptions give $\sigma_n = \sigma$ which is bounded away from zero and infinity. Further the union bound, Markov's inequality, \eqref{eq:design-cond-pen}, and \eqref{eq:MestP-dim-cond} imply
  \begin{align*}
    \sumin \Pr\lf(\| \phi_{i}(\btheta^*)\|_\infty >\sqrt{n \sigma^2/4\ln p_n} \ri) 
    &\le n p_n \max_{k} \Pr\lf( \phi_{i}(\btheta^*)_k^2 > n \sigma^2/4\ln  p_n \ri) \\
    &\le  \frac{n p_n \max_{k}\E[\rho( \phi_{i}(\btheta^*)_k^2)]}{\rho( n \sigma^2/4\ln  p_n)} = O\lf(\frac{n p_n}{\rho(n/ \ln p_n)}  \ri) = o(1).
  \end{align*}
    
  \item \eqref{eq:Mest-mutual-incoherence} implies \ref{A:lambda} (see \cref{sec:mutual}), and \ref{A:lambda2} holds with the proposed choice of $\lambda_n$ since
  \begin{align*}
   & \quad \max_{k  =s_n + 1, \ldots, p_n } \sup_{\btheta \in \Theta'_n, \| \bx \|_\infty \le 1 } \|\bar J(\btheta^*, \btheta)_{k,(1)} \;\bar J(\btheta^*, \btheta)^{-1}_{(1)}  \bx \|_2  
   \\ & = \max_{k  =s_n + 1, \ldots, p_n } \sup_{\btheta \in \Theta'_n} \| (\bar J(\btheta^*, \btheta)_{k,(1)} \;\bar J(\btheta^*, \btheta)^{-1}_{(1)})^T \|_1 = \sup_{\btheta \in \Theta'_n} \lf\|   \bar J(\btheta^*, \btheta)_{(2,1)} \;\bar J(\btheta^*, \btheta)^{-1}_{(1)} \ri\|_{\infty} \le \alpha .
  \end{align*}

  \item Observe that for $|\psi'| \le K$ and $\be_k$ the $k$th unit vector, we have
  \begin{align*}
    | \phi_{i}(\btheta)_k - \phi_{i}(\btheta')_k | \le K |\be_{k}^\top \bX_{i} | |\bX_{i(1)}^\top (\btheta_{(1)} - \btheta_{'(1)})|,
  \end{align*}
  using that $\btheta_{(2)} = \0$ for $\btheta \in \Theta_n'$.
  By our design conditions, we get 
  \begin{align*}
    \max_{1 \le k \le p_n}\E[| \phi_{i}(\btheta)_k - \phi_{i}(\btheta')_k |^2 ] = O(\|\btheta - \btheta'\|^2),
  \end{align*}
  as in the proof of \cref{cor:Mest}, so the first condition of \ref{A:Penalty_emp_pr2} holds because $(s_n^2 + s_n \ln p_n)/n  = o(1)$.
  For the second condition, choose $\tD_n = K\sqrt{s_n} \rho^{-1}(np_n \omega_n)$ with $\omega_n \to \infty$ arbitrarily slowly. It holds that
  \begin{align*}
    \sumin \Pr\lf( \sup_{\btheta, \btheta' \in \Theta_n'}  \frac{\| \phi_{i}(\btheta) - \phi_{i}(\btheta') \|_\infty}{\| \btheta - \btheta' \|} > \tD_n \ri) 
    &\le n \Pr\lf(K \|\bX\|_\infty \|\bX_{(1)}\| > \tD_n \ri) \\
    &\le n \Pr\lf(K \sqrt{s_n} \|\bX\|_\infty^2  > \tD_n \ri) \\
    &\le n p_n \max_{1 \le k \le p_n} \Pr\lf(|X_{ k}|^2  > \tD_n / (K \sqrt{s_n})  \ri) \\
    &\le O\lf(\frac{n p_n \max_{1 \le  k \le p_n}\E[\rho( X_{k}^2)] }{\rho( \tD_n / (K \sqrt{s_n}))}\ri) \\
    &= O\lf(\frac{n p_n}{\rho( \rho^{-1}(n p_n \omega_n) )}\ri)  = o(1).
  \end{align*}
  This choice satisfies
  \begin{align*}
    \frac{\tD_n  \tilde r_n(s_n + \ln p_n)}{n \eta_n } \le  \frac{\tD_n  \sqrt{s_n} (s_n+ \ln p_n)}{2\sigma n}= \frac{K}{2\sigma} \frac{ (s_n^2 + s_n \ln p_n) \rho^{-1}(n p_n \omega_n)}{ n} =o(1),
  \end{align*}
  since \eqref{eq:MestP-dim-cond} and $\omega_n \to \infty$ arbitrarily slowly imply $n p_n \omega_n \le \rho(n/(s_n \ln p_n)^2)$.

  \item The first two conditions in \ref{A:Asymp} can be verified as in the proof of \cref{cor:Mest} by the choice $D_n =n/\sqrt{s_n^3 \ln p_n} \omega_n$ and $\omega_n \to 0$ arbitrarily slowly. 
  For the second condition, $n / D_n^k \to 0$ follows from $s_n^{3k} (\ln p_n)^k/n^{(2k - 2)} \to 0$.
  Similarly, it follows
  \begin{align*}
    \sup_{\| \bu\|  \le \tilde r_n C} \| A_n[\bar J(\btheta^*, \btheta^* +   \bu) - J(\btheta^*)] \|   = O(\tilde r_n) = o\lf(\frac{1}{\sqrt{n} \tilde r_n}\ri), 
  \end{align*}
  because $\tilde r_n^2 \sqrt{n} =  s_n \ln p_n / \sqrt{n} = o(1)$ by \eqref{eq:MestP-dim-cond}.
  Assumptions \ref{A:Cons2} is verified exactly as in \cref{cor:Mest}, so asymptotic normality of $\hbtheta_{(1)}$ follows from \cref{theorem6}.

  \item To verify the conditions of \cref{theorem5}, we can choose $H_n$ as in \cref{cor:Mest}.
  This construction is independent of $\btheta^*$.
  Since $|\psi'|$ is bounded away from zero, the eigenvalue condition implies that \ref{A:phi-H-penalty}\ref{eq:Hn-def2} and \ref{eq:Hn-identifiable2} hold for all $\btheta \in \Theta_n$.
  It remains to verify \ref{A:phi-H-penalty}\ref{eq:Hn-bounds2} with $\nu_n = s_n$.
  As $\max_{1 \le k \le p_n} \E[X_{i, k}^4] = O(1)$ by \eqref{eq:design-cond}, the first condition holds since $s_n^2 \ln p_n / n =o(1)$.
  For the second condition, choose $\tilde B_n = \rho^{-1}(n p_n \omega_n)$ with $\omega_n \to \infty$ arbitrarily slowly. 
  Then
  \begin{align*}
  \sumin \Pr \lf(\max_{1 \le j, k \le p_n} |X_{i,j} X_{i,k}| > \tilde B_n \ri)
  &\le \sumin \Pr \lf(\max_{1 \le k \le p_n} X_{i,k}^2 > \tilde B_n \ri) \\
  & \le n p_n \max_{1 \le  k \le p_n}  \Pr \lf( X_{k}^2 > \tilde B_n \ri) \\
  & \le \frac{n p_n \max_{1 \le  k \le p_n}\E[\rho( X_{k}^2)]}{\rho(\tilde B_n)} = O\lf(\frac{n p_n}{n p_n \omega_n}\ri) = o(1),
  \end{align*}
  using \eqref{eq:design-cond-pen}.
  Since \eqref{eq:MestP-dim-cond} and $\omega_n \to \infty$ arbitrarily slowly implies $n p_n \omega_n \le \rho(n/(s_n \ln p_n)^2)$, this choice satisfies $\tilde B_n = o(n/(s_n \ln p_n))$.
\end{itemize}

\subsection{Proof of \autoref{cor:MestSCAD}}

Since $(\min_{g \in I_{(1)}} \| \btheta^*_{G_g} \|_2 - \sqrt{s_n / n})/\lambda_n \to \infty$ and $\sqrt{n/s_n} \min_{g \in I_{(1)}} \| \btheta^*_{G_g} \|_2 \to \infty$, we have $b_n^* = 0, \tilde r_n = \sqrt{s_n/n}$ and \ref{A:Penalty1}, \ref{A:Penalty2} and \ref{A:Penalty4} hold automatically.
It remains to verify \ref{A:Penalty3} for the nonconvex part, i.e., for $\lambda_n < \| \btheta_{G_g} \| \le a \lambda_n$, where the gradient is given by $(a \lambda_n -  \| \btheta_{G_g} \| )/(a - 1) \cdot \btheta_{G_g}  / \| \btheta_{G_g} \| $.
One can show that \ref{A:Penalty3} is implied by convexity of $f(\btheta_{G_g}) = p_{\lambda_n}(\| \btheta_{G_g}\|) + \btheta_{G_g}^\top \btheta_{G_g} \mu_n/2$ ($\mu$-amenability, \citet{LohWW2} and \citet{Loh}), where $p_{\lambda_n} (\theta)$ denotes the (one-dimensional) SCAD.
Computing the Hessian and its eigenvalues, one can show that group SCAD satisfies \ref{A:Penalty3} with $\mu_n = (a-1)^{-1}$. 

That the reduced problem \eqref{eq:Def_M_reduced} satisfies \ref{A:Cons1} follows from the proof of \cref{cor:Mest}.
\ref{A:phi-moments} follows from \eqref{eq:design-cond-pen} and $p_n = o(\rho(n/\ln p_n)/n)$.
\ref{A:lambda} is trivial for Group SCAD (with $\alpha = 0$), since $\nabla_{\btheta_{(1)}} p_{\blambda_n}(\btheta)  = \bnull$.
\ref{A:lambda2} holds with the proposed choice of $\lambda_n$.

\ref{A:Penalty_emp_pr2} can be verified similar to the proof of \cref{cor:MestP}. 
  The first condition holds since $s_n(s_n + \ln p_n)/(n \eta_n)^2 \to 0$.
  For the second condition, $p_n = o(\rho(n^{3/2} \eta_n)/(s_n^2 + s_n \ln p_n)/n)$ implies $\tilde D_n \tilde r_n (s_n + \ln p_n)/(n \eta_n) = o(1)$.
  
The first condition in \ref{A:Asymp} holds since $s_n = o(\sqrt{n})$, see the proof of \cref{cor:Mest}.
  The second condition holds with $D_n  = n/\sqrt{s_n^3}\omega$, $\omega_n \to \infty$ arbitrarily slowly, since $n/D_n^k \to 0$ follows from $s_n^{3k}/n^{(2k - 2)} = o(1)$. 
  The third condition holds since $\tilde r_n^2 \sqrt{n} = s_n/\sqrt{n} = o(1)$.
  \ref{A:phi-H-penalty} can be verified exactly as in the proof of \cref{cor:MestP} using $p_n = o(\rho(n/s_n \ln p_n)/n)$.

\subsection{Proof of \autoref{cor:distributed}}

We first verify \ref{A:Cons1}.
Since
\[
 \nabla_{\btheta}\phi_{i}(\btheta) =   \begin{pmatrix}
  	 \mathrm{diag} \lf( K_n \ind \{k_{i} = k \}\psi'(\bX_{i}; \theta_k) \ri)_{k = 1, \ldots, K_n}  &  \0 \\
     \frac{1}{K_n} \bm 1^\top  & -1
  \end{pmatrix},
\]
and there is some $c > 0$ such that $\sup_{\theta \in \Theta_0, \bx} \psi'(\bx; \theta) \le -c$,
we can choose
\begin{align*}
 H_n(\bx_{i}, k_{i}) =   \begin{pmatrix}
  	 \mathrm{diag} \lf( -\ind \{k_{i} = k \} K_n c\ri)_{k = 1, \ldots, K_n}  + I_{K_n}/\sqrt{4 K_n} &  \0 \\
   \0  & -1 +  1/\sqrt{4 K_n} 
  \end{pmatrix}
\end{align*}
by \cref{lem:diagonal-H-cover}.
For $n$ large, both $ I_{K_n}/\sqrt{4 K_n}$ and $1/\sqrt{4 K_n}$ are negligible.
Then, since $n_k = n / K_n$,
\begin{align*}
 \frac 1 n \sumin \E\lf[H_n(\bX_{i}, k_{i}) \ri] =   \begin{pmatrix}
  	 -c I_{K_n}  &  \0 \\
   \0  & -1 
  \end{pmatrix}, \
  \| H_n(\bx_i, k_i) \| = O(K_n), \ \| H_n(\bx_i, k_i )^2 \| = O(K_n^2).
\end{align*}
It therefore holds that $\limsup_{n \to \infty} \lambda_{\max}(n^{-1} \sumin  \E[H_n(\bX_{i}) ]) \le \max\{ -1,  -c \} < 0$.
We can choose $B_n = K_n C$ in \ref{A:Cons1} with some large enough $C$. Then, \ref{A:Cons1} is satisfied since $K_n^2 \ln K_n / n \to 0$.
  The given $H_n(\bx_{i}, k_{i}) $ is valid for all $\btheta \in \Theta_0^{K_n + 1}$, so the solution is also unique by \cref{theorem1-uniqueness}.
Next, we have
\begin{align*}
  I(\btheta^*) = K_n \mathrm{diag}\lf( \E[\psi(\bX_{i}; \theta^*_1)^2], \dots, \E[\psi(\bX_{i}; \theta^*_{K_n})^2], 0\ri),
\end{align*}
which implies that the convergence rate of the stacked parameter vector $\hbtheta$ is $\sqrt{\tr(I(\btheta^*)) / n} = \sqrt{K_n^2/n} = \sqrt{K_n/n_1}$.
We now verify \ref{A:Asymp} and \ref{A:Cons2}.
Uniform boundedness of $\psi'$ implies that $J(\btheta) = \frac 1 n \sumin \E[\nabla_{\btheta} \phi(\btheta)]$, so
\begin{align*}
  J(\btheta) =
  \begin{pmatrix}
    \mathrm{diag}\lf( \E[\psi'(\bX_{i}; \theta_1)], \dots, \E[\psi'(\bX_{i}; \theta_{K_n})] \ri) & \0 \\
    \frac{1}{K_n} \bm 1^\top                                                                          & - 1
  \end{pmatrix},
\end{align*}
for which the block inversion formula yields
\begin{align*}
  J(\btheta)^{-1} =
  \begin{pmatrix}
    \mathrm{diag}\lf( \E[\psi'(\bX_{i}; \theta_1)]^{-1}, \dots, \E[\psi'(\bX_{i}; \theta_{K_n})]^{-1} \ri)         & \0 \\
    \frac{1}{K_n}  \vec\lf(\E[\psi'(\bX_{i}; \theta_{1})]^{-1}, \dots,   \E[\psi'(\bX_{i}; \theta_{K_n})]^{-1}\ri) & -1
  \end{pmatrix}.
\end{align*}
Choosing $A_n = \ba_n^\top J(\btheta^*)^{-1} \in \R^{1 \times p_n}$ with $\ba^\top_n = (0, \dots, 0, 1)$ gives the statement for $ \htheta_{ K_n + 1} $, as
\begin{align*}
A_n = \ba_n^\top J(\btheta^*)^{-1} = \lf(-\frac{1}{K_n}  \E[\psi'(\bX_{i}; \theta^*_1)]^{-1}, \dots, -\frac{1}{K_n}  \E[\psi'(\bX_{i}; \theta^*_{K_n})]^{-1},- 1\ri),
\end{align*}
and therefore
\begin{align*}
  \ba_n^\top J(\btheta^*)^{-1} I_{}(\btheta^*) J(\btheta^*)^{-\top} \ba_n =  \frac 1 {K_n} \sum_{k = 1}^{K_n} \frac{\E[\psi(\bX_{i}; \theta^*_k)^2] }{\E[\psi'(\bX_{i}; \theta^*_k)]^2}.
\end{align*}
Using the formula for $\nabla_{\btheta} \phi_i(\btheta)$, we obtain
 \begin{align*}
& \quad A_n [\phi_{i}(\btheta^* + \bu) -   \phi_{i}(\btheta^* + \bu')] \\
 & = u_{K_n + 1} - u'_{K_n + 1} - \frac{1}{K_n} \sum_{k=1}^{K_n}  (u_k - u'_k) - \sum_{k=1}^{K_n} \frac{\ind\{k_{i} = k\} \psi'(\bX_{i}; \theta_k^* +  \tilde u_k) (u_k - u'_k)}{\E[\psi'(\theta_k^*)]} \\ & = O(\| \bu - \bu' \|),
 \end{align*}
 since $\psi'$ is bounded away from 0 and $- \infty$.
 Choosing $D_n = C$ with some large enough $C$, the first two conditions in \ref{A:Asymp} are satisfied since 
 $r_n^2 p_n = K_n^3 / n \to 0$ and $\sqrt{n}/(r_n p_n) = n/K_n^2 \to \infty$.
 \begin{align*}
& A_n [\bar J(\btheta^*, \btheta^* + \bu) - J(\btheta^*)] \\ &= -\frac{1}{K_n}  \lf( \frac{\int_0^1\E[ \psi'(\theta_1^* + t u_1) - \psi'(\theta_1^*)] dt}{\E[\psi'(\theta_1^*)]},\ldots, \frac{\int_0^1 \E[\psi'(\theta_{K_n}^* + t u_{K_n}) - \psi'(\theta_{K_n}^*)] dt}{\E[\psi'(\theta_{K_n}^*)]}, 0\ri) \lesssim \frac{\| \bu^\top \|}{ K_n}
 \end{align*}
since $\psi'$ is Lipschitz. 
As $\| \bu^\top \|  = \| \bu \|$, we have $\| A_n [\bar J(\btheta^*, \btheta^* + \bu) - J(\btheta^*)] \| = O(r_n / K_n)$, so the third condition in \ref{A:Asymp} holds since $(\sqrt{n} r_n)^{-1} = K_n^{-1}$
We further have 
\[
A_n \phi_{i}(\btheta^*) =  -\sum_{k=1}^{K_n} \frac{\ind\{k_{i} = k\}\, \psi(\bX_{i}; \theta_k^*)}{\E[\psi'(\bX_i; \theta_k^*)]} - \frac 1 {K_n} \sum_{k=1}^{K_n} \theta_k^* + \theta_{K_n + 1}^* =   -\sum_{k=1}^{K_n} \frac{\ind\{k_{i} = k\}\, \psi(\bX_{i}; \theta_k^*)}{\E[\psi'(\bX_i; \theta_k^*)]}.
\]
This implies that $ \max_{1 \le i \le n}\mathbb{E}\left[ \| A_n \phi_{i}(\btheta^*)\|^4 \right] = O(1 )$ using that $\max_{i, k}\E[\psi(\bX_{i}; \theta_k^*)^4] = O(1)$ since $\psi$ is bounded and $\E[\psi'(\bX_i; \theta_k^*)]$ is bounded away from $0$. This verifies \ref{A:Cons2}.

To obtain the asymptotic distribution of $\sqrt{n/K_n}(\hat{\theta}_k - \theta_k^*)$, choose $\ba_n^\top = K_n^{-1/2}\be_k^\top$ and
\[
A_n = \ba_n^\top J(\btheta^*)^{-1} = K_n^{-1/2}  \E[\psi'(\bX_{i}; \theta^*_k)]^{-1}\be_k^\top, \quad \text{so} \quad A_n \phi_i(\btheta) = \frac{\sqrt{K_n} \ind\{k_i = k\} \psi(\bX_i, \theta_k)}{\E[\psi'(\bX_{i}; \theta^*_k)]}.
\]
Simple calculations yield $\E[\| A_n \phi_i(\btheta^*) \|^4 ] = K_n^2 = o(n)$ and $O(1)$, $D_n = \sqrt{K_n} C$ and $O(r_n / \sqrt{K_n})$ for the three quantities in \ref{A:Asymp}, so the conditions of \cref{theorem2} are satisfied since $K_n^3/n \to 0$.

\subsection{Proof of \autoref{cor:network}}

We first verify \ref{A:Cons1}.
Similar to the proof of \cref{cor:distributed}, we can choose
\begin{align*}
 H_n(\bx_i, \bk_i) =  \frac{1}{m_N} \sum_{j \in I_{i,N}} \begin{pmatrix}
  	 \mathrm{diag} \lf( -\ind \{k_j = k \} K_n c\ri)_{k = 1, \ldots, K_n}  + I_{K_n}/\sqrt{4 K_n} &  \0 \\
   \0  & -1 +  1/\sqrt{4 K_n} 
  \end{pmatrix}
\end{align*}
with some $c > 0$ by the assumptions on $\psi'$.
We obtain 
\[
\frac{1}{n} \sumin \E[H_n(\bx_i, \bk_i)] = \begin{pmatrix}
  	 -c I_{K_n}  &  \0 \\
   \0  & -1 
  \end{pmatrix},
  \| H_n(\bx_i, \bk_i) \| = O(K_n), 
  \| H_n(\bx_i, \bk_i)^2 \| = O(K_n^2)
\]
using $n \, m_N = N, K_n = N/N_k, N_k = \sum_{j=1}^N \ind\{ k_j = k \}$.
Choosing $B_n = K_n C$ with some $C < \infty$ large enough, \ref{A:Cons1} holds since $K_n^2 \ln K_n / n \to 0$ and $B_n = O(n/(n - L_N))$ (since $K_n (n-L_N)/n = K_n/\sqrt{n} \, \cdot \, (n-L_N)/\sqrt{n} = o(1)$) by assumption.
For $I(\btheta)$, we have
\begin{align*}
I (\btheta)_{K_n} &  = \frac 1 n \sumin  \mathrm{Cov}[\phi(\bk_i, \bX_i ; \btheta)] \\ 
&= \frac 1 n \sumin \frac{1}{m_N^2} K_n^2 \sum_{j \in I_{i,N}}  \sum_{j'\in I_{i,N}} \left( \ind\{ k_j = k \land k_{j'} = k' \} \E[\psi(\bZ_j; \btheta_k) \psi(\bZ_{j'}; \theta_{k'}) ] \right)_{k, k' = 1, \ldots, K_n},
\end{align*}
where $I (\btheta)_{K_n}$ is the upper left $K_n \times K_n$ block of $I(\btheta)$ and the remaining column and row are $\bnull$.
By the definition of $d_n$, we have $\tr(I(\btheta^*)) = d_n K_n^2/m_N$ and can therefore take $r_n = \sqrt{d_n K_n^2/N}$.
Since $\| \phi_i(\btheta^*) \|_\infty \le K_n C$ for some $C < \infty$, \eqref{eq:assumpDep} holds since $r_n n /(K_n^{3/2} (n - L_N)) \to \infty$ because $K_n N (n - L_N)^2/(d_n n^2) \to 0$.
Now \cref{th:dependence} implies consistency and uniqueness (since $H_n(\bx_i, \bk_i)$ does not depend on $\hbtheta$) of $\hbtheta$.

For asymptotic normality, note that $J(\btheta)$ is the exact same as in the \emph{iid} case.
Choosing $A_n = \sqrt{m_N}\sigma_n^{-1} \ba_n^\top J(\btheta^*)^{-1} \in \R^{1 \times p_n}$ with $\ba^\top_n = (0, \dots, 0, 1)$ gives the statement for $ \htheta_{ K_n + 1}$, since, by \eqref{eq:cov-decay},
\[
\frac{m_N}{\sigma_n^2}\ba_n^\top J(\btheta^*)^{-1} I_{}(\btheta^*) J(\btheta^*)^{-\top} \ba_n = 1 + o(1).
\]
For \ref{A:Asymp}, it holds that
 \begin{align*}
& \quad A_n [\phi_{i}(\btheta^* + \bu) -   \phi_{i}(\btheta^* + \bu')] \\
 & = \sqrt{m_N}\sigma_n^{-1} \left(u_{K_n + 1} - u'_{K_n + 1} - \frac{1}{K_n} \sum_{k=1}^{K_n}  (u_k - u'_k) - \frac{1}{m_N} \sum_{j \in I_{i,N}} \frac{\psi'(\Z_j; \theta_{k_j}^* +  \tilde u_{k_j}) (u_{k_j} - u'_{k_j})}{\E[\psi'(\bZ_j; \theta_{k_j}^*)]} \right) \\ & = O(\sqrt{m_N}\sigma_n^{-1}\| \bu - \bu' \|),
 \end{align*}
 so the first condition holds since $m_N\sigma_n^{-2} = o((r_n^2 K_n)^{-1})$ because $K_n^3 d_n/(n \sigma_n^2) = o(1)$, which follows from $K_n^3/n \to 0$ and $d_n = O(\sigma_n^2 )$.
 For the second condition, we choose $D_n = \sqrt{m_N}\sigma_n^{-1} C$ with some large enough $C < \infty$.
 This $D_n$ satisfies $D_n = o(\sqrt{n}/(r_n K_n))$ since $d_nK_n^4/(n^2\sigma_n^2) = o(1)$ and $D_n = O(\sqrt{n}/(n-L_N))$ by assumption.
 For the third condition, we obtain, similar to the proof of \cref{cor:distributed},
 \[
  A_n [\bar J(\btheta^*, \btheta^* + \bu) - J(\btheta^*)] = O(\sqrt{m_N}r_n/(\sigma_n K_n)) = o((\sqrt{n} r_n)^{-1})
 \]
 since $r_n^2\sqrt{N} /(\sigma_n K_n) = d_n K_n /(\sqrt{N} \sigma_n)  = O(\sqrt{d_n} K_n/\sqrt{N}) = O(K_n / \sqrt{n})= o(1)$ because $d_n = O(\sigma_n^2)$ and $d_n = O(m_N)$.
\ref{A:Cons2} holds by assumption since 
 \[
 A_n \phi_i(\btheta^*) = - \frac{1}{\sigma_n \sqrt{m_N}} \sum_{j \in I_{i,N}} \frac{\psi(\bZ_j, \theta^*_{k_j})}{\E[\psi'(\bZ_j; \theta^*_{k_j})]}.
 \]
 For the asymptotic distribution of $\htheta_k$, choose $A_n = \sqrt{m_N}\sigma_{n,k}^{-1} \ba_n^\top J(\btheta^*)^{-1}$ with $\ba_n^\top = K_n^{-1/2} \be_k^\top$.
 Now
 \[
\frac{m_N}{\sigma_{n,k}^2}\ba_n^\top J(\btheta^*)^{-1} I(\btheta^*) J(\btheta^*)^{-\top} \ba_n = 1 + o(1)
 \]
 and
 \[
 A_n \phi_i(\btheta) 
  = \frac{\sqrt{m_N}}{\sigma_{n,k} \sqrt{K_n} \E[\psi'(\bZ_j; \theta^*_k)]} \be_k^\top \phi_i(\btheta)
 = \frac{\sqrt{K_n}}{\sigma_{n,k} \sqrt{m_N}}  \sum_{j \in I_{i,N}} \ind\{ k_j = k \} \frac{\psi(\bZ_j; \theta_k)}{\E[\psi'(\bZ_j; \theta_k^*]},
 \]
 so \ref{A:Cons2} holds by assumption.
 In the first condition in \ref{A:Asymp}, the left-hand side is $O(m_N\sigma_{n,k}^{-2}) = o((r_n^2 K_n)^{-1})$ since
\begin{align*}
  & \quad \frac{K_n}{\sigma_{n,k}^2 m_N} \frac 1 n \sumin \sum_{j \in I_{i,N}}\sum_{j' \in I_{i,N}} \ind\{k_j = k_{j'} = k \} \E[\psi'(\bZ_j; \theta_k)\psi'(\bZ_{j'}; \theta_k)]
  = O\left(m_N\sigma_{n,k}^{-2}\right).
\end{align*}
For the second condition, we can choose $D_n = \sqrt{K_nm_N} C / \sigma_{n,k}$ with some large enough $C < \infty$, which satisfies $D_n =  o(\sqrt{n}/(r_n K_n))$ and $D_n = O(\sqrt{n}/(n-L_N))$ by assumption.
For the third condition, simple calculations (see also the proof of \cref{cor:distributed}) yield $O(\sqrt{m_N}r_n/(\sigma_{n,k}\sqrt{K_n })) = o((\sqrt{n} r_n)^{-1})$ since $r_n^2 \sqrt{N}/(\sigma_{n,k}\sqrt{K_n }) = d_n K_n^{3/2}/(\sqrt{N} \sigma_{n,k}) = o(1)$.

\subsection{Proof of \autoref{cor:IPW}}

First note that we can multiply the first block of $\phi$ with any $\kappa > 0$ without changing the solution.
Denote $\bX = (Y, T, \bZ, \bW)$.
  By the mean value theorem, for any $\bu $, there exists $s \in (0, 1)$ such that 
  \begin{align*}
    &\quad \langle \bu, \phi(\bX; \btheta + \bu) -  \phi(\bX; \btheta) \rangle  = \bu^\top \nabla_{\btheta} \phi(\bX; \btheta + s \bu) \bu,
  \end{align*}
  where
  \begin{align*}
    \nabla_{\btheta} \phi(\bX; \btheta) = \begin{pmatrix}
      \kappa[T(\ln \sigma)''(\bW^\top \btheta_1) + (1 - T)(\ln \bar \sigma)''(\bW^\top \btheta_1)] \bW \bW^\top & 0 \\
      \lf[\frac{Y T}{\sigma(\bW^\top \btheta_1)^2} + \frac{Y(1 - T)}{\bar  \sigma(\bW^\top \btheta_1)^2}\ri] \sigma'(\bW^\top \btheta_1) \bZ \bW^\top & -\bZ \bZ^\top
    \end{pmatrix}.
  \end{align*}
  To obtain suitable matrices $H_n(\bX)$, we distinguish the two cases $T = 0, 1$. Consider first the case $T = 1$, in which we can simplify
  \begin{align*}
    &\quad \bu^\top \nabla_{\btheta} \phi(\bX; \btheta) \bu  \\
    &= \bu^\top \begin{pmatrix}
      \kappa(\ln \sigma)''(\bW^\top \btheta_1)  \bW \bW^\top & 0 \\
      \frac{Y \sigma'(\bW^\top \btheta_1)}{\sigma(\bW^\top \btheta_1)^2}   \bZ \bW^\top & -\bZ \bZ^\top
    \end{pmatrix} \bu   \\
    &=
    \begin{pmatrix}
      \bu_1^\top \bW \\ \bu_2^\top \bZ
    \end{pmatrix}
     \begin{pmatrix}
      \kappa(\ln \sigma)''(\bW^\top \btheta_1) & 0 \\
      \frac{Y \sigma'(\bW^\top \btheta_1)}{\sigma(\bW^\top \btheta_1)^2}   & -1
    \end{pmatrix} \begin{pmatrix}
      \bW^\top \bu_{1} \\ \bZ^\top \bu_2
    \end{pmatrix} \\
    &\le
    \begin{pmatrix}
      \bu_1^\top \bW \\ \bu_2^\top \bZ
    \end{pmatrix}
     \begin{pmatrix}
      \kappa(\ln \sigma)''(\bW^\top \btheta_1) + \left|\frac{Y \sigma'(\bW^\top \btheta_1)}{2\sigma(\bW^\top \btheta_1)^2} \right| & 0 \\
       0  & -1 + \left|\frac{Y \sigma'(\bW^\top \btheta_1)}{2\sigma(\bW^\top \btheta_1)^2} \right|
    \end{pmatrix} \begin{pmatrix}
      \bW^\top \bu_{1} \\ \bZ^\top \bu_2
    \end{pmatrix} \\
  \end{align*}
  using \cref{lem:diagonal-H-cover}.
By the same arguments, we get a similar result for $T = 0$. Denoting
  $\bar \sigma = 1 - \sigma$, we obtain
  \begin{align*}
    &\quad \bu^\top \nabla_{\btheta} \phi(\bX; \btheta) \bu \\
  &\le 
  \begin{pmatrix}
    \bu_1^\top \bW \\ \bu_2^\top \bZ
  \end{pmatrix}
 \begin{pmatrix}
     \kappa (\ln \bar \sigma)''(\bW^\top \btheta_1) + \lf|\frac{Y \sigma'(\bW^\top \btheta_1)}{2 \bar \sigma(\bW^\top \btheta_1)^2}\ri| & 0 \\
     0  & -1 + \lf|\frac{Y \sigma'(\bW^\top \btheta_1)}{2 \bar \sigma(\bW^\top \btheta_1)^2}\ri| 
  \end{pmatrix} 
  \begin{pmatrix}
    \bW^\top \bu_{1} \\ \bZ^\top \bu_2
  \end{pmatrix}.
  \end{align*}
  By the assumptions on $\bW$, $\sigma$ and $Y$, there is $K \in (0, \infty)$ such that 
  \begin{align*}
    |Y \sigma'(\bW^\top \btheta_1) /( 2  \sigma(\bW^\top \btheta_1)^2)| \le K , \quad |Y \sigma'(\bW^\top \btheta_1) / (2 \bar \sigma(\bW^\top \btheta_1)^2)| \le K, \quad \|\E[\bW \bW^\top]\| \le K.
  \end{align*}
  Now the matrix
  \begin{align*}
    H_n(T, Y, \bW) = \begin{pmatrix}
      [\kappa \alpha_1(T, \bW) + K] \bW \bW^\top & 0 \\
      0 &  \alpha_2(T, Y, \bW) \bZ\bZ^\top,
    \end{pmatrix}
  \end{align*}
  satisfies \ref{A:Cons1}\ref{eq:Hn-def}--\ref{eq:Hn-identifiable} with $\kappa \ge 2K^2 / c$.
  Finally, because $\alpha_1$ and $\alpha_2$ are uniformly bounded, \ref{A:Cons1}\ref{eq:Hn-bounds} and \ref{A:Asymp} can be verified exactly as for the generalized linear model (\cref{cor:Mest}).
  \ref{A:Cons2} follows from $\max_k\E[\phi(Y, t, \bZ, \bW; \btheta^*)^4_k] = O(1)$, see the proof of \cref{cor:Mest}.

  \subsection{Proof of \autoref{cor:causalSCAD}}
 
  Since $\sqrt{n/s_n} \min_{1 \le k \le s_n} | \theta_k^* | \to \infty$ and $(\min_{1 \le k \le s_n} |\theta_k^* | - \sqrt{s_n / n})/\lambda_n \to \infty$, \ref{A:Penalty1}, \ref{A:Penalty2} and \ref{A:Penalty4} are satisfied for the SCAD penalty.
  Consistency with $\tilde r_n = \sqrt{s_n/n}$ and asymptotic normality of $\hbtheta_{(1)}$ as well as its oracle property (efficiency) follow from \cref{cor:IPW}, $\rho^{-1}(n)s_n^2/n \to 0$ and the properties of the SCAD penalty with $\lambda_n \to 0$.

  \ref{A:phi-moments} follows from the assumptions on $p_n$, see the proof of \cref{cor:MestP}.
  \ref{A:lambda} always holds for SCAD with $(\min_{1 \le k \le s_n} |\theta_k^* | - \sqrt{s_n / n})/\lambda_n \to \infty$.
  The proposed choice of $\lambda_n$ satisfies \ref{A:lambda2} since
  \[
  \max_{s_n + 1 \le k \le p_n} \sup_{\btheta \in  \Theta'_n, \| \bx \|_\infty \le 1}\| \bar J(\btheta^*, \btheta)_{k,(1)} \;\bar J(\btheta^*, \btheta)^{-1}_{(1)}  \bx \|_2 \le  \sup_{\btheta \in \Theta'_n} \lf\| \bar  J(\btheta^*, \btheta)_{(2, 1)} \; \bar J(\btheta^*, \btheta)^{-1}_{(1)} \ri\|_{\infty}
  \] 
  (see the proof of \cref{cor:MestP}), and \ref{A:Penalty_emp_pr2} can be verified as in the proof of \cref{cor:MestP}, so $\hbtheta_{(2)} = \0$ with probability tending to 1. \ref{A:phi-H-penalty} follows from the assumptions on $p_n$, see the proof of \cref{cor:MestP}.
  The upper bound on $\limsup_{n \to \infty}\lambda_{\max}(n^{-1} \sumin\E[H_n(\bX_{i})])$ is given by $- \min\{c, K^2 \}$, so the assumption on $a$, which determines the non-convexity of SCAD, implies that \ref{A:Penalty3} holds with $\mu = (a - 1)^{-1} <  \min\{c, K^2 \}$, so $\hbtheta$ is unique with probability tending to 1 by \cref{theorem5}.
 
  \subsection{Proof of \autoref{cor:GD}}

  The Jacobian of $\phi$ is a bidiagonal matrix with
  \begin{align*}
    \nabla \phi(\bX; \btheta) = K_n\begin{pmatrix}
      - \ind_{i \in \Bcal_1} & & &    \\
      [1 -  \alpha  f'(\bX_{i}; \theta_1)] \ind_{i \in \Bcal_2}  & - \ind_{i \in \Bcal_2} &  &   \\
&  \ddots &  &    \\
       & &   [1 - \alpha f'(\bX_{i}; \theta_{K_n - 1})] \ind_{i \in \Bcal_{K_n}}&  -\ind_{i \in \Bcal_{K_n}} 
    \end{pmatrix}.
  \end{align*}
  We have 
  \begin{align*}
    &\quad K_n^{-1}\bu^\top \nabla \phi(\bX; \btheta) \bu = -\sum_{k = 1}^{K_n} u_k^2\ind_{i \in \Bcal_k}  + \sum_{k = 1}^{K_n - 1} u_{k} u_{k + 1} [1 - \alpha  f'(\bX_{i}; \theta_k)]\ind_{i \in \Bcal_{k + 1}} \\
    &\le -\sum_{k = 1}^{K_n} u_k^2 \ind_{i \in \Bcal_k}  +  (1 - \alpha  \kappa) \sum_{k = 1}^{K_n - 1} \sqrt{u_{k}^2 u_{k + 1}^2}\ind_{i \in \Bcal_{k + 1}}  \expl{$\kappa \le f' \le 1/\alpha$} \\
    &\le -\sum_{k = 1}^{K_n} u_k^2\ind_{i \in \Bcal_k}  + (1 - \alpha  \kappa) \frac 1 2 \sum_{k = 1}^{K_n - 1} (u_{k}^2 + u_{k + 1}^2 )\ind_{i \in \Bcal_{k + 1}}   \expl{AM-GM inequality} \\
    &\le -\sum_{k = 1}^{K_n} u_k^2\ind_{i \in \Bcal_k}  +  (1 - \alpha  \kappa) \sum_{k = 1}^{K_n} u_{k}^2 \ind_{i \in \Bcal_{k}}   \le - \alpha \kappa \sum_{k = 1}^{K_n} u_k^2\ind_{i \in \Bcal_{k}}.
  \end{align*}
  Hence, \ref{A:Cons1} is satisfied with
  \begin{align*}
    H_n(\bx_i) = - \alpha \kappa K_n  \diag(\ind_{i \in \Bcal_k} )_{k = 1, \dots, K_n} ,
  \end{align*}
  $c = \alpha \kappa$, $B_n = K_n \ln K_n$ since $K_n^2 \ln K_n/n \to 0$. 
  Further, \ref{A:Cons1}\ref{eq:Hn-def} is valid for all $\btheta \in \R^{K_n}$. 
  Finally, 
  \begin{align*}
    I(\btheta^*) = \cov[\phi(\btheta^*)] = \alpha^2 K_n \diag\lf( \var[f'(\bX; \theta_{k-1}^*)]\ri)_{k = 1, \dots, K_n} \eqcolon \alpha^2 K_n \Gamma,
  \end{align*}
  so $\tr(I(\btheta^*)) = O(K_n^2)$ and $r_n = O( \sqrt{K_n^2 / n})$.
  Now \cref{theorem1} and \cref{theorem1-uniqueness} show that,  with probability tending to 1, a unique solution path $\hbtheta$ exists and satisfies $  \| \hbtheta - \btheta^*\| = O_p(\sqrt{K_n^2 / n}).$

  We now verify the conditions of \cref{theorem2} for a matrix $\tilde A_n$ such that $\| \tilde A_n\| = O(1/\sqrt{K_n})$ to be chosen later.
  It holds that
  \begin{align*}
    \| \phi_{i}(\btheta^* + \bu) - \phi_{i}(\btheta^* + \bu') \| \le (2 + \alpha L) K_n \|\bu - \bu'\| \le 3 K_n \|\bu - \bu'\|,
  \end{align*}
  so the first two conditions of \ref{A:Asymp} are satisfied for any matrix $\tilde A_n$ with $\|\tilde A_n\| = O(1 / \sqrt{K_n})$ and $D_n = K_n$ since $K_n^3 / n  \to 0$. Next, 
  \begin{align*}
    J(\btheta) = \frac 1 n \sumin \E[\nabla_{\btheta} \phi(\bX_i; \btheta)] = 
      \begin{pmatrix}
        - 1 & & &    \\
        1 - \alpha \E[f'(\bX_{}; \theta_1)]  & - 1 &  &   \\
         &  \ddots &  &    \\
         & &   1 - \alpha  \E[f'(\bX_{}; \theta_{K_n - 1})]&  -1 
      \end{pmatrix}.
  \end{align*}
  We have 
  \begin{align*}
    \|\tilde A_n \bar J(\btheta^*, \btheta^* + \bu )  - \tilde A_nJ(\btheta^* )\|   &\le \alpha \|\tilde A_n\| \sqrt{ \sum_{j = 1}^{K_n - 1} \lf(\int_0^1 \E[f'(\bX_{}; \theta_{k}^* + t u_k)] - \E[f'(\bX_{}; \theta_{k}^* )] dt\ri)^2 } \\
    &\le \alpha \|\tilde A_n\| \, L \, \|\bu \| = O(\|\bu \| / \sqrt{K_n}) = o(\|\bu \| / ( \sqrt{n} r_n^2)),
  \end{align*}
  where we used $r_n^2 = K_n^2 / n$ and $K_n^3 / n \to 0$. This verifies the third condition in \ref{A:Asymp}. Assumption \ref{A:Cons2} also holds because only one entry of $\phi_{i}(\btheta^*)$ can be non-zero at a time and, thus,
  \begin{align*}
    &\quad \E[\|\tilde A_n \phi_{i}(\btheta^*)\|^4] = O(1 / K_n^2) \times \E[\|\phi_{i}(\btheta^*)\|^4] \\
    &= O(1 / K_n^2) \times  K_n^4 \sum_{k = 1}^{K_n} \sum_{j = 1}^{K_n}  \E[(\theta_{k - 1}^* - \theta_{k}^* - f'(\bX_{i}; \theta_k^*))^2(\theta_{j - 1}^* - \theta_{j}^* - f'(\bX_{i}; \theta_j^*))^2 \ind_{i \in \Bcal_{k} \cap \Bcal_{j}}] \\
    &= O(K_n^2) \times  \sum_{k = 1}^{K_n}  \E[(\theta_{k - 1}^* - \theta_{k}^* - f'(\bX_{i}; \theta_k^*))^4 \ind_{i \in \Bcal_{k}}] = O(K_n^2) = o(n),
  \end{align*}
  again using $K_n^3 / n \to 0$.
  We can now apply \cref{theorem2}, and it remains to verify the asymptotic covariance structure.
  By the inversion formula for bidiagonal matrices,
  \begin{align*}
    (J(\btheta)^{-1})_{i, j} = \begin{cases}
      -1, &i = j \\
      -\prod_{\ell = j}^{i - 1} (1 - \alpha \E[f'(\bX_{}; \theta_{\ell})]),  & i > j\\
      0, & i < j.
    \end{cases}
  \end{align*}
  Let $\tilde A_n =  A_n  J(\btheta^*)^{-1} / \sqrt{K_n} = O(1 / \sqrt{K_n})$, so that
  \begin{align*}
   \tilde  A_n I(\btheta^*) \tilde A_n^\top  =  \frac{1}{K_n } A_n  J(\btheta^*)^{-1} I(\btheta^*)J(\btheta^*)^{-\top}  A_n^\top  .
  \end{align*}
  Since $I(\btheta^*) = \alpha^2 K_n \Gamma$ with $\Gamma$ diagonal, we have for $i \le j$,
  \begin{align*}
    \Sigma_n & \coloneq \frac 1 {K_n} (J(\btheta^*)^{-1} I(\btheta^*)J(\btheta^*)^{-\top})_{i, j} \\
    &= \alpha^2 \sum_{\ell = 1}^{K_n} \sum_{k = 1}^{K_n} (J(\btheta)^{-1})_{i, \ell} \Gamma_{\ell, k}  (J(\btheta)^{-1})_{j, k} \\
    &= \alpha^2    \sum_{\ell = 1}^{i} \sum_{k = 1}^{j} (J(\btheta)^{-1})_{i, \ell} \Gamma_{\ell, k}  (J(\btheta)^{-1})_{j, k} \\
    &= \alpha^2   \sum_{k = 1}^{i} \Gamma_{k, k} (J(\btheta)^{-1})_{i, k}   (J(\btheta)^{-1})_{j, k} \\
    &= \alpha^2   \sum_{k = 1}^{i} \Gamma_{k, k} \lf[\prod_{m = k}^{i - 1} (1 - \alpha \E[f'(\bX_{}; \theta_{m}^*)])\ri] \lf[\prod_{m = k}^{j - 1} (1 - \alpha \E[f'(\bX_{}; \theta_{m}^*)])\ri] \\
    &= \alpha^2   \sum_{k = 1}^{i} \Gamma_{k, k} \lf[\prod_{m = k}^{i - 1} (1 - \alpha \E[f'(\bX_{}; \theta_{m}^*)])\ri]^2 \lf[\prod_{m = i}^{j - 1} (1 - \alpha \E[f'(\bX_{}; \theta_{m}^*)])\ri].
  \end{align*}

\section{Lemmas}\label{sec:Lemmas}
\setcounter{lemma}{0}

\begin{lemma}\label{lem:dep2}
 Suppose that condition \ref{ass:dep} holds, and let $\Fcal$ be a collection of functions. 
 If 
 \begin{align} \label{cond:envelope-dropped2}
  \sum_{i = L_N + 1}^{n_N} \Pr\left(\sup_{f \in \Fcal} |f(\bX_i)| > \frac{\eps_N}{n_N - L_N} \right) = o(1),
 \end{align}
 holds for a given $\eps_N$, then there are independent random variables $\bX_1^*, \dots, \bX_{n_N}^*$ with $\bX_i^* =_d \bX_i$, $i=1,\dots,n_N$, such that
 \begin{align*}
   \sup_{f \in \Fcal} \left| \frac{1}{n_N}\sum_{i=1}^{n_N} f(\bX_i) -  \frac{1}{n_N}\sum_{i=1}^{n_N}f(\bX_i^*) \right| = o_p(\eps_N / n_N).
 \end{align*}
\end{lemma}
\begin{proof}
Recall that $\bX_i=(\bZ_j:j\in I_{i,N})$ and denote 
  \begin{align*}
    \delta_N := \max_{1\le i\le L_N}
    \beta\!\left(
      \sigma\{\bZ_j: j\in I_{i,N}\},
      \sigma\{\bZ_j: j\in\cup_{ \ell \le L_N, \ell\ne i} I_{\ell,N} \}
    \right).
  \end{align*}
  Recursively applying Berbee's Lemma, we can construct an independent sequence $\bX_1^*, \dots, \bX_{L_N}^*$ satisfying $\bX_i^* =_d \bX_i$ and $\max_{1 \le i \le L_N}\Pr(\bX_i^* \neq \bX_i) \le \delta_N$ \citep[Proposition 2]{doukhan1995invariance}. Now draw $\bX_{L_N + 1}^*, \dots, \bX_{n_N}^*$ independently of one another and of $(\bX_{1}^*, \dots, \bX_{L_N}^*)$, each following the distributions of $\bX_{L_N + 1}, \dots, \bX_{n_N}$, respectively.
  The triangle inequality gives that
   \begin{align*}
   &\quad \, \sup_{f \in \Fcal} \left| \frac{1}{n_N}\sum_{i=1}^{n_N} f(\bX_i) - f(\bX_i^*) \right| \le \frac{1}{n_N}\sum_{i=1}^{n_N} \sup_{f \in \Fcal} \left|  f(\bX_i) - f(\bX_i^*) \right| \\
   &\le  \frac{1}{n_N}\sum_{i=1}^{L_N} \sup_{f \in \Fcal} \left|  f(\bX_i) - f(\bX_i^*) \right| + \frac{1}{n_N}\sum_{i=L_N + 1}^{n_N} \sup_{f \in \Fcal} \left|  f(\bX_i) - f(\bX_i^*) \right| \coloneqq T_1 + T_2.
 \end{align*}
 For the first term, we have 
 \begin{align*}
  \Pr(T_1 \neq 0) &\le \Pr\left(\bigcup_{i=1}^{L_N} \{\bX_i^* \neq \bX_i\}\right) 
  \le \sum_{i=1}^{L_N} \Pr(\bX_i^* \neq \bX_i) 
  \le L_N \delta_N = o(1).
 \end{align*}
 For the second term, we have
  \begin{align*}
    \Pr(T_2 > \eps_N / n_N) & = \Pr\left(\sum_{i=L_N + 1}^{n_N} \sup_{f \in \Fcal} \left|  f(\bX_i) - f(\bX_i^*) \right| > \eps_N \right) \\
    &\le \Pr\left(\bigcup_{i=L_N + 1}^{n_N} \left\{\sup_{f \in \Fcal} |f(\bX_i)| > \frac{\eps_N}{n_N - L_N} \right\} \cup \left\{\sup_{f \in \Fcal} |f(\bX_i^*)| > \frac{\eps_N}{n_N - L_N} \right\}\right) \\
    &\le 2\sum_{i = L_N + 1}^{n_N} \Pr\left(\sup_{f \in \Fcal} |f(\bX_i)| > \frac{\eps_N}{n_N - L_N} \right) = o(1).
  \end{align*}
  Together, this gives that $T_1 = 0$ with probability $\to 1$ and $T_2 = o_p(\eps_N / n_N)$, and hence the claim follows.
\end{proof}

\begin{lemma}\label{lem:network}
  In the setting described in \cref{ex:network}, condition \ref{ass:dep} holds.  
\end{lemma}
\begin{proof}
Let 
\[
  M_N= \max_{C\in\mathcal C_N}|C\setminus U_N|, \qquad \ell_N= \max\{ M_N, N\beta(q_N)\},
  \qquad
  m_N=\left\lceil (N\ell_N)^{1/2}\right\rceil ,
\]
suppressing immaterial rounding so that $m_N$ divides $N$. We have
\[
  \ell_N=o(m_N),
  \qquad
  m_N=o(N),
  \qquad
  n_N=N/m_N\to\infty .
\]

For each component $C\in\mathcal C_N$, define its retained core
$A_C=C\setminus U_N .$
Then $|A_C|\le M_N=o(m_N)$. Construct retained blocks greedily from the
sets $A_C$: traverse the non-empty cores in any order and add whole cores
to the current block until the next core would make its cardinality exceed
$m_N$. Then use only as many vertices of that core as needed to complete
the block and discard the rest of that core. Start a new block and
continue. In this construction no core contributes vertices to two
different retained blocks, and at most $M_N$ vertices are discarded per
retained block.

Let $L_N$ be the number of retained blocks. Since the hub neighborhood
$U_N$ is discarded and at most $M_N$ additional vertices are lost per
retained block,
\[
  L_N
  \ge
  \frac{N-|U_N|}{m_N+M_N}-1 .
\]
Hence,
\[
  \frac{L_N}{n_N}
  =
  \frac{L_Nm_N}{N}
  \ge
  \frac{m_N}{m_N+M_N}\left(1-\frac{|U_N|}{N}\right)
  -\frac{m_N}{N}
  \to1 .
\]

It remains to verify the beta-mixing part. Distinct retained blocks are
made from cores belonging to distinct connected components of
$G_N[V_N\setminus H_N]$. 
Therefore every path between two distinct retained
blocks must pass through $H_N$. Since all retained vertices lie outside
$U_N$, each retained vertex has graph distance greater than $q_N$ from
$H_N$. 
Thus distinct retained blocks are separated by graph distance at
least $q_N$. Consequently,
\begin{align*}
    L_N
  \max_{1\le i\le L_N}
  \beta\!\left(
    \sigma\{\bZ_j:j\in I_{i,N}\},
    \sigma\{\bZ_j:j\in\cup_{\ell\le L_N,\ell\ne i}I_{\ell,N}\}
  \right)
  \lesssim
  \frac{N}{m_N}\beta(q_N)
  \lesssim
  \frac{\ell_N}{m_N}
  \to 0 . 
\end{align*}
\end{proof}

\begin{lemma}\label{lem:Yokoyama}
  Let $X_i, i = 1, \ldots, n$ be a stationary, polynomially $\beta$-mixing time series with $\beta(q) \lesssim q^{-\kappa}$ (see \cref{ex:timeseries}), $\E[X_i] = 0$ and $\E[|X_i|^{2r + \eps}] = O(1)$ for some $r > 1, \eps > 0$.
  If $\kappa > r(2r + \eps)/\eps$, then $\E[| \sumin X_i |^{2r}] = O(n^{r})$.  
\end{lemma}
\begin{proof}
  The assumption on $\kappa$ implies
  \[
  \sum_{q=1}^\infty q^{r - 1} \beta(q)^{\eps / (2r + \eps)}
  \lesssim \sum_{q=1}^\infty q^{r - 1 -  \frac{\kappa \eps}{2r + \eps}} < \infty, \quad \text{since} \quad r - 1 - \kappa \frac{\eps}{2r + \eps} < -1 
  \Leftrightarrow \kappa > \frac{r (2r + \eps)}{\eps}.
  \]
  Now the claim follows directly from Theorem 1 in \citet{Yokoyama1980}.
\end{proof}

\begin{lemma}\label[lemma]{lem:semi-definite}
  Suppose that $H_n(\bx)$ is negative semi-definite for all $\bx \in \Xcal$, $\frac 1 n \sumin \E[H_n(\bX_i)] = O(1)$, and $B_n = o(n / \ln p_n)$. Then 
  \begin{align*}
    & \frac 1 n \lf\| \sumin  \E[H_n(\bX_i)^2 \ind_{\|H_n(\bX_i)\| \le B_n}] \ri\| = o(n / \ln p_n).
  \end{align*}
\end{lemma}

\if1\suppl
\begin{proof}
  Using $\tr(AB) \le |\tr(A)| \|B\|$, we have
  \begin{align*}
    \bu^\top H_n(\bX_i)^2 \bu \ind_{\|H_n(\bX_i)\| \le B_n} &= \tr(\bu^\top H_n(\bX_i)^2 \bu)\ind_{\|H_n(\bX_i)\| \le B_n} \\
    &= \tr( \bu \bu^\top H_n(\bX_i)^2)\ind_{\|H_n(\bX_i)\| \le B_n} \\
    &\le |\tr( \bu \bu^\top H_n(\bX_i))| \| H_n(\bX_i)\| \ind_{\|H_n(\bX_i)\| \le B_n} \\
    &\le |\bu^\top H_n(\bX_i)\bu | B_n = -\bu^\top H_n(\bX_i)\bu B_n.
  \end{align*}
  Thus,
  \begin{align*}
  \frac 1 n \sumin \bu^\top  \E[ H_n(\bX_i)^2 \ind_{\|H_n(\bX_i)\| \le B_n} ]\bu \le - \frac 1 n \sumin \bu^\top \E[H_n(\bX_i)] \bu B_n,
  \end{align*}
  which implies the claim.
\end{proof}
\fi

\begin{lemma} \label[lemma]{lem:RSC-LW}
  Let $\blambda_n = \lambda_n \bm 1$ with $\lambda_n \to 0$ and $\lambda_n \ge 2  \eta_n$.   Denote $\tilde{\Theta}_n = \{ \btheta: \| \btheta \|_1 \le k_n\}$ with some $k_n = o(\eta_n^{-1})$ such that $\btheta^* \in \tilde \Theta_n$.
  Suppose that \ref{A:Penalty3} and \ref{A:phi-moments} hold, 
  \begin{align} \label{eq:penalty-upper-bound}
    \sup_{\btheta \in \tilde \Theta_n, \: \bz \in \partial p_{\lambda_n}(\btheta)} \| \bz \|_\infty = O(\lambda_n),
  \end{align}
  and for all $\|\bu\| = o(\lambda_n)$ and large enough $n$,
  \begin{align} \label{eq:penalty-lower-bound}
    \sup_{\bz \in \partial p_{\lambda_n}(\btheta^* + \bu)} \langle \bu_{(2)}, \bz_{(2)} \rangle \ge \lambda_n \|\bu_{(2)}\|_1.
  \end{align}
  Then, if 
  \begin{align} \label{eq:RSC-LW-apx}
    \lf\langle \bu, \frac{1}{n} \sum_{i = 1}^n [\phi_{i}(\btheta^*) - \phi_{i}(\btheta^* + \bu) ] \ri\rangle \ge  c\|\bu\|^2 - c_1\|\bu\|_1^2 \eta_n^2
  \end{align}
  holds with $c > \limsup_{n \to \infty} \mu_n$, where $\mu_n$ is defined in \ref{A:Penalty3}, any solution $\hbtheta \in \tilde{\Theta}_n$ satisfies
  \begin{align*}
    \|\hbtheta - \btheta^*\|_1 \le \sqrt{\nu_n} \|\hbtheta - \btheta^*\|,
  \end{align*}
  for some $\nu_n = O(s_n)$.
\end{lemma}

\if1\suppl
\begin{proof}
  Write $\hbtheta = \btheta^* + \bu$ with $\|\bu\|_1 \le 2k_n$.
  We will show that $\|\bu\|_1 \le O(\sqrt{s_n}) \|\bu\|$.
  Let $t_n = o(1)$ fast enough such that $\|t_n\bu\| = o(\lambda_n)$. 
  Let $\bz \in \partial p_{\lambda_n}(\btheta^* + \bu)$ and $\bz^{(t_n)} \in \partial p_{\lambda_n}(\btheta^* + t_n \bu)$.
  It holds that
  \begin{align*}
    \langle \bu,  \bz \rangle 
    &= \langle \bu,  \bz - \bz^{(t_n)} \rangle + \langle \bu,  \bz^{(t_n)} \rangle \\
    &= \frac 1 {1 - t_n}\langle \bu(1 - t_n),  \bz - \bz^{(t_n)}  \rangle + \langle \bu,  \bz^{(t_n)}  \rangle \\
    &\ge -\mu_n  \|\bu\|^2+ \langle \bu,  \bz^{(t_n)}  \rangle\\
    &= -\mu_n  \|\bu\|^2 + \langle \bu_{(1)}, \bz^{(t_n)}_{(1)} \rangle + \frac{1}{t_n} \langle t_n\bu_{(2)},  \bz^{(t_n)}_{(2)} \rangle \\
    &\ge -\mu_n  \|\bu\|^2 - O(\lambda_n) \|\bu_{(1)}\|_1 + \lambda_n \| \bu_{(2)}\|_1.
  \end{align*}
  using \ref{A:Penalty3} and $|1 - t_n| \le 1$ in the first, and  \eqref{eq:penalty-upper-bound}--\eqref{eq:penalty-lower-bound} in the second inequality.
  Furthermore, by \eqref{eq:RSC-LW-apx}, \ref{A:phi-moments} and \cref{lem:eta_n},
  \begin{align*}
    \langle \bu, \P_n \phi(\btheta^* + \bu) \rangle &= \langle \bu, \P_n [\phi(\btheta^* + \bu) - \phi(\btheta^*)] \rangle + \langle \bu, \P_n  \phi(\btheta^*) \rangle \\
    &\le -c\|\bu\|^2 + c_1 \eta_n^2\|\bu\|_1^2 + \eta_n \|\bu\|_{1} \\
    &\le -c\|\bu\|^2 + [1 + o(1)]\eta_n\|\bu\|_1 \\
    &= -c\|\bu\|^2 + O(\lambda_n)\|\bu_{(1)}\|_1 + [1 + o(1)]\eta_n \|\bu_{(2)}\|_1,
  \end{align*}
  where we have used $\eta_n\|\bu\|_1  \le \eta_n k_n = o(1)$ in the third and $\lambda_n \ge 2\eta_n$ in the last step.
  Together this yields  
  \begin{align*}
    0 &= \langle \bu, \P_n \phi(\btheta^* + \bu) -  \bz \rangle \\
    &\le -(c -  \mu_n)\|\bu\|^2 + O(\lambda_n)  \|\bu_{(1)}\|_1 - \frac{\lambda_n}{2}[1 + o(1)]\|\bu_{(2)}\|_1,
  \end{align*}
  as $ \eta_n - \lambda_n \le -  \lambda_n / 2$ by assumption.
  Since $c - \mu_n $ is strictly positive asymptotically, it must hold that $ \|\bu_{(2)}\|_1 \le  O(1) \|\bu_{(1)}\|_1$.
  Now the claim follows from
  \begin{align*}
    \|\bu\|_1 = \|\bu_{(1)}\|_1 + \|\bu_{(2)}\|_1 \le O(1) \|\bu_{(1)}\|_1 \le O(\sqrt{s_n}) \|\bu_{(1)}\| \le O(\sqrt{s_n}) \|\bu\|. 
  \end{align*}
\end{proof}
\fi \begin{lemma} \label[lemma]{lem:diagonal-H-cover}
  For any 
  $A \in \R^{q_1 \times q_1}$, $B \in \R^{q_1 \times q_2}$, 
  $C \in \R^{q_2 \times q_1}$, $D \in \R^{q_2 \times q_2}$, 
  $q_1, q_2 \in \N$,
  it holds that
  \begin{align*}
    \begin{pmatrix}
      A & B \\ C & D
    \end{pmatrix}
    \preceq  \begin{pmatrix}
      A +  I_{q_1} (\|B\| + \|C\|)/2 & 0 \\ 0 & D + I_{q_2} (\|B\| + \|C\|)/2
    \end{pmatrix}.
  \end{align*}
\end{lemma}
\begin{proof}
  Let $\by_1 \in \R^{q_1}$ and $\by_2 \in \R^{q_2}$ arbitrary. Then using sub-multiplicativity of the norm (first step) and the AM-GM inequality (second step), we get
  \begin{align*}
    &\quad  \begin{pmatrix}
      \by_1^\top & \by_2^\top 
    \end{pmatrix} \begin{pmatrix}
      A & B \\ C & D
    \end{pmatrix}   
    \begin{pmatrix}
      \by_1  \\ \by_2 
    \end{pmatrix} \\
    &= \by_1^\top A\by_1 + \by_1^\top B \by_2 + \by_2^\top C \by_1 + \by_2^\top D \by_2 \\
    &\le  \by_1^\top A\by_1 + \|\by_1\| \| \by_2\| (\| B\| + \|C\|) + \by_2^\top D \by_2 \\
    &\le  \by_1^\top A\by_1 + \frac 1 2 (\|\by_1\|^2 + \| \by_2\|^2) (\| B\| + \|C\|)  + \by_2^\top D \by_2 \\
    &= \begin{pmatrix}
      \by_1^\top & \by_2^\top 
    \end{pmatrix} \begin{pmatrix}
      A + I_{q_1}(\|B\| + \|C\|)/2 & 0 \\ 0 & D + I_{q_2}(\|B\| + \|C\|)/2
    \end{pmatrix}   
    \begin{pmatrix}
      \by_1  \\ \by_2 
    \end{pmatrix}. 
  \end{align*}
\end{proof}
 \begin{lemma}
\label[lemma]{lem:Cons1}
It holds that
\[
\| \P_n \phi (\btheta^*  )  \| = O_p\left(\sqrt{\frac{\tr\lf(I(\btheta^*)\ri)}{n}}\right).
\]
\end{lemma}

\begin{proof}
Using $Y = O_p(\mathbb{E}(Y^2)^{1/2})$ and $ \| \bu \| = \sqrt{ \tr(\bu \bu^\top) }$, we obtain
\begin{align*}
\| \P_n \phi (\btheta^*  )  \|& = O_p(\sqrt{\E[\tr( \P_n \phi (\btheta^*  )\,  \P_n \phi (\btheta^*  )^\top)]}) 
 = O_p(\sqrt{\tr(\E[ \P_n \phi (\btheta^*  ) \, \P_n \phi (\btheta^*  )^\top])})
\end{align*}
and
\[
\E[ \P_n \phi (\btheta^*  ) \, \P_n \phi (\btheta^*  )^\top]
= \mathrm{Cov} [ \P_n \phi (\btheta^*  )] = \frac{1}{n^2} \sumin  \mathrm{Cov} [\phi_{i}(\btheta^*)] = \frac 1 n I(\btheta^*).
\]
    
\end{proof}

\begin{lemma} \label[lemma]{lem:truncation}
  Let $\Fcal_n$ be classes of real-valued functions from $\Xcal$ to $\R$, $F_n$ be any function with $\sup_{f \in \Fcal_n}|f| \le F_n$, and $B_n$ be any sequence. 
  If $ P\ind_{F_n > B_n} = o(1/n)$, it holds that
   \begin{align*}
    \sup_{f \in \Fcal_n}|(\P_n - P)f| \le  \sup_{f \in \Fcal_n}|(\P_n - P)f \ind_{F_n \le B_n}| + o_p\lf( \sqrt{\frac{\sup_{f \in \Fcal_n} P f^2}{n}}\ri).
   \end{align*}
\end{lemma}
\begin{proof}
  It holds that
    \begin{align*}
      \sup_{f \in \Fcal_n}|(\P_n - P)f| 
      &\le \sup_{f \in \Fcal_n}|(\P_n - P)f \ind_{F_n \le B_n}| + \sup_{f \in \Fcal_n}|(\P_n - P) f\ind_{F_n > B_n}|.
    \end{align*}
    Decompose
  \begin{align*}
    \sup_{f \in \Fcal_n}|(\P_n - P) f\ind_{F_n > B_n}| \le |  \P_n F_n\ind_{F_n > B_n}|  +  \sup_{f \in \Fcal}|P f\ind_{F_n > B_n}|.
  \end{align*}
  Since $ P\ind_{F_n > B_n} = o(1/n)$ it holds that
  \begin{align*}
    \Pr(\exists i\colon F_n(\bX_i) > B_n) \le n P\ind_{F_n > B_n} = o(1),
  \end{align*}
  so $\P_n F_n\ind_{F_n > B_n} = 0$ with probability tending to 1. Now the claim follows from the Cauchy Schwarz inequality:
  \begin{align*}
    \sup_{f \in \Fcal}|P f\ind_{F_n > B_n}| \le \sqrt{\sup_{f \in \Fcal} P f^2} \sqrt{P \ind_{F_n > B_n}} = o_p\lf( \sqrt{\frac{\sup_{f \in \Fcal_n} P f^2}{n}}\ri).
  \end{align*}
\end{proof}

\begin{lemma} \label[lemma]{lem:Hn-convergence}
  Under assumption \ref{A:Cons1}\ref{eq:Hn-bounds}, it holds that $\|  (\P_n - P) H_n\| = o_p(1)$.
\end{lemma}
\begin{proof}
Let $S_n = \{\bx \colon \| H_n(\bx)\| \le B_n\}$ with $B_n$ as in \ref{A:Cons1}\ref{eq:Hn-bounds}.  
Applying \cref{lem:truncation}, we obtain
\[
 \|  (\P_n - P) H_n\| \le \|  (\P_n - P) H_n \ind_{S_n}\|  + o_p\left(\sqrt{\frac{\sup_{\| \bu \| = 1} \frac{1}{n}  \sumin \E\left[ (\bu^\top H_n(\bX) \bu)^2 \right]}{n}} \right)
\]
The second term is $o_p(1)$ by \ref{A:Cons1}\ref{eq:Hn-bounds}.
Defining $M_n^2 = P H_n^2 \ind_{S_n}$, the Bernstein inequality for random matrices \citep[Theorem 6.17]{WWbook} yields
\begin{align*}
  \| (\P_n - P) H_n \ind_{S_n} \| = O_p\left(\sqrt{\frac{M_n^2 \ln p_n}{n}} + \frac{B_n \ln p_n}{ n}\right) = o_p(1). 
\end{align*}
\end{proof}

\begin{lemma} \label[lemma]{lem:Hn-convergence-pen}
  Under assumption \ref{A:phi-H-penalty}, it holds that
  \begin{align*}
    \max_{1 \le j, k \le p_n} |  (\P_n - P) H_{n, j, k} | = o_p(1 / \nu_n).
  \end{align*}
\end{lemma}
\begin{proof}
  Define
  \begin{align*}
    M_n^2 = \max_{1 \le j, k \le p_n} P H_{n, j, k}^2, \qquad S_n = \lf\{\bx \colon \max_{1 \le j, k \le p_n} |H_n(\bx)_{n, j, k}| \le \tilde B_n\ri\}.
  \end{align*}
  \cref{lem:truncation} and \ref{A:phi-H-penalty} give
  \begin{align*}
    \max_{1 \le j, k \le p_n} |  (\P_n - P) H_{n, j, k} | 
    &\le \max_{1 \le j, k \le p_n} |  (\P_n - P) H_{n, j, k} \ind_{S_n} | + o_p\lf(\sqrt{M_n^2 / n}\ri) \\
    &= \max_{1 \le j, k \le p_n} |  (\P_n - P) H_{n, j, k} \ind_{S_n} | + o_p(1/\nu_n),
  \end{align*}
  with probability tending to 1.
The union bound and Bernstein's inequality give
\begin{align*}
  \Pr\left( \max_{1 \le j, k \le p_n} |  (\P_n - P) H_{n, j, k}\ind_{S_n} | > \eps \right) \le 2 p_n^2 \exp\left(- \frac{n \eps^2}{  2 M_n^2 + \tilde  B_n \eps}\right).
\end{align*}
The claim follows upon choosing 
\begin{align*}
  \eps = C \sqrt{\frac{ M_n^2  \ln p_n}{n}} + C \frac{ \tilde B_n  \ln p_n}{n},
\end{align*}
with some large enough constant $C$ and the assumptions on $M_n$ and $\tilde B_n$ from  \ref{A:phi-H-penalty}.
\end{proof}

\begin{lemma} \label[lemma]{lem:mixed-entropy}
  For some $c_n \in (0, \infty)$, $d_n, K_n \in \N$, let
  \begin{align*}
    \Fcal_n = \{f_{\bu, k}\colon \bu \in \R^{d_n}, \| \bu \| \le c_n, f_{\bnull, k} \equiv 0, k = 1, \dots, K_n \},
  \end{align*}
  be classes of functions such that 
  \begin{align*}
     &\max_{1 \le k \le K_n} \sup_{\|\bu\|, \|\bu'\| \le c_n} \frac 1 n \sumin \frac{\E[|f_{\bu, k}(\bX_i) - f_{\bu', k}(\bX_i)|^2]}{\|\bu - \bu'\|^2}  \le M_n^2,            \\
    &  \sumin   \Pr\lf(\max_{1 \le k \le K_n}  \sup_{\|\bu\|, \|\bu'\| \le c_n}  \frac{|f_{\bu, k}(\bX_i) - f_{\bu', k}(\bX_i)|}{\|\bu - \bu'\|} > D_n \ri) = o(1).
  \end{align*}
  Then
  \begin{align*}
    \sup_{f \in \Fcal_n} | (\P_n - P)f|
     & = O_p\left( \sqrt{\frac{M_n^2 c_n^2 (d_n + \ln K_n )}{n}} + \frac{D_n c_n (d_n + \ln K_n )}{n} \right).
  \end{align*}
\end{lemma}
\begin{proof}
  We proceed in three steps.

  \myparagraph{Step 1: Truncation}
  We start with a truncation argument. Let 
  \begin{align*}
    \quad F_n(\bx) =  \max_{1 \le k \le K_n}  \sup_{\|\bu\|, \|\bu'\| \le c_n}  \frac{|f_{\bu, k}(\bx) - f_{\bu', k}(\bx)|}{\|\bu - \bu'\|}.
  \end{align*}
  Since $f_{k, \0} \equiv 0$ by assumption, $F_n$ is an envelope for the functions in $c_n^{-1} \Fcal_n$:
  \begin{align*}
    \sup_{f \in \Fcal_n} c_n^{-1} |f| \le \max_{1 \le k \le K_n}\sup_{\|\bu\| \le c_n} c_n^{-1} \|\bu\| \frac{|f_{\bu, k}(\bx) - 0|}{\|\bu - \0\|} \le  F_n(\bx).
  \end{align*}
  Now \cref{lem:truncation} and our assumptions give 
  \begin{align*}
    c_n^{-1} \sup_{f \in \Fcal_n} | (\P_n - P)f| &\le  c_n^{-1}\sup_{f \in \Fcal_n} | (\P_n - P)f \ind_{F_n \le D_n}| + o_p(M_n/\sqrt{n}),
  \end{align*}
  so
  \begin{align*}
    \sup_{f \in \Fcal_n} | (\P_n - P)f| &\le  \sup_{f \in \Fcal_n} | (\P_n - P)f \ind_{F_n \le D_n}| + o_p(\sqrt{ M_n^2 c_n^2/n}).
  \end{align*}

  \myparagraph{Step 2: Bounding covering numbers}
  Let $\Acal$ be some set equipped with a norm $\| \cdot \|$. A collection of $N$ balls $B(a_k, \epsilon) = \{ a \in \Acal \colon \| a - a_k \| \le \epsilon\}$ is called an $\epsilon$-covering of $\Acal$ with respect to the norm $\| \cdot \|$ if $\Acal \subset \bigcup_{k = 1}^N B(a_k, \epsilon)$.
  The minimal number of balls of radius $\epsilon$ needed to cover $\Acal$ is the covering number $N(\epsilon, \Acal, \| \cdot \|)$.

  Fix $k$ and consider $\Fcal_n^{(k)} = \{f_{\bu, k} \ind_{S_n}\colon \bu \in \R^{p_n}, \| \bu \| \le c_n \}$.
  Recall that by our definition of $Pf$, the $L_2(P)$-norm is defined as $\| f - g \|^2_{L_2(P)} = \frac 1 n \sumin \E[|f(X_i) - g(X_i)|^2]$.
We will show that
  \begin{align} \label{eq:covering-numbers}
    \ln N(\eps, \Fcal_n^{(k)}, L_2(P))  \lesssim d_n \ln ( 3M_nc_n/ \eps), \quad   \ln N(\eps, \Fcal_n^{(k)}, \| \cdot \|_{\infty}) & \lesssim  d_n \ln (3 D_nc_n/ \eps),
  \end{align}
  where $\lesssim$ means ``bounded up to a universal constant''.
  Let $\bu_1, \dots, \bu_N$ be the centers of an $\eta$-covering of $\{\bu \in \R^{d_n}\colon \|\bu\| \le  c_n\}$, which we can find with $N = (3c_n/\eta)^{d_n}$.
  Then, by the definitions of $M_n$ and $D_n$,
  the functions $f_{\bu_1, k}, \dots, f_{\bu_N, k}$ are centers of a $M_n\eta$-covering of $\Fcal_n^{(k)}$ in $L_2(P)$, and a $D_n\eta$ covering for  $\| \cdot \|_{\infty}$, respectively.
  Choosing $\eta = \eps/M_n$ and $\eta = \eps/D_n$, respectively, gives \eqref{eq:covering-numbers}. Now we can take a union over the coverings of all $\Fcal_n^{(k)}$ to find a covering of $\Fcal_n$, which gives 
  \begin{align} \label{eq:covering-numbers-2}
    \begin{split}
      \ln N(\eps, \Fcal_n, L_2(P))  &\lesssim \ln K_n + d_n \ln ( 3 M_nc_n/ \eps), \\
       \ln N(\eps, \Fcal_n, \| \cdot \|_{\infty}) & \lesssim \ln K_n + d_n \ln (3  D_nc_n/ \eps).
    \end{split}
  \end{align}

  \myparagraph{Step 3: Bounding the truncated process}
  Denote $S_n = \{\bx: F_n(\bx) \le D_n \}$.
  Theorem 2.14.21 of \citet{vdV} gives
  \begin{align*}
     \E\lf[\sup_{f \in \Fcal_n} | (\P_n - P)f\ind_{S_n}|\ri]  
      \lesssim &  \frac{\int_{0}^{M_n c_n} \sqrt{1 + \ln N(\epsilon, \Fcal_n, L_2(P))}d \epsilon}{\sqrt{n}} \\ & + \frac{\int_{0}^{D_n c_n} [1 + \ln N(\epsilon, \Fcal_n, \| \cdot \|_{\infty})] d \epsilon}{n}.
  \end{align*}
  Substituting the covering number bounds \eqref{eq:covering-numbers-2} and the changes of variables $t= \epsilon/M_n c_n$ and $t= \epsilon/D_n c_n$, respectively, gives
  \begin{align*}
    \E\lf[\sup_{f \in \Fcal_n} | (\P_n - P)f\ind_{S_n}|\ri]
     & \lesssim \frac{M_n c_n \sqrt{d_n + \ln K_n} }{\sqrt{n}} + \frac{D_n c_n (d_n + \ln K_n)}{n}.
  \end{align*}
  Now the result follows from Markov's inequality.
\end{proof}

\begin{lemma}
  \label[lemma]{lem:AsN1} 
  Under assumption \ref{A:Asymp}, for all $C  < \infty$, it holds that
  $$
    \sup_{\| \bu \|\le r_n C} \| (\P_n - P) A_n  [\phi(\btheta^* +  \bu) - \phi(\btheta^*)] \| = o_p(1 / \sqrt{n}).
  $$
\end{lemma}
\begin{proof}
  We show that for each row $\ba_n$ from $A_n \in \R^{q \times p_n} $, it holds that
  $$
    \sup_{\| \bu \|\le r_n C} |(\P_n - P) \ba_n^\top [\phi(\btheta^* +  \bu) - \phi(\btheta^*)]| = o_p(1 / \sqrt{n}).
  $$
  Since $A_n [\phi(\btheta^* +  \bu) - \phi(\btheta^*)]$ is a finite dimensional vector, this implies the claim.
  Let $\ba_n$ be some row of $A_n$.
  We have
  $$
    \sup_{\| \bu \|\le r_n C} | (\P_n - P) \ba_n^\top[\phi(\btheta^* +  \bu) - \phi(\btheta^*)] |
    = \sup_{f_{\bu} \in \Fcal_n}  | (\P_n - P) f_{\bu}|
  $$
  with $\Fcal_n = \{\ba_n^\top [\phi(\btheta^* +  \bu) - \phi(\btheta^*) ]] : \| \bu \| \le r_n C \}$ or $\Fcal_n = \{\ba_n^\top [\phi(\btheta^* +  \bu) - \phi(\btheta^*) - \E[\phi(\btheta^* +  \bu) - \phi(\btheta^*)]] : \| \bu \| \le r_n C \}$.
  Both function classes lead to the same $\sup_{f_{\bu} \in \Fcal_n}  | (\P_n - P) f_{\bu}|$, and it depends on the setting which version yields more easily verifiable assumptions in \ref{A:Asymp}.
    Now apply \cref{lem:mixed-entropy} with $d_n = p_n$, $c_n = r_n C$, $K_n = 1$. This gives 
  \begin{align*}
    \sup_{f_{\bu} \in \Fcal_n}  | (\P_n - P) f_{\bu}| = o_p\left(\frac{1}{\sqrt{n} }M_n \,r_n \sqrt{p_n} +  \frac{1}{\sqrt{n} } \frac{D_n \, r_n \,p_n}{ \sqrt{n} }\right) = o_p(1 / \sqrt{n})
  \end{align*}
  since $M_n = o(1/(r_n \sqrt{ p_n}))$ and $D_n = o(\sqrt{n}/(r_n \, p_n))$ by \ref{A:Asymp}.
\end{proof}
 \begin{lemma}
\label[lemma]{lem:sparsity}
Let $\hbtheta = (\hbtheta_{(1)}, \bnull)$ be a solution to $\Phi_n((\hbtheta_{(1)}, \0))_{(1)}  \in \partial p_{\blambda_n}((\hbtheta_{(1)}, \0))_{(1)}$.
Under assumptions \ref{A:lambda2} and \ref{A:Penalty_emp_pr2}, it holds that
	\begin{align*}
& \max_{g \in I_{(2)}} \lambda_{n,g}^{-1} \ \| \Phi_n(\hbtheta)_{G_g}  \|_2  \\
 & \, \le \quad \,   \max_{g \in I_{(2)}} \lambda_{n,g}^{-1} \ \|    \P_n \phi(\btheta^*)_{G_g} \|_2 +   \max_{g \in I_{(2)}} \lambda_{n,g}^{-1} \ \| \bar J(\btheta^*, \hbtheta)_{G_g,(1)} \;\bar J(\btheta^*, \hbtheta)^{-1}_{(1)}\;  \nabla_{\btheta_{(1)}} p_{\blambda_n}(\hbtheta) \|_2 \\
 &\quad \,  +   \max_{g \in I_{(2)}} \lambda_{n,g}^{-1} \ \|  \bar J(\btheta^*, \hbtheta)_{G_g,(1)} \;\bar J(\btheta^*, \hbtheta)^{-1}_{(1)} \; \P_n \phi(\btheta^*)_{(1)} \|_2   + o_p(1),
\end{align*}
	where $\bar J(\btheta^*, \btheta)_{G_g,(1)}  \coloneqq \int_0^1 J(\btheta^* + t(\btheta - \btheta^*))_{G_g,(1)} dt $ and $\bar J(\btheta^*, \btheta)_{(1)} \coloneqq  \int_0^1 J(\btheta^* + t(\btheta - \btheta^*))_{(1)} dt $ with $J(\btheta)_{G_g,(1)}$ and $J(\btheta)_{(1)}$ as defined in \cref{subsec:assump_penalty}.
\end{lemma}
\begin{proof}
  We have $\P_n \phi(\hbtheta)_{G_g} =   \P_n \phi(\btheta^*)_{G_g} +  \P_n \left[ \phi(\hbtheta)_{G_g)} - \phi(\btheta^*)_{G_g}  \right]$ and
  \begin{align*}
   \P_n \left[ \phi(\hbtheta)_{G_g} - \phi(\btheta^*)_{G_g}  \right] 
   &=  P \left[ \phi(\hbtheta)_{G_g} - \phi(\btheta^*)_{G_g}  \right] 
    + (\P_n - P) \left[ \phi(\hbtheta)_{G_g} - \phi(\btheta^*)_{G_g}  \right] \\
   &=  \bar J(\btheta^*, \hbtheta)_{G_g,(1)} (\hbtheta_{(1)} - \btheta^*_{(1)}) 
    +  (\P_n - P) \left[ \phi(\hbtheta)_{G_g} - \phi(\btheta^*)_{G_g}  \right].
  \end{align*}
  Similar to the proof of \cref{theorem2}, one obtains
  \begin{align*}
   \hbtheta_{(1)} - \btheta^*_{(1)} = \bar J(\btheta^*, \hbtheta)^{-1}_{(1)} \left[ -  \P_n \phi(\btheta^*)_{(1)} - (\P_n - P) [   \phi(\hbtheta)_{(1)} - \phi(\btheta^*)_{(1)} ]  + \nabla_{\btheta_{(1)}} p_{\blambda_n}(\hbtheta) \right].
  \end{align*}
  Combining the two displays, we obtain
\begin{align*}
   \Phi_n(\hbtheta)_{G_g}  & =  \P_n \phi(\btheta^*)_{G_g}   +    (\P_n - P) \left[ \phi(\hbtheta)_{G_g} - \phi(\btheta^*)_{G_g}  \right]  \\
  &  \quad +   \bar J(\btheta^*, \hbtheta)_{G_g,(1)} \; \bar J(\btheta^*, \hbtheta)^{-1}_{(1)} \left[ -  \P_n \phi(\btheta^*)_{(1)} + \nabla_{\btheta_{(1)}} p_{\blambda_n}(\hbtheta) \right] \\
  &  \quad +   \bar J(\btheta^*, \hbtheta)_{G_g,(1)} \;\bar J(\btheta^*, \hbtheta)^{-1}_{(1)} \left[- (\P_n - P) [   \phi(\hbtheta)_{(1)} - \phi(\btheta^*)_{(1)} ]  \right].
\end{align*}
The claim follows if we show that 
\begin{align*}
& \max_{g \in I_{(2)}} \lambda_{n,g}^{-1} \ \|  (\P_n - P) \left[ \phi(\hbtheta)_{G_g} - \phi(\btheta^*)_{G_g}  \right]  \|_2 = o_p(1) \quad \text{and} \\
& \max_{g \in I_{(2)}} \lambda_{n,g}^{-1} \ \|  \bar J(\btheta^*, \hbtheta)_{G_g,(1)} \;\bar J(\btheta^*, \hbtheta)^{-1}_{(1)}  \left[- (\P_n - P) [   \phi(\hbtheta)_{(1)} - \phi(\btheta^*)_{(1)} ]  \right] \|_2 = o_p(1) .
\end{align*}
For the first term, it holds that
\[
\max_{g \in I_{(2)}} \lambda_{n,g}^{-1} \ \|  (\P_n - P) \left[ \phi(\hbtheta)_{G_g} - \phi(\btheta^*)_{G_g}  \right]  \|_2 
\le \max_{g \in I_{(2)}} \frac{\sqrt{ |G_g| }}{\lambda_{n,g}}   \ \| (\P_n - P) \left[ \phi(\hbtheta) - \phi(\btheta^*)  \right] \|_\infty
= o_p(1)
\]
by \cref{lem:emp_process_penalty} and the condition on $\lambda_{n,g}$ in \ref{A:lambda2}.
For the second term, we obtain using $\| A \bx \|_2 = (\| A \bx \|_2 / \| \bx \|_{\infty}) \| \bx \|_{\infty} \le  \| \bx \|_{\infty} \sup_{\| \bx \|_\infty \le 1 } \| A \bx \|_2$,
\begin{align*}
  & \max_{g \in I_{(2)}} \lambda_{n,g}^{-1} \ \|  \bar J(\btheta^*, \hbtheta)_{G_g,(1)} \;\bar J(\btheta^*, \hbtheta)^{-1}_{(1)} \left[- (\P_n - P) [   \phi(\hbtheta)_{(1)} - \phi(\btheta^*)_{(1)} ]  \right] \|_2 \\
  & \le \max_{g \in I_{(2)}} \lambda_{n,g}^{-1} \ \sup_{\| \bx \|_\infty \le 1}\|  \bar J(\btheta^*, \hbtheta)_{G_g,(1)} \;\bar J(\btheta^*, \hbtheta)^{-1}_{(1)}  \bx \|_2 \ \| (\P_n - P) [   \phi(\hbtheta)_{(1)} - \phi(\btheta^*)_{(1)} ]  \|_\infty = o_p(1)
\end{align*}
by \cref{lem:emp_process_penalty} and the condition on $\lambda_{n,g}$ in \ref{A:lambda2}.
\end{proof}

\begin{lemma}
\label[lemma]{lem:emp_process_penalty}
Under \ref{A:Penalty_emp_pr2}, for all $C < \infty$, it holds that
\begin{align*}
 \textstyle \sup_{\bu \in \R^{p_n}, \ \bu_{(2)} = \0 ,\ \| \bu \| \le \tilde r_n C}  \left\| (\P_n - P)[ \phi(\btheta^* +  \bu) -  \phi(\btheta^*)] \right\|_{\infty}
 & = o_p(\eta_n).
\end{align*}
\end{lemma}
\begin{proof}
We have
\[
 \sup_{\substack{\bu \in \R^{p_n}, \bu_{(2)} = \0 \\ \| \bu \| \le \tilde r_n C}}  \| (\P_n - P)[ \phi(\btheta^* +  \bu) -  \phi(\btheta^*)] \|_{\infty}= \sup_{f_{\bu, k} \in \Fcal_n}  | (\P_n - P) f_{\bu, k}|
\]
with 
\[ 
\Fcal_n = \{ \tilde \phi(\btheta^* +  (\bu_{(1)}, \bnull))_k - \tilde \phi(\btheta^*)_k : \bu_{(1)} \in \R^{s_n}, \bu_{(2)} = \bnull, \| \bu \| \le \tilde r_n C, k =  1, \ldots p_n \},
\]
where $\tilde \phi(\btheta) = \phi(\btheta)$ or $\tilde \phi(\btheta) = \phi(\btheta) - \E[\phi(\btheta)]$.
Both version lead to the same $ \sup_{f_{\bu, k} \in \Fcal_n}  | (\P_n - P) f_{\bu, k}|$, and it depends on the setting which one leads to more easily verifiable assumptions.
Define 
\begin{align*}
  \tM_n^2 =  \max_{1 \le k \le p_n}  \sup_{\btheta, \btheta' \in \Theta_n'}  \frac{P| \tilde \phi_{i}(\btheta)_k - \tilde \phi_{i}(\btheta')_k |^2}{\| \btheta - \btheta' \|^2}
\end{align*}
By assumption  \ref{A:Penalty_emp_pr2}, we can apply \cref{lem:mixed-entropy} with $d_n  = s_n, K_n = p_n$, $c_n =\tilde  r_n C$, $M_n = \tilde M_n$ and $D_n = \tilde D_n$. 
This gives
\begin{align*}
 \E\left( \sup_{f_{\bu, k} \in \Fcal_n}  | (\P_n - P) f_{\bu, k}| \right) 
 & = O\left( \sqrt{\frac{\tM_n^2 \tilde r_n^2 C^2(s_n + \ln p_n )}{ n}} + \frac{\tD_n \tilde r_n C (s_n +\ln p_n  )}{n} \right) = o(\eta_n), 
\end{align*}
where we used the growth bounds from \ref{A:Penalty_emp_pr2} in the last step.
Now the claim follows from Markov's inequality.
\end{proof}

\begin{lemma} \label[lemma]{lem:eta_n}
  Under assumption \ref{A:phi-moments} and $p_n \to \infty$, it holds that
  \begin{align*}
    \Pr\left(\left\| \frac{1}{n} \sumin \phi_{i}(\btheta^*) \right\|_{\infty} \le 2\sigma_n \sqrt{\frac{\ln p_n}{n}} \right) \to 1.
  \end{align*}
 \end{lemma}
 \begin{proof}
  Let 
  \begin{align*}
    \eta_n = 2\sigma_n \sqrt{\frac{\ln p_n}{n}}, \quad  B_n = \sigma_n \sqrt{\frac{n}{4\ln p_n}} .
  \end{align*}
  Using \cref{lem:truncation}, \ref{A:phi-moments}, and $\sqrt{\max_k P \phi_{}(\btheta^*)_k^2 }\le \sqrt{n} \eta_n$, we get
   \begin{align*}
    \left\| \P_n \phi_{i}(\btheta^*) \right\|_{\infty} \le \left\| (\P_n - P)\phi_{i}(\btheta^*) \ind_{\|\phi_{i}(\btheta^*)\|_\infty \le B_n}\right\|_{\infty} + o_p(\eta_n),  
  \end{align*}
  Further, the union bound and Bernstein's inequality give
  \begin{align*}
    & \quad  \Pr\left(\left\| (\P_n - P) \phi_{i}(\btheta^*) \ind_{\|\phi_{i}(\btheta^*)\|_\infty \le B_n}  \right\|_{\infty} > \eta_n\right)
      \le  2\, p_n \max_{1 \le k \le p_n} \exp\left(-\frac{\frac 1 2  \eta_n^2 }{\frac 1 n \sigma_n^2 + \frac 1 3 \eta_n B_n/n}\right) \\
    & \le  2 \exp\left(\ln p_n -\frac{\eta_n^2 n}{2 \sigma_n^2 + \eta_n B_n}\right)  \le  2 \exp\left(\ln p_n -\frac{\eta_n^2 n}{3 \sigma_n^2 }\right) \\
    & =  2 \exp\left(\ln p_n - \frac 4 3 \ln p_n\right) = o(1). \tag*{\qedhere}
  \end{align*}
 \end{proof}

\end{document}